\documentclass[11pt,reqno,a4paper]{amsart}

\usepackage{etoolbox}
\usepackage{amsmath}
\usepackage{enumerate}
\usepackage{amssymb}
\usepackage{amscd} 
\usepackage{amsthm}
\usepackage{amsfonts}
\usepackage{graphicx}
\usepackage[all,cmtip]{xy}
\usepackage{enumitem}

\patchcmd{\subsection}{-.5em}{.5em}{}{}
\patchcmd{\subsubsection}{-.5em}{.5em}{}{}

\usepackage{enumitem}

\usepackage[T1]{fontenc}

\usepackage[scaled]{helvet} 
\usepackage{courier} 
\usepackage{eulervm}
\normalfont

\makeatother
\usepackage{hyperref}

\numberwithin{equation}{section}

\newcommand{\SL}{\operatorname{SL}}

\newcommand{\cD}{\mathcal{D}}

\newcommand{\cK}{\mathcal{K}}
\newcommand{\cL}{\mathcal{L}}

\newcommand{\cN}{\mathcal{N}}

\newcommand{\cP}{\mathcal{P}}
\newcommand{\cQ}{\mathcal{Q}}

\newcommand{\cX}{\mathcal{X}}
\newcommand{\cY}{\mathcal{Y}}

\newcommand{\bN}{\mathbb{N}}

\newcommand{\bR}{\mathbb{R}}

\newcommand{\bZ}{\mathbb{Z}}

\newcommand{\ra}{\rightarrow}

\newcommand{\qen}{\enskip \textrm{and} \enskip}
\newcommand{\qand}{\quad \textrm{and} \quad}

\newcommand\subsetsim{\mathrel{%
\ooalign{\raise0.2ex\hbox{$\subset$}\cr\hidewidth\raise-0.8ex\hbox{\scalebox{0.9}{$\sim$}}\hidewidth\cr}}}
\newcommand{\eps}{\varepsilon}

\DeclareMathOperator{\supp}{supp}

\DeclareMathOperator{\cum}{Cum}

\theoremstyle{theorem}
\newtheorem{theorem}{Theorem}[section]
\newtheorem{corollary}[theorem]{Corollary}
\newtheorem{proposition}[theorem]{Proposition}
\newtheorem{lemma}[theorem]{Lemma}

\theoremstyle{definition}

\newtheorem{remark}[theorem]{Remark}

\begin{document}

\title{Central Limit Theorems for Diophantine approximants}

\author{Michael Bj\"orklund}
\address{Department of Mathematics, Chalmers, Gothenburg, Sweden}
\email{micbjo@chalmers.se}

\author{Alexander Gorodnik}
\address{University of Bristol, Bristol, UK}
\email{a.gorodnik@bristol.ac.uk}



\date{}


\begin{abstract}
In this paper we study counting functions representing the number of 
solutions of systems of linear inequalities which arise in the theory of 
Diophantine approximation. We develop a method that allows us to 
explain the random-like behavior that these functions exhibit and prove 
a Central Limit Theorem for them. Our approach is based on a quantitative 
study of higher-order correlations for functions defined on the space of 
lattices and a novel technique for estimating cumulants of Siegel transforms. 
\end{abstract}

\maketitle

\addtocontents{toc}{\setcounter{tocdepth}{1}} 
{\small 
\tableofcontents}

\section{Introduction and main results}

\subsection{Motivation}

Many objects which arise in Diophantine Geometry exhibit random-like behavior. For instance, the classical Khinchin theorem in Diophantine approximation can be interpreted as the Borel-Cantelli Property for quasi-independent events, while Schmidt's 
quantitative generalization of Khinchin's Theorem is analogous to the Law of Large Numbers. 
One might ask whether much deeper probabilistic phenomena also take place.
In this  paper, we develop a general framework which allows us to capture certain quasi-independence properties which govern the asymptotic behavior of arithmetic counting functions. We expect that the new methods will have a wide range of applications in Diophantine Geometry; here we apply the techniques to study the distribution of counting functions involving Diophantine approximants.

A basic problem in Diophantine approximation is to find ``good'' rational approximants of vectors $\overline{u} = (u_1,\ldots,u_m) \in \bR^m$. More precisely, given positive numbers $w_1,\ldots,w_m$, which we shall assume sum to one, 
and positive constants $\vartheta_1,\ldots,\vartheta_m$, we consider the system of inequalities
\begin{equation}
\label{eq:diop0}
\big|u_j - \frac{p_j}{q}\big| \leq \frac{\vartheta_j}{q^{1+w_j}}, \quad \textrm{for $j = 1,\ldots,m$},
\end{equation} 
with $(\overline{p},q)\in \mathbb{Z}^m\times \mathbb{N}$. It is well-known that for Lebesgue-almost all $\overline{u} \in\bR^m$,
the system \eqref{eq:diop0} has infinitely many solutions $(\overline{p},q)\in \mathbb{Z}^m\times \mathbb{N}$, so it is natural to 
try to count solutions in bounded regions, which leads us to the counting function
$$
\Delta_T(\overline{u}):=|\{(\overline{p},q)\in \mathbb{Z}^m\times \mathbb{N}:\, 1\le q< T \qen \hbox{ \eqref{eq:diop0} holds}  \}|.
$$
W. Schmidt \cite{sch1} proved that for Lebesgue-almost all $\overline{u}\in [0,1]^m$,
\begin{equation}\label{eq:sch0}
\Delta_T(\overline{u})=C_m\, \log T+O_{\overline{u},\eps}\left((\log T)^{1/2+\eps}\right),\quad \hbox{for all $\eps>0,$}
\end{equation}
where $C_m:=2^m\vartheta_1\cdots \vartheta_m$. One may view this as an analogue of the Law of Large Numbers,
the heuristic for this analogy runs along the following lines. First, note that
$$
\Delta_T(\overline{u})\approx\sum_{s=0}^{\lfloor \log T\rfloor}  \Delta^{(s)}(\overline{u}),  
$$
where 
$$
\Delta^{(s)}(\overline{u}):=|\{(\overline{p},q)\in \mathbb{Z}^m\times \mathbb{N}:\, e^s\le q< e^{s+1} \qen \hbox{ \eqref{eq:diop0} holds}  \}|.
$$
If one could prove that the functions $\Delta^{(s_1)}(\cdot)$ and $\Delta^{(s_2)}(\cdot)$ were ``quasi-independent'' random variables
on $[0,1]^m$, at least when $s_1$, $s_2$, and $|s_1-s_2|$ are sufficiently large, then \eqref{eq:sch0} would follow by some version
of the Law of Large Numbers. Moreover, the same heuristic further suggests that, 
in addition to the Law of Large Numbers, a Central Limit Theorem and perhaps other probabilistic limit laws also hold for $\Delta_T(\cdot)$. 

\medskip

In this paper, we put the above heuristic on firm ground. We do so by 
representing $\Delta^{(s)}(\cdot)$ as a function on
the space of unimodular lattices. It turns out that the ``quasi-independence'' of the family $(\Delta^{(s)})$ that we are trying 
to capture can be translated into 
the dynamical language of higher-order mixing
for a subgroup of linear transformations acting on the space of lattices.

\subsection{Main results}

We are not the first to explore Central Limit Theorems for Diophantine approximants. The one-dimensional case $(m=1)$
has been thoroughly investigated by Leveque \cite{lev1,lev2}, Philipp \cite{ph}, and Fuchs \cite{f}, leading to
the following result proved by Fuchs \cite{f}: there exists an explicit $\sigma > 0$ such that the counting function
$$
\Delta_T(u):=|\{(p,q)\in \bZ\times \bN:\, 1\le q<T,\quad \left|u-p/q\right|< \vartheta\,q^{-2}\}|
$$
satisfies
\begin{equation}  
\label{eq:CLT1}
\left|\left\{u\in [0,1]:\, \frac{\Delta_T(u)-2c\,\log T}{(\log T\cdot\log\log T)^{1/2}}<\xi\right\}\right|
\longrightarrow \hbox{Norm}_\sigma(\xi)
\end{equation}
as $T\to \infty$, where 
$$
\hbox{Norm}_\sigma(\xi):=(2\pi\sigma)^{-1/2} \int_{-\infty}^\xi e^{-s^2/(2\sigma)}\,ds
$$
denotes the normal distribution with the variance $\sigma$. 

\medskip

Central Limit Theorems in higher dimensions when $w_1=\cdots=w_m=1/m$ have recently been studied Dolgopyat, Fayad and Vinogradov \cite{dfv}. In this paper, using very different techniques, we establish the following CLT for \emph{general} exponents 
$w_1,\ldots,w_m$.

\begin{theorem}
	\label{th:CLT_vectors}
Let $m\ge 2$. Then for every $\xi\in \bR$,
\begin{equation}
\label{eq:CLT2}
\left|\left\{\overline{u}\in [0,1]^m:\, \frac{\Delta_T(\overline{u})-C_m\,\log T}{(\log T)^{1/2}}<\xi\right\}\right|
\longrightarrow \hbox{\rm Norm}_{\sigma_m}(\xi)
\end{equation}
as $T\to\infty$,
where
\begin{align*}
\sigma_{m}&:=2 C_m\left(2\zeta(m)\zeta(m+1)^{-1}-1\right),
\end{align*}
and $\zeta$ denotes Riemann's $\zeta$-function.
\end{theorem}

Our proof of Theorem \ref{th:CLT_vectors}, as well as the proof in \cite{dfv}, proceeds by
interpreting $\Delta_T(\cdot)$ as a function on a certain subset $\cY$ of the space of all unimodular 
lattices in $\bR^{m+1}$, and then studies how the sequence $a^s\cY$, where $a$ is a fixed
linear transformation of $\bR^{m+1}$, distributes inside this space. However, the arguments in the two papers 
follow very different routes. The proof in \cite{dfv} contains a novel refinement of the martingale method
(this approach was initiated in this setting by Le Borgne \cite{leb}). Here, one crucially uses the fact that
when $w_1=\cdots=w_m=1/m$, then the set $\cY$ is an unstable manifold for the action of $a$ on the
space of lattices. For general weights, $\cY$ has strictly smaller dimension than the unstable leaves,
and it seems challenging to apply martingale approximation techniques. Instead, our method 
involves a quantitative analysis of higher-order correlations for functions on the space of lattices. We establish an 
asymptotic formula for correlations of arbitrary orders and use this formula to compute limits of all the moments 
of $\Delta_T(\cdot)$ directly. One of the key innovations of our approach 
is an efficient way of estimating sums of cumulants (alternating sums of moments) developed in our recent work 
\cite{BG}. 

\medskip

We also investigate the more general problem of Diophantine approximation for systems of linear forms. The space 
$\hbox{M}_{m,n}(\bR)$ of $m$ linear forms in $n$ real variables is parametrized by real $m\times n$ matrices.
Given $u\in\hbox{M}_{m,n}(\bR)$, we consider the family $(L_u^{(i)})$ of linear forms defined by
$$
L_u^{(i)}(x_1,\ldots,x_n)=\sum_{j=1}^n u_{ij} x_j,\quad i=1,\ldots,m.
$$
Let $\|\cdot\|$ be a norm on $\bR^n$. Fix $\vartheta_1,\ldots,\vartheta_m > 0$ and $w_1,\ldots,w_m>0$ which satisfy
$$
w_1+\cdots+w_m=n,
$$
and consider the system of Diophantine inequalities
\begin{equation}
\label{eq:diop2_0}
\left|p_i+L_u^{(i)}(q_1,\ldots,q_n)\right|< \vartheta_i\, \|\overline{q}\|^{-w_i},\quad i=1,\ldots,m,
\end{equation}
with $(\overline{p},\overline{q})=(p_1,\ldots,p_m,q_1,\ldots,q_n)\in \mathbb{Z}^m\times (\mathbb{Z}^n\backslash \{0\})$.
The number of solutions of this system with the norm of the ``denominator'' $\overline{q}$ bounded by $T$ is given by
\begin{equation}
\label{eq:N_T}
\Delta_T({u}):=\left|\left\{(\overline{p},\overline{q})\in \mathbb{Z}^m\times \mathbb{Z}^n:\, 0< \|\overline{q}\|< T \qen \hbox{ \eqref{eq:diop2_0} holds}  \right\}\right|.
\end{equation}
Our main result in this paper is the following generalization of Theorem \ref{th:CLT_vectors}.

\begin{theorem}
	\label{th:CLT_forms}
	If $m\ge 2$, then for every $\xi\in \bR$,
	\begin{equation}
	\label{eq:CLT3}
	\left|\left\{{u}\in \hbox{\rm M}_{m,n}([0,1]):\, \frac{\Delta_T({u})-C_{m,n}\,\log T}{(\log T)^{1/2}}<\xi\right\}\right|
	\longrightarrow \hbox{\rm Norm}_{\sigma_{m,n}}(\xi)
	\end{equation}
	as $T\to\infty$, where
\[
C_{m,n} :=C_m\omega_n\quad\hbox{with $\omega_n:=\int_{S^{n-1}} \|\overline{z}\|^{-n}\,d\overline{z}$}
\]
and
\[
\sigma_{m,n}:=2C_{m,n}\left(2\zeta(m+n-1)\zeta(m+n)^{-1}-1\right).
\]
\end{theorem}

The special case $w_1 =\ldots=w_m=n/m$ was proved earlier in \cite{dfv}.

\subsection{An outline of the proof of Theorem \ref{th:CLT_forms}}
\label{sec:outline}

We begin by observing that $\Delta_T(\cdot)$ can be interpreted 
as a function on the space of lattices in $\bR^{m+n}$.
Given $u\in \hbox{M}_{m,n}([0,1])$, we define the unimodular lattice $\Lambda_u$ in $\bR^{m+n}$ by
\begin{equation}
\label{eq:lll}
\Lambda_{u}:=\left\{\left(p_1+\sum_{j=1}^n u_{1j} q_j,\ldots, p_m+\sum_{j=1}^n u_{mj} q_j, \overline{q}\right):\, (\overline{p},\overline{q})\in \mathbb{Z}^m\times \mathbb{Z}^n\right\},
\end{equation}
and we see that
$$
\Delta_T(u)=|\Lambda_u\cap \Omega_T|+O(1), \quad \textrm{for $T > 0$},
$$
where $\Omega_T$ denotes the domain
\begin{equation}
\label{eq:omega_t}
\Omega_T:=\left\{(\overline{x},\overline{y})\in \bR^{m+n}:\, 1\le \|\overline{y}\|<T,\, |x_i|<\vartheta_i\,\|\overline{y}\|^{-w_i}, \, i=1,\ldots, m\right\}.
\end{equation}
The space $\cX$ of unimodular lattices in $\bR^{m+n}$  is naturally a homogeneous space of
the group  $\hbox{SL}_{m+n}(\bR)$
equipped with the invariant probability measure $\mu_\cX$. The set
$$
\cY:=\{\Lambda_u:\, u\in \hbox{M}_{m,n}([0,1])\}
$$
is a $mn$-dimensional torus embedded in $\cX$, and we equip $\cY$ with the Haar probability measure $\mu_\cY$, 
interpreted as a Borel measure on $\cX$. 

\medskip

We further observe (see Section \ref{sec:CLT_counting} for more details) that each domain $\Omega_T$ can be tessellated 
using a fixed diagonal matrix $a$ in $\hbox{SL}_{m+n}(\bR)$, so that for a suitable function $\hat\chi:\cX\to\bR$, we have 
$$
|\Lambda\cap \Omega_T|\approx\sum_{s=0}^{N-1} \hat\chi(a^s\Lambda)\quad \hbox{for $\Lambda\in\cX$.}
$$
Hence we are left with analyzing the distribution of values for the sums
$\sum_{s=0}^N \hat \chi(a^sy)$ with $y\in \cY$.
This will allow us to apply techniques developed in our previous work \cite{BG}, as well as in \cite{BEG} (joint with M. Einsiedler). 
Intuitively, our arguments will be guided by the hope that 
the observables $\hat \chi\circ a^s$ are ``quasi-independent'' 
with respect to $\mu_{\cY}$. Due to the discontinuity and unboundedness of the function $\hat\chi$ on $\cX$, 
it gets quite technical to formulate this quasi-independence directly. Instead, we shall argue in steps. 

We begin in Section \ref{sec:correlations} by establishing quasi-independence for observables of 
the form $\phi \circ a^s$, where $\phi$ is a smooth and compactly supported function on $\cX$.
This amounts to an asymptotic formula (Corollary \ref{cor:corr0}) for the higher-order correlations
\begin{equation}
\label{eq:cor000}
\int_{\cY} \phi_1(a^{s_1}y)\cdots \phi_r(a^{s_r}y)\, d\mu_\cY(y) \quad \hbox{with $\phi_1,\ldots,\phi_r\in C_c^\infty(\cX).$}
\end{equation}
It will be crucial for our arguments later that the error term in this formula is explicit in terms of the exponents $s_1,\ldots,s_r$ 
and in (certain norms of) the functions $\phi_1,\ldots,\phi_r$. In Section \ref{sec:CLT_compact},
we use these estimates to prove the Central Limit Theorems for sums of the form
$$
F_N(y) :=\sum_{s=0}^{N-1} \left(\phi(a^sy)-\mu_\cY(\phi)\right)\quad \hbox{with $\phi\in C_c^\infty(\cX)$ and $y\in \cY$}.
$$
To do this, we use an adaption of the classical Cumulant Method (see Proposition \ref{th:CLT}), 
which provides bounds on cumulants (alternating sums of moments) given estimates on expressions as in \eqref{eq:cor000}, at least in certain ranges of the parameters $(s_1,\ldots,s_r)$. Here we shall exploit the decomposition \eqref{eq:decomp} into
``seperated''/``clustered'' tuples. 
We stress that the cumulant $\cum^{(r)}(F_N)$ of order $r$ can be expressed as a sum of $O(N^r)$
terms, normalized by $N^{r/2}$, so that in order to prove that it vanishes asymptotically, we require more than just square-root cancellation; however, the error term in the asymptotic formula for \eqref{eq:cor000} is rather weak. Nonetheless, by using intricate combinatorial 
cancellations of cumulants, we can establish the required bounds. 

In order to extend the method in Section \ref{sec:CLT_compact} 
to the kind of \emph{unbounded} functions which arise in our subsequent approximation 
arguments we have to investigate possible escapes of mass for the sequence of tori $a^s\cY$
inside the space $\cX$. In Section \ref{sec:non_div}, we prove several results 
in this direction (see e.g. Proposition \ref{p:div1}), as well as $L^p$-bounds (see Propositions \ref{prop:sup} and \ref{prop:sup2}).
We stress that the general non-divergence estimates for unipotent flows developed by Kleinbock-Margulis \cite{km1} are not 
sufficient for our purposes, and in particular, the exact value of the exponent in Proposition \ref{p:div1} will be crucial for our argument.
The proof of the $L^2$-norm bound in Proposition \ref{prop:sup2} is especially interesting in this regard since it uncovers 
that the escape of mass is related to delicate arithmetic questions; our arguments require careful estimates on the number of solutions of certain Diophantine equations.

To make the technical passages in the final steps of the proof of Theorem \ref{th:CLT_forms} a bit more readable, we shall
devote Section \ref{sec:CLT_siegel_smooth} to Central Limit Theorems for sums of the form
$$
\sum_{s=0}^{N-1} \hat f(a^sy)\quad\hbox{for $y\in \cY$},
$$
where $f$ is a smooth and compactly supported function on $\bR^{m+n}$, and $\hat{f}$ denotes the Siegel transform of $f$ (see Section \ref{sec:sieg_tran} for definitions). We stress that even though $f$ is assumed to be bounded, $\hat{f}$ is \emph{unbounded} on $\cX$. To prove
the Central Limit Theorems in this setting, we approximate $\hat f$ by compactly supported functions on $\cX$ and then use the 
estimates from Section \ref{sec:CLT_compact}. However, the bounds in these estimates crucially depend on the order of 
approximation, so this step requires a delicate analysis of the error terms. The non-divergence results established in Section \ref{sec:non_div} play important role here.

Finally, to prove the Central Limit Theorem for the function $\hat{\chi}$ (which is the Siegel transform of an indicator function on a 
nice bounded domain in $\bR^{m+n}$), and thus establish Theorem \ref{th:CLT_forms}, we need to approximate $\chi$ with smooth functions, and show that the arguments in Section \ref{sec:CLT_siegel_smooth} can be adapted to certain \emph{sequences} of Siegel transforms of smooth and compactly supported functions. This will be done in Section \ref{sec:CLT_counting}.

\section{Estimates on higher-order correlations}\label{sec:correlations}
\label{sec:esthoc}

Let $\mathcal{X}$ denote the space of unimodular lattices in $\bR^{m+n}$.
Setting 
$$
G:=\hbox{SL}_{m+n}(\bR)\quad\hbox{ and }\quad \Gamma:=\hbox{SL}_{m+n}(\bZ),
$$
we may consider the space $\cX$ as a homogeneous space under the linear action of the group $G$,
so that 
$$
\cX\simeq G/\Gamma.
$$
Let $\mu_{\cX}$ denote the $G$-invariant probability measure on $\cX$. 

\medskip

We fix $m,n\ge 1$  and denote by $U$ the subgroup
\begin{equation}
\label{eq:u}
U:=\left\{
\left(
\begin{tabular}{ll}
$I_m$ & $u$\\
$0$ & $I_n$
\end{tabular}
\right):\, u\in \hbox{M}_{m,n}(\mathbb{R})
\right\} < G,
\end{equation}
and set $\mathcal{Y}:=U\mathbb{Z}^{m+n}\subset \mathcal{X}$. Geometrically, $\cY$ can be visualized as a $m n$-dimensional torus embedded in the spaces of lattices $\cX$. We denote by $\mu_\cY$ the probability measure on $\mathcal{Y}$ induced by the 
Lebesgue probability measure on $\hbox{M}_{m,n}([0,1])$, and we note that $\cY$ corresponds to the collection of unimodular 
lattices $\Lambda_u$, for $u\in \hbox{M}_{m,n}([0,1))$, introduced earlier in \eqref{eq:lll}. 

\medskip

Let us further fix positive numbers $w_1,\ldots,w_{m+n}$ satisfying 
$$
\sum_{i=1}^{m}w_i=\sum_{i=m+1}^{m+n}w_i,
$$
and denote by $(a_t)$ the one-parameter semi-group
\begin{equation}
\label{eq:at}
a_t:=\hbox{diag}\left(e^{w_1 t},\ldots, e^{w_{m}t},e^{-w_{m+1} t},\ldots, e^{-w_{m+n}t} \right),\quad t>0.
\end{equation}
The aim of this section is to analyze the asymptotic behavior of $a_t\mathcal{Y} \subset \mathcal{X}$ as $t\to\infty$, and 
investigate ``decoupling'' of correlations of the form
\begin{equation}
\label{eq:corr_Y}
\int_{\cY}\phi_1(a_{t_1}y)\ldots\phi_r(a_{t_r}y)\, d\mu_\cY(y)\quad \hbox{ for $\phi_1,\ldots,\phi_r\in C_c^\infty(\cX)$},
\end{equation}
for ``large'' $t_1,\ldots,t_r > 0$. It will be essential for our subsequent argument that the error terms in this
``decoupling'' are explicit in terms of the parameters $t_1,\ldots,t_r>0$ and suitable norms of the functions $\phi_1,\ldots, \phi_r$,
which we now introduce. 

\medskip

Every $Y \in \hbox{Lie}(G)$ defines a first order differential operator 
$\cD_Y$ on $C_c^\infty(\cX)$ by
$$
\cD_Y(\phi)(x):=\frac{d}{dt}\phi(\exp(tY)x)|_{t=0}.
$$
If we fix an (ordered) basis $\{Y_1,\ldots,Y_r\}$ of $\hbox{Lie}(G)$, then every monomial $Z=Y_1^{\ell_1}\cdots Y_r^{\ell_r}$
defines a differential operator by
\begin{equation}
\label{eq:diff}
\cD_Z:=\cD_{Y_1}^{\ell_1}\cdots \cD_{Y_r}^{\ell_r},
\end{equation}
of \emph{degree} $\deg(Z) = \ell_1+\cdots +\ell_r$. For $k\ge 1$ and $\phi\in C_c^\infty(\cX)$,  we define the norm
\begin{equation}
\label{eq:sob_def0}
\|\phi\|_{L^2_k(\cX)}:=\left( \sum_{\deg(Z)\le k} \int_{\cX} | (\cD_Z\phi)(x)|^2\, d\mu_\cX(x)\right)^{1/2}.
\end{equation}

The starting point of our discussion is a well-known quantitative estimate on correlations of smooth functions on $\cX$:

\begin{theorem}\label{th:cor0}
	There exist $\gamma>0$ and $k\ge 1$ such that for all $\phi_1,\phi_2\in C_c^\infty(\cX)$ and $g\in G$,
	$$
	\int_{\cX} \phi_1(g x)\phi_2(x)\, d\mu_\cX(x)=
	\left(\int_{\mathcal{X}}\phi_1 \, d\mu_\cX\right)
	\left(\int_{\mathcal{X}}\phi_2 \, d\mu_\cX\right)+O\left(\|g\|^{-\gamma}\, \|\phi_1\|_{L^2_k(\cX)}\|\phi_1\|_{L^2_k(\cX)}\right).
	$$
\end{theorem}
This theorem has a very long history that we will not attempt to survey here, but only mention that a result of this form 
can be found, for instance, in \cite{km0,km2}. 

\smallskip

\begin{center}
\emph{Let us from now on we fix $k\ge 1$ so that Theorem \ref{th:cor0} holds.} 
\end{center}

\smallskip

Our goal is to decouple the higher-order correlations in \eqref{eq:corr_Y}, but in order to state our results we first need to introduce a
family of finer norms on $C^\infty_c(\cX)$ than $(\|\cdot\|_{L^2(\cX)_k})$. Let us denote by $\|\cdot\|_{C^0}$ the uniform norm 
on $C_c(\cX)$. If we fix a right-invariant Riemannian metric on $G$, then it induces a metric $d$ on $\cX\simeq G/\Gamma$, 
which allows us to define the norms
$$
\|\phi\|_{Lip}:=\sup\left\{\frac{|\phi(x_1)-\phi(x_2)|}{d(x_1,x_2)}:\, x_1,x_2\in\cX, x_1\ne x_2  \right\},
$$
and
\begin{equation}
\label{eq:n_k}
\cN_k(\phi):=\max\left\{\|\phi\|_{C^0}, \|\phi\|_{Lip}, \|\phi\|_{L^2_k(\cX)}\right\},
\end{equation}
for $\phi \in C_c^\infty(\cX)$. We shall prove:

\begin{theorem}\label{th:corr}
There exists $\delta>0$ such that for every compact $\Omega\subset U$,  $f\in C_c^\infty(U)$ with $\hbox{\rm supp}(f)\subset \Omega$, $\phi_1,\ldots,\phi_r\in C_c^\infty(\mathcal{X})$,
$x_0\in \mathcal{X}$, and $t_1,\ldots,t_r>0$, we have 
\begin{align*}
\int_{U} f(u)\left(\prod_{i=1}^r \phi_i(a_{t_i}ux_0)\right)\,du=& \left(\int_{U}f(u)\,du\right) \prod_{i=1}^r\left(\int_{\mathcal{X}}\phi_i \, d\mu_\cX\right)\\
& +O_{x_0,\Omega,r}\left( e^{-\delta D(t_1,\ldots,t_r)}\, \|f\|_{C^k}\prod_{i=1}^r \cN_k(\phi_i)\right),
\end{align*}
where 
$$
D(t_1,\ldots,t_r):=\min\{t_i, |t_i-t_j|:\, 1\le i\ne j\le r\}.
$$
\end{theorem}

\begin{remark}
The case $r=1$ was proved by Kleinbock and Margulis in \cite{km3}, and our arguments are inspired by theirs. We stress
that the constant $\delta$ in Theorem \ref{th:corr} is \emph{independent} of $r$.
\end{remark}

We also record the following corollary of Theorem \ref{th:corr}.

\begin{corollary}\label{cor:corr0}
There exists $\delta' >0$ such that for every 
$\phi_0\in C^\infty(\cY)$,
$\phi_1,\ldots,\phi_r\in C_c^\infty(\mathcal{X})$,
$x_0\in \mathcal{X}$, and $t_1,\ldots,t_r>0$, we have 
\begin{align*}
\int_{\mathcal{Y}} \phi_0(y) \left( \prod_{i=1}^r \phi_i(a_{t_i}y) \right )d\mu_\cY(y)=&  \left(\int_{\mathcal{Y}}\phi_0 \, d\mu_\cY \right)\prod_{i=1}^r \left(\int_{\mathcal{X}}\phi_i \, d\mu_\cX\right)\\
& +O_{r}\left( e^{-\delta' D(t_1,\ldots,t_r)}\, \|\phi_0\|_{C^k}\prod_{i=1}^r \cN_k(\phi_i)\right).
\end{align*}
\end{corollary}

\begin{proof}[Proof of Corollary \ref{cor:corr0} (assuming Theorem \ref{th:corr})]
	Let $x_0$ denote the identity coset in $\cX \cong G/\Gamma$, which corresponds to the standard lattice $\bZ^{m+n}$, and 
	recall that 
	$$
	\cY=Ux_0\simeq U/(U\cap \Gamma).
	$$
	Let $\tilde{\phi}_0\in C^\infty(U)$ denote the lift of the function $\phi_0$ to $U$, and $\chi$ the characteristic function 
	of the subset
	$$
	U_0:=\left\{
	\left(
	\begin{tabular}{ll}
	$I_m$ & $u$\\
	$0$ & $I_n$
	\end{tabular}
	\right):\, u\in \hbox{M}_{m,n}([0,1])
	\right\}.
	$$
	Given $\eps > 0$, let $\chi_\eps\in C_c^\infty(U)$ be a smooth approximation of $\chi$
	with uniformly bounded support which satisfy
	$$
	\chi\le \chi_\eps\le 1,\quad \|\chi-\chi_\eps\|_{L^1(U)}\ll \eps,\quad \|\chi_\eps\|_{C^k}\ll \eps^{-k}.
	$$
	We observe that if $f_\eps:=\tilde{\phi}_0\chi_\eps$ and $f_0:=\tilde{\phi}_0\chi$, then
\[
	\|f_0-f_\eps\|_{L^1(U)}\ll \eps\, \|\phi_0\|_{C^0}
	\]
	and
	\[
	\|f_\eps \|_{C^k}\ll \|\tilde{\phi}_0\|_{C^k}\|\chi_\eps\|_{C^k}\ll \eps^{-k} \|{\phi}_0\|_{C^k},
\]
	which implies that 
	\begin{align*}
	\int_{\mathcal{Y}} \phi_0(y) \left( \prod_{i=1}^r \phi_i(a_{t_i}y) \right )d\mu_\cY(y) &=
	\int_{U} f_0(u) \left( \prod_{i=1}^r \phi_i(a_{t_i}ux_0) \right )\, du\\
	&=
	\int_{U} f_\eps(u) \left( \prod_{i=1}^r \phi_i(a_{t_i}u x_0) \right )du
	+O\left(\eps \prod_{i=0}^r \|\phi_i\|_{C^0}\right),
	\end{align*}
	and 
	\begin{align*}
	\int_{\cY} \phi_0\, d\mu_\cY=\int_U f_0(u)\, du=
	\int_{U} f_\eps(u)\, du +O\left(\eps\|\phi_0\|_{C^0}\right).
	\end{align*}
	Therefore, Theorem \ref{th:corr} implies that
	\begin{align*}
	\int_{\mathcal{Y}} \phi_0(y) \left( \prod_{i=1}^r \phi_i(a_{t_i}y) \right )d\mu_\cY(y)=&  \left(\int_{U}f_\eps(u) \, du \right)\prod_{i=1}^r \left(\int_{\mathcal{X}}\phi_i \, d\mu_\cX\right)\\
	& +O_{r}\left( \eps \prod_{i=0}^r \|\phi_i\|_{C^0}+ e^{-\delta D(t_1,\ldots,t_r)}\, \|f_\eps\|_{C^k}\prod_{i=1}^r \cN_k(\phi_i)\right)\\
	=&  \left(\int_{\mathcal{Y}}\phi_0 \, d\mu_\cY \right)\prod_{i=1}^r \left(\int_{\mathcal{X}}\phi_i \, d\mu_\cX\right)\\
	& +O_{r}\left( \left(\eps+ \eps^{-k} e^{-\delta D(t_1,\ldots,t_r)}\right)\, \|\phi_0\|_{C^k}\prod_{i=1}^r \cN_k(\phi_i)\right).
	\end{align*}
	The corollary (with $\delta' = \delta/(k+1)$) follows by choosing $\eps=e^{-\delta D(t_1,\ldots,t_r)/(k+1)}$.
\end{proof}

\subsection{Preliminary results}
We recall that $d$ is a distance on $\cX \cong G/\Gamma$ induced from a right-invariant Riemannian metric on $G$. 
We denote by $B_G(\rho)$ the ball of radius $\rho$ centered at the identity in $G$. For a point $x\in\cX$, we let $\iota(x)$ denote
the \emph{injectivity radius} at $x$, that is to say, the supremum over $\rho>0$ such that the map 
$B_G(\rho)\to B_G(\rho)x:g\mapsto gx$ is injective. \\

Given $\eps > 0$, let 
\begin{equation}
\label{eq:Keps}
\cK_\eps = \big\{ \Lambda \in \cX \, : \, \|v\| \geq \eps, \enskip \textrm{for all $v \in \Lambda \setminus \{0\}$} \big\}.
\end{equation}
By Mahler's Compactness Criterion, $\cK_\eps$ is a compact subset of $\cX$. Furthermore, using Reduction Theory,
one can show:

\begin{proposition}[\cite{km3}, Prop.~3.5] \label{p:iota}
 $\iota(x)\gg \eps^{m+n}$ for any $x\in \cK_\eps$.
\end{proposition}

An important role in our argument will be played by the one-parameter semi-group
\begin{equation}
\label{eq:bt}
b_t:=\hbox{diag}\left(e^{t/m},\ldots,e^{t/m},e^{-t/n},\ldots, e^{-t/n}\right),\quad t>0,
\end{equation}
which coincides with the semi-group $(a_t)$ as defined in \eqref{eq:at} with the special choice of exponents 
\[
w_1 = \ldots = w_m = \frac{1}{m} \qand w_{m+1} = \ldots = w_{m+n} = \frac{1}{n}.
\]
The submanifold $\cY \subset \cX$ is an unstable manifold for the flow $(b_t)$ which makes the analysis of the
asymptotic behavior of $b_t\cY$ significantly easier than that of $a_t \cY$ for general parameters. Using Theorem
\ref{th:cor0}, Kleinbock and Margulis proved in \cite{km0} a quantitative equidistribution result for the family $b_t\cY$ 
as $t\to \infty$, we shall use a version of this result from their later work \cite{km3}.

\begin{theorem}[\cite{km3}; Th.~2.3]\label{p:km}
There exist $\rho_0>0$ and $c,\gamma>0$  such that for every $\rho\in (0,\rho_0 )$, $f\in C_c^\infty(U)$
satisfying $\hbox{\rm supp}(f)\subset B_G(\rho)$, $x\in \mathcal{X}$ with $\iota(x)>2\rho$,
$\phi\in C_c^\infty(\cX)$, and $t\ge 0$,
\begin{align*}
\int_{U} f(u) \phi(b_tux)\, du=&\left(\int_{U} f(u)\, du \right)
\left(\int_{\cX} \phi\, d\mu_\cX \right)\\
&+O\left( \rho\|f\|_{L^1(U)} \|\phi\|_{Lip} + \rho^{-c} e^{-\gamma t}\, \|f\|_{C^k}\|\phi\|_{L^2_k(\cX)}\right).
\end{align*}
\end{theorem}
\begin{remark}
Although the dependence on $\phi$ is not stated in \cite[Theorem~2.3]{km3},
the estimate is explicit in the proof.
\end{remark}

We will prove Theorem \ref{th:corr} through successive uses of Theorem \ref{p:km}. In order to 
make things more transparent, it will be convenient to embed the flow $(a_t)$ as defined in \eqref{eq:at} 
in a multi-parameter flow as follows.  For $\overline{s}=(s_1,\ldots,s_{m+n})\in \mathbb{R}^{m+n}$,
we set
\begin{equation}
\label{eq:as}
a(\overline{s}):=\hbox{diag}\left(e^{s_1},\ldots,e^{s_m},e^{-s_{m+1}},\ldots, e^{-s_{m+n}}\right).
\end{equation}
We denote by $S^+$ the cone in $\mathbb{R}^{m+n}$ consisting of those $\overline{s}=(s_1,\ldots,s_{m+n})$ which satisfy
$$
s_1,\ldots,s_{m+n}>0\quad\hbox{and}\quad 
\sum_{i=1}^{m}s_i=\sum_{i=m+1}^{m+n}s_i.
$$
For $\overline{s}=(s_1,\ldots,s_{m+n})\in S^+$, we set 
$$
\lfloor\overline{s}\rfloor :=\min(s_1,\ldots,s_{m+n}),
$$
and thus, with $\overline{s}_t:=(w_1t,\ldots,w_{m+n}t)$, we see that $a_t=a(\overline{s}_t)$.\\

In addition to Theorem \ref{p:km}, we shall also need the following quantitative non-divergence
estimate for unipotent flows established by Kleinbock and Margulis in \cite{km3}.

\begin{theorem}[\cite{km3}; Cor.~3.4]\label{p:non_div}
There exists $\theta=\theta(m,n)>0$ such that
for every compact $L\subset \cX$ and a Euclidean ball $B\subset U$ centered at the identity,
there exists $T_0>0$ such that for every $\eps\in (0,1)$, $x\in L$,
and $\overline{s}\in S^+$ satisfying $\lfloor \overline{s}\rfloor\ge T_0$,
one has
$$
\left|\{u\in B: a(\overline{s})ux\notin \cK_\eps \}\right|\ll \eps^\theta\, |B|.
$$
\end{theorem}

\subsection{Proof of Theorem \ref{th:corr}}

Let us fix $r \geq 1$ and a $r$-tuple $(t_1,\ldots,t_r)$. Upon re-labeling, we may assume that
$t_1\le\ldots\le t_{r}$, so that 
\begin{equation}
\label{eq:D}
D := D(t_1,\ldots,t_r)=\min\{t_1,t_2-t_1,\ldots, t_r-t_{r-1}\}.
\end{equation}
The next lemma provides an additional parameter $\overline{s} \in S^{+}$,
which depends on the $r$-tuple $(t_1,\ldots,t_r)$. This parameter will be used 
throughout the proof of Theorem \ref{th:corr}, and we stress that the accompanying 
constants $c_1,c_2$ and $c_3$ are independent of $r$ and the $r$-tuple $(t_1,\ldots,t_r)$.

\begin{lemma}\label{l:s}
	There exist $c_1,c_2,c_3>0$ such that given any $t_r>t_{r-1}>0$,
	there exists $\overline{s}\in S^+$ satisfying:
	\begin{enumerate}
		\item[(i)] $\lfloor  \overline{s} \rfloor \ge c_1 (t_r -t_{r-1})$,
		\item[(ii)] $\lfloor  \overline{s}-\overline{s}_{{t}_{r-1}} \rfloor \ge c_2 (t_r-t_{r-1})$,
		\item[(iii)] $\overline{s}_{t_r}-\overline{s}=(\frac{z}{m},\ldots,\frac{z}{m},\frac{z}{n},\ldots,\frac{z}{n})$ for some $z \ge c_3 \min(t_{r-1},t_r-t_{r-1})$.
	\end{enumerate}
\end{lemma}

\begin{proof}
	We start the proof by defining $\overline{s}$ by the formula in (iii),
	where the parameter $z$ will be chosen later, that is to say, we set
	$$
	\overline{s}=\left(w_1t_r-\frac{z}{m},\ldots,w_mt_r-\frac{z}{m},
	w_{m+1}t_r-\frac{z}{n},\ldots,w_{m+n}t_r-\frac{z}{n}\right).
	$$
	Then (i) holds provided that 
	$$
	A_1 t_r- \frac{z}{m}\ge c_1(t_r-t_{r-1})\quad \hbox{and}\quad
	A_2 t_r- \frac{z}{n}\ge c_1(t_r-t_{r-1}),
	$$
	where $A_1:=\min(w_i:1\le i\le m)$ and $A_2:=\min(w_i:m+1\le i\le m+n)$,
	so if we set  $c_1= \min(A_1,A_2)$, then (i) holds when
	\begin{equation}
	\label{eq:z1}
	z\le c_1\min(m,n) t_{r-1}.
	\end{equation}
	To arrange (ii), we observe that
	\begin{align*}
	\overline{s}-\overline{s}_{{t}_{r-1}}=
	&\Big(w_1(t_r-t_{r-1})-\frac{z}{m},\ldots,w_m(t_r-t_{r-1})-\frac{z}{m},\\
	& \;\; w_{m+1}(t_r-t_{r-1})-\frac{z}{n},\ldots,w_{m+n}(t_r-t_{r-1})-\frac{z}{n}
	\Big),
	\end{align*}
	and thus (ii) holds provided that
	$$
	A_1 (t_r-t_{r-1})- \frac{z}{m}\ge c_2(t_r-t_{r-1})\quad \hbox{and}\quad
	A_2 (t_r-t_{r-1})- \frac{z}{n}\ge c_2(t_r-t_{r-1}).
	$$
	If we let $c_2= \min(A_1,A_2)/2$, then (ii) holds when
	\begin{equation}
	\label{eq:z2}
	z\le c_2\min(m,n) (t_r-t_{r-1}).
	\end{equation}
	So far we have arranged so that (i) and (ii) hold provided that $z$ satisfies \eqref{eq:z1}
	and \eqref{eq:z2}. Let $c_3=\min(c_1,c_2)\min(m,n)$, and note that if we pick $z = c_3 \min(t_{r-1},t_{r}-t_{r-1})$,
	then (i), (ii) and (iii) are all satisfied.
\end{proof}

Let us now continue with the proof of Theorem \ref{th:corr}. With the parameter $\overline{s}$ provided by Lemma \ref{l:s} above,
we have
\begin{align}
\lfloor \overline{s} \rfloor &\ge c_1\, D, \label{eq:c1}\\
\lfloor \overline{s}- \overline{s}_{t_{i}}  \rfloor &\ge c_2\,D\quad\hbox{for all $i=1,\ldots,r-1$}, \label{eq:c2}\\
a(\overline{s}_{t_{r}} - \overline{s})&=b_z\quad \hbox{for some $z\ge c_3\, D$,}\label{eq:c3}
\end{align}
where $b_z$ is defined as in \eqref{eq:bt} and $D$ as in \eqref{eq:D} Let us throughout the rest of the proof fix a 
compact set $\Omega \subset U$. Our aim now is to estimate integrals of the form
$$
I_r:=\int_{U} f(u)\left(\prod_{i=1}^r \phi_i(a_{t_i}ux_0)\right)\, du,
$$
where $f$ ranges over $C^\infty_c(U)$ with $\supp(f) \subset \Omega$. Our proof will proceed by induction over $r$,
the case $r=0$ being trivial. \\

Before we can start the induction, we need some notation. Let $\rho_0$ and $k$ be as in Theorem \ref{p:km}, and pick for
$0 < \rho < \rho_0$, a non-negative $\omega_\rho \in C^\infty_c(\cX)$ such that
\begin{equation}
\label{eq:omegarho}
\hbox{supp}(\omega_\rho)\subset B_{G}(\rho),\quad \|\omega_\rho\|_{C^k}\ll \rho^{-\sigma},\quad \int_{U} \omega_\rho(v)\, dv=1,
\end{equation}
for some fixed $\sigma=\sigma(m,n,k)>0$. The integral $I_r$ can now be rewritten as follows: 
\begin{align*}
I_r=&I_r\cdot \left(\int_{U} \omega_\rho(v)\,dv\right) = \int_{U\times U} f(u)\omega_\rho(v)\left(\prod_{i=1}^r\phi_i(a_{t_i}ux_0)\right)\,dudv \\
=&
\int_{U\times U} f(a(-\overline{s})va(\overline{s})u)\omega_\rho(v)\left(\prod_{i=1}\phi_i(a_{t_i}a(-\overline{s})va(\overline{s})ux_0)\right)\,dudv.
\end{align*}
If we set
$$
f_{\overline{s},u}(v):=f(a(-\overline{s})va(\overline{s})u)\omega_\rho(v)
\quad \hbox{and} \quad
x_{\overline{s},u}:= a(\overline{s})ux_0,
$$
then 
\begin{align*}
I_r=& \int_{U} \left( \int_{U} f_{\overline{s},u}(v) \prod_{i=1}^r \phi_i(a(\overline{s}_{{t}_i}-\overline{s})v x_{\overline{s},u})\, dv  \right)\, du.
\end{align*}
We observe that if $f_{\overline{s},u}(v)\ne 0$, then 
$$
v\in \hbox{supp}(\omega_\rho)\subset B_{G}(\rho) \qand a(-\overline{s})va(\overline{s})u\in \hbox{supp}(f),
$$
so that 
$$
u\in a(-\overline{s})v^{-1}a(\overline{s})\hbox{supp}(f)\subset 
a(-\overline{s})\hbox{supp}(\omega_\rho)^{-1}a(\overline{s})\Omega.
$$
Since $\overline{s}\in S^+$, the linear map
$v\mapsto a(-\overline{s})v a(\overline{s})$,
for $v\in U \cong \bR^{mn}$, is \emph{non-expanding}, and
thus we can conclude that there exists
a fixed Euclidean ball $B$ in $U$ 
(depending only on $\Omega$ and $\rho_0$), and centered at the identity,
such that if $f_{\overline{s},u}(v)\ne 0$, then $u\in B$.
This implies that the integral $I_r$ can be written as
\begin{equation}
\label{eq:I}
I_r=\int_{B} \left( \int_{U} f_{\overline{s},u}(v)\prod_{i=1}^r\phi_i(a(\overline{s}_{{t}_i}-\overline{s})v x_{\overline{s},u})\, dv  \right)\, du,
\end{equation}
and
\begin{equation}
\label{eq:sobolev0}
\|f_{\overline{s},u}\|_{C^k}\ll  \|f\|_{C^k}\, \|\omega_\rho\|_{C^k}\ll \rho^{-\sigma}\, \|f\|_{C^k}.
\end{equation}
Furthermore,
\begin{align}
\label{eq:integral0}
\int_{U} \|f_{\overline{s},u}\|_{L^1(U)}\, du
&= \int_{U\times U} |f(a(-\overline{s})va(\overline{s})u)\omega_\rho(v)|\,dvdu\\
&= \left(\int_{U} |f(u)|\,du\right)
\left(\int_{U} \omega_\rho(v)\,dv\right)= \|f\|_{L^1(U)}.\nonumber
\end{align}
We decompose the integral $I_r$ in \eqref{eq:I} as
$$
I_r=I_r'(\eps)+I_r''(\eps),
$$
where 
$I_r'(\eps)$ is the integral over
$$
B_\eps:=\{u\in B:\, x_{\overline{s},u}\notin \cK_\eps \},
$$ 
and $I_r''(\eps)$ is the integral over $B\backslash B_\eps$. \\

To estimate $I_r'(\eps)$, we first recall that $\overline{s} \geq c_1 D$ by \eqref{eq:c1}, so if $D \geq T_0/c_1$, where $T_0$
is as in Theorem \ref{p:non_div} applied to $\cL = \cK_\eps$ and $B$, then the same theorem implies that there exists $\theta > 0$
such that
\begin{equation}
\label{eq:B}
|B_\eps|\ll \eps^\theta |B|
\end{equation}
for every $\eps\in (0,1)$, and thus
\begin{equation}
\label{eq:I0}
I_r'(\eps)\ll_{\Omega} 
\eps^\theta |B| \|f\|_{C^0}\left(\int_{U} \omega_\rho(v)\,dv\right) \prod_{i=1}^r\|\phi_i\|_{C^0}
\ll_\Omega \eps^\theta \|f\|_{C^0} \prod_{i=1}^r \|\phi_i\|_{C^0}.
\end{equation}
Let us now turn to the problem of estimating $I_{r}''(\eps)$. Since the Riemannian distance $d$ on $G$, restricted to 
$U$, and the Euclidean distance on $U$ are equivalent on a small open identity neighborhood, we see that 
$$
d\left(a(-\overline{t})va(\overline{t}),e\right)\ll e^{-\lfloor \overline {t}\rfloor }\, d(v,e).
$$
for $v\in B_{G}(\rho_0)$ and any $\overline{t}\in S^+$. Hence, using \eqref{eq:c2},
we obtain that for all $i=1,\ldots,r-1$,
$$
\phi_i(a(\overline{s}_{{t}_i}-\overline{s})v x_{\overline{s},u})
=\phi_i(a(\overline{s}_{{t}_i}-\overline{s}) x_{\overline{s},u})+O\left(e^{-c_2 D}\, \|\phi_i\|_{Lip}\right),
$$
and thus, for all $v\in B_G(\rho_0)$,
\begin{equation}
\label{eq:pprod}
\prod_{i=1}^{r-1}\phi_i(a(\overline{s}_{{t}_i}-\overline{s})v x_{\overline{s},u})
=\prod_{i=1}^{r-1}\phi_i(a(\overline{s}_{{t}_i}-\overline{s}) x_{\overline{s},u})+O_r\left(e^{-c_2 D}\, \prod_{i=1}^{r-1}\|\phi_i\|^*_{Lip}\right),
\end{equation}
where $\|\phi\|^*_{Lip}:=\max(\|\phi\|_{C^0},\|\phi_i\|_{Lip})$.
This leads to the estimate
\begin{align*}
I_r''(\eps)=&\int_{B\backslash B_\eps} 
\prod_{i=1}^{r-1}
\phi_i(a(\overline{s}_{{t}_i}-\overline{s}) x_{\overline{s},u})
\left( \int_{U} f_{\overline{s},u}(v) \phi_r(a(\overline{s}_{{t}_r}-\overline{s})v x_{\overline{s},u})\, dv  \right)\, du \\
&+O_r\left(e^{-c_2 D} \prod_{i=1}^{r-1}\|\phi_i\|^*_{Lip} \left(\int_{B\backslash B_\eps} \|f_{\overline{s},u}\|_{L^1(U)}\, du\right) \|\phi_r\|_{C^0}\right).
\end{align*}
Hence, using \eqref{eq:integral0}, we obtain that 
\begin{align*}
I_r''(\eps) =&\int_{B\backslash B_\eps} 
\prod_{i=1}^{r-1}
\phi_i(a(\overline{s}_{{t}_i}-\overline{s}) x_{\overline{s},u})
\left( \int_{U} f_{\overline{s},u}(v) \phi_r\left(a(\overline{s}_{{t}_r}-\overline{s})v x_{\overline{s},u}\right)\, dv  \right)\, du \\
&+O_r\left(e^{-c_2 D} \|f\|_{L^1(U)}\prod_{i=1}^{r}\|\phi_i\|^*_{Lip}\right).
\end{align*}
Since by \eqref{eq:c3},
$$
\int_{U} f_{\overline{s},u}(v) \phi_r(a(\overline{s}_{{t}_r}-\overline{s})v x_{\overline{s},u})\, dv=
\int_{U} f_{\overline{s},u}(v) \phi_r(b_z v x_{\overline{s},u})\, dv,
$$
we apply Theorem \ref{p:km} to estimate this integral.
We recall that $\hbox{supp}(f_{\overline{s},u})\subset B_{G}(\rho)$.
For $u\in B\backslash B_\eps$, we have $x_{\overline{s},u}\in \cK_\eps$,
so that $\iota(x_{\overline{s},u})\gg \eps^{m+n}$ by Proposition \ref{p:iota}.
In particular, we may take 
$$
\eps\gg \rho^{1/(m+n)}
$$
to arrange that $\iota(x_{\overline{s},u})>2\rho$. Hence, by applying
Theorem \ref{p:km}, we deduce that there exist $c, \gamma > 0$ such that 
\begin{align*}
\int_{U} f_{\overline{s},u}(v) \phi_r(b_z v x_{\overline{s},u})\, dv
=&\left(\int_{U} f_{\overline{s},u}(v)\, dv \right)
\left(\int_{\cX} \phi_r\, d\mu_\cX \right)\\
&+O\left( \rho\|f_{\overline{s},u}\|_{L^1(U)} \|\phi_r\|_{Lip} + \rho^{-c} e^{-\gamma z}\, \|f_{\overline{s},u}\|_{C^k} \|\phi_r\|_{L^2_k(\cX)}\right),
\end{align*}
for all $u\in B\backslash B_\eps$. Using \eqref{eq:sobolev0}--\eqref{eq:integral0} and $z\ge c_3\,D$, we deduce that
\begin{align*}
I_r''(\eps)=& \left(\int_{B\backslash B_\eps} 
\prod_{i=1}^{r-1}
\phi_i(a(\overline{s}_{{t}_i}-\overline{s}) x_{\overline{s},u})
\left( \int_{U} f_{\overline{s},u}(v)\, dv  \right)
\, du \right)\left(\int_{\cX} \phi_r\, d\mu_\cX \right)\\
&+O_{r,B}\left( \prod_{i=1}^{r-1}\|\phi_i\|_{C^0} \left(\rho\|f\|_{L^1(U)} \|\phi_r\|_{Lip} + \rho^{-(c+\sigma)} e^{-\gamma c_3\,D}\, \|f\|_{C^k}\|\phi_r\|_{L^2_k(\cX)}\right)\right)\\
&+O_r\left(e^{-c_2 D} \|f\|_{L^1(U)}\prod_{i=1}^{r}\|\phi_i\|^*_{Lip}\right).
\end{align*}
Since $\|f\|_{L^1(U)}\ll_\Omega \|f\|_{C^k}$, 
\begin{align}\label{eq:eee}
I_r''(\eps)=& \left(\int_{B\backslash B_\eps} 
\prod_{i=1}^{r-1}
\phi_i(a(\overline{s}_{{t}_i}-\overline{s}) x_{\overline{s},u})
\left( \int_{U} f_{\overline{s},u}(v)\, dv  \right)
\, du \right)\left(\int_{\cX} \phi_r\, d\mu_\cX \right)\\
&+O_{r,\Omega}\left( \left(\rho + \rho^{-(c+\sigma)} e^{-\gamma c_3\,D} +e^{-c_2 D}\right)
\|f\|_{C^k}\prod_{i=1}^{r} \cN_k(\phi_i) \right).\nonumber
\end{align}
Applying \eqref{eq:pprod} one more time (in the backward direction), we get
\begin{align*}
&\int_{B\backslash B_\eps} 
\prod_{i=1}^{r-1}
\phi_i(a(\overline{s}_{{t}_i}-\overline{s}) x_{\overline{s},u})
\left( \int_{U} f_{\overline{s},u}(v)\, dv\right)\, du\\
=&
\int_{B\backslash B_\eps} 
\left( \int_{U} f_{\overline{s},u}(v)\prod_{i=1}^{r-1}
\phi_i(a(\overline{s}_{{t}_i}-\overline{s}) vx_{\overline{s},u})\, dv\right) \, du\\
&+O_r\left(e^{-c_2 D}\, \left(\int_B \|f_{\overline{s},u}\|_{L^1(U)}\,du\right) \prod_{i=1}^{r-1}\|\phi_i\|^*_{Lip}\right).
\end{align*}
It follows from \eqref{eq:B} that
\begin{align*}
&\int_{B\backslash B_\eps} 
\left( \int_{U} f_{\overline{s},u}(v)\prod_{i=1}^{r-1}
\phi_i(a(\overline{s}_{{t}_i}-\overline{s}) vx_{\overline{s},u})\, dv\right) \, du\\
=&
\int_{B} 
\left( \int_{U} f_{\overline{s},u}(v)\prod_{i=1}^{r-1}
\phi_i(a(\overline{s}_{{t}_i}-\overline{s}) vx_{\overline{s},u})\, dv\right) \, du\\
&+O_r \left(\eps^\theta 
\left(\int_B \|f_{\overline{s},u}\|_{L^1(U)}\,du\right) \prod_{i=1}^{r-1}\|\phi_i\|_{C^0}\right),
\end{align*}
where we recognize the first term as $I_{r-1}$. Using \eqref{eq:I} and \eqref{eq:integral0}, we now conclude that
\begin{align*}
&\int_{B\backslash B_\eps} 
\prod_{i=1}^{r-1}
\phi_i(a(\overline{s}_{{t}_i}-\overline{s}) x_{\overline{s},u})
\left( \int_{U} f_{\overline{s},u}(v)\, dv\right)\, du\\
=&I_{r-1}
+O_{r} \left(\left(\eps^\theta + e^{-c_2 D}\right)
\|f\|_{L^1(U)} \prod_{i=1}^{r-1}\|\phi_i\|^*_{Lip}\right).
\end{align*}
Hence, combining this estimate with \eqref{eq:eee}, we deduce that
\begin{align*}
I_r''(\eps)=& I_{r-1}\left(\int_{\cX} \phi_r\, d\mu_\cX \right)\\
&+O_{r,\Omega}\left( \left(\eps^\theta+\rho + \rho^{-(c+\sigma)} e^{-\gamma c_3\,D} +e^{-c_2 D}\right)
\|f\|_{C^k}\prod_{i=1}^{r} \cN_k(\phi_i) \right),
\end{align*}
and thus, in view of \eqref{eq:I0}, 
\begin{align*}
I_r=& I_{r-1}'(\eps)+I_{r-1}''(\eps)=I_{r-1}\left(\int_{\cX} \phi_r\, d\mu_\cX \right)\\
&+O_{r,\Omega}\left( \left(\eps^\theta+\rho + \rho^{-(c+\sigma)} e^{-\gamma c_3\,D} +e^{-c_2 D}\right)
\|f\|_{C^k}\prod_{i=1}^{r} \cN_k(\phi_i) \right).
\end{align*}
This estimate holds whenever $\rho<\rho_0$ and $\eps\gg \rho^{1/(m+n)}$. Taking $\rho=e^{-c_4 D}$ for sufficiently small $c_4>0$ and $\eps\ll \rho^{1/(m+n)}$,
we conclude that there exists $\delta>0$ such that for all sufficiently large $D$,
\begin{align}\label{eq:l2l}
I_r=& I_{r-1}\left(\int_{\cX} \phi_r\, d\mu_\cX \right)+
O_{r,\Omega}\left( e^{-\delta D}\,
\|f\|_{C^k}\prod_{i=1}^{r} \cN_k(\phi_i) \right).
\end{align}
The exponent $\delta$ depends on the constants $c_2$ and $c_3$ given by Lemma \ref{l:s}
and the parameters $\theta,c,\sigma,\gamma$. In particular, $\delta$ is independent of $r$.
By possibly enlarging the implicit constants we can ensure that the estimate \eqref{eq:l2l} also holds 
for all $r$-tuples $(t_1,\ldots,t_r)$, and not just the ones with sufficiently large $D(t_1,\ldots,t_r)$.
By iterating the estimate \eqref{eq:l2l}, using that $I_0$ is a constant, the proof of Theorem \ref{th:corr} is finished.

\section{CLT for functions with compact support}\label{sec:CLT_compact}

Let $a=\hbox{diag}(a_1,\ldots,a_{m+n})$ 
be a diagonal linear map of $\bR^{m+n}$
with 
$$
a_1,\ldots,a_m>1,\quad 0<a_{m+1},\ldots,a_{m+n} <1,\quad\hbox{and}\quad a_1\cdots a_{m+n}=1.
$$
The map $a$ defines a continuous self-map of the space $\cX$, which preserves $\mu_{\cX}$.
We recall that the torus $\cY=U\bZ^{m+n}\subset \cX$ is equipped with the probability measure $\mu_\cY$.
In this section, we shall prove a Central Limit Theorem for the averages
\begin{equation}
\label{eq:FNcpt}
F_N:=\frac{1}{\sqrt{N}} \sum_{s=0}^{N-1} \left(\phi\circ a^s-\mu_\cY(\phi\circ a^s)\right),
\end{equation}
\emph{restricted to $\cY$}, for a fixed function $\phi \in C^\infty_c(\cX)$. We stress that this result is
not needed in the proof of Theorem \ref{th:CLT_forms}, but we nevertheless include it here because
we feel that its proof might be instructive before entering the proof of the similar, but far more 
technical, Theorem \ref{th:CLT_char_Siegel}.

\begin{theorem}\label{th:CLT_compact}
For every $\xi\in \mathbb{R}$, 
$$
\mu_\cY \left(\{y\in \cY:\, F_N(y)<\xi\}\right)\to \hbox{\rm Norm}_{\sigma_\phi}(\xi)
$$	
as $N\to \infty$, where 
$$
\sigma_\phi^2:=\sum_{s=-\infty}^\infty \left( \int_{\cX} (\phi\circ a^s)\phi \, d\mu_\cX
-\mu_\cX(\phi)^2\right).
$$
\end{theorem}
\begin{remark}
It follows from Theorem \ref{th:cor0} that the variance $\sigma_\phi$ is finite. 
\end{remark}

Our main tool in the proof of Theorem \ref{th:CLT_compact} will be the estimates on higher-order correlations 
established in Section \ref{sec:correlations}. To make notations less heavy, we shall use a simplified version of 
Corollary \ref{cor:corr0} stated in terms of $C^k$-norms (note that $\cN_k \ll \|\cdot\|_{C^k}$):
\begin{corollary}\label{cor:corr}
	There exists $\delta>0$ such that for every  
	$\phi_0,\ldots,\phi_r\in C_c^\infty(\mathcal{X})$
	and $t_1,\ldots,t_r>0$, we have 
	\begin{align*}
	\int_{\mathcal{Y}} \phi_0(y) \left( \prod_{i=1}^r \phi_i(a_{t_i}y) \right )d\mu_\cY(y)=&  \left(\int_{\mathcal{Y}}\phi_0 \, d\mu_\cY \right)\prod_{i=1}^r \left(\int_{\mathcal{X}}\phi_i \, d\mu_\cX\right)\\
	& +O_{r}\left( e^{-\delta D(t_1,\ldots,t_r)}\, \prod_{i=0}^r \|\phi_i\|_{C^k}\right).
	\end{align*}
\end{corollary}

\subsection{The method of cumulants}

Let $(\cX,\mu)$ be a probability space. Given bounded measurable functions $\phi_1,\ldots,\phi_r$
on $\cX$, we define their \emph{joint cumulant} as
$$
\cum^{(r)}_\mu(\phi_1,\ldots,\phi_r) = \sum_{\cP} (-1)^{|\cP|-1} (|P|-1)!\prod_{I \in \cP} 
\int_{\cX}\Big(\prod_{i\in I} \phi_i\Big)\, d\mu,
$$
where the sum is taken over all partitions $\cP$ of the set $\{1,\ldots,r\}$.
When it is clear from the context, we skip the subscript $\mu$.
For a bounded measurable function $\phi$ on $\cX$, we also set
$$
\cum^{(r)}(\phi) = \cum^{(r)}(\phi,\ldots,\phi).
$$
We shall use the following classical CLT-criterion (see, for instance, \cite{FS}).

\begin{proposition} 
	\label{th:CLT}
	Let $(F_T)$ be a sequence of real-valued bounded measurable functions such that
	\begin{equation}
	\label{eq:cond1}
	\int_\cX F_T\, d\mu=0 \qand \sigma^2 := \lim_{T\to\infty} \int_{\cX}F_T^2\, d\mu < \infty 
	\end{equation}
	and
	\begin{equation}
	\lim_{T\to\infty} \cum^{(r)}(F_T) = 0, \quad \textrm{for all $r \geq 3$}.\label{eq:cond2}
	\end{equation}
	Then for every $\xi\in\mathbb{R}$, 
	$$
	\mu(\{F_T<\xi\})	\rightarrow \hbox{\rm Norm}_\sigma (\xi)\quad \hbox{	as $T\to\infty$.}
	$$
\end{proposition}

Since all the moments of a random variable can be expressed in terms of its cumulants,
this criterion is equivalent to the more widely known ``Method of Moments''. However, the
cumulants have curious, and very convenient, cancellation properties that will play an important
role in our proof of Theorem \ref{th:CLT_compact}. \\

For a partition $\cQ$ of $\{1,\ldots,r\}$, we define the \emph{conditional joint cumulant}
with respect to $\cQ$ as
$$
\cum^{(r)}_\mu(\phi_1,\ldots,\phi_r|\cQ) = \sum_{\cP} (-1)^{|\cP|-1} (|\cP|-1)!\prod_{I \in \cP} 
\prod_{J \in \cQ}
\int_{\cX}\Big(\prod_{i\in I\cap J} \phi_i\Big)\, d\mu,
$$
In what follows, we shall make frequent use of the following proposition.

\begin{proposition}[\cite{BG}, Prop.~8.1]\label{p:cum_cond}
For any partition $\cQ$ with $|\cQ|\ge 2$,
\begin{equation}
\label{eq:conditional}
\cum^{(r)}_\mu(\phi_1,\ldots,\phi_r|\cQ)=0,
\end{equation}
for all $\phi_1,\ldots,\phi_r \in L^\infty(\cX,\mu)$.
\end{proposition}

\subsection{Estimating cumulants}\label{sec:cumm_comp}
Fix $\phi \in C_c^\infty(\cX)$. It will be convenient to write $\psi_s(y):=\phi(a^s y)-\mu_\cY (\phi\circ a^s)$,
so that 
$$
F_N:=\frac{1}{\sqrt{N}} \sum_{s=0}^{N-1} \psi_s \qand \int_{\cY} F_N\, d\mu_\cY=0.
$$
In this section, we shall estimate cumulants of the form
\begin{equation}
\label{eq:cum_sum}
\hbox{Cum}_{\mu_\cY}^{(r)}(F_N)=\frac{1}{N^{r/2}} \sum_{s_1,\ldots,s_r=0}^{N-1} 
\hbox{Cum}_{\mu_\cY}^{(r)}(\psi_{s_1},\ldots,\psi_{s_r})
\end{equation}
for $r\ge 3$. Since we shall later need to apply these estimate in cases when the function $\phi$ is allowed to vary 
with $N$, it will be important to keep track of the dependence on $\phi$ in our estimates. \\
 
We shall decompose \eqref{eq:cum_sum} into sub-sums where the parameters $s_1,\ldots,s_r$
are either ``separated'' or ``clustered'', and it will also be important to control their sizes. For this purpose, 
it will be convenient to consider the set  $\{0,\ldots,N-1\}^r$ as a subset of $\bR_+^{r+1}$ via the embedding 
$(s_1,\ldots,s_r)\mapsto (0,s_1,\ldots,s_r)$. Following the ideas developed in the paper \cite{BG}, we define for 
non-empty subsets $I$ and $J$ of $\{0,\ldots, r\}$ and $\overline{s} = (s_0,\ldots,s_r) \in \bR_+^{r+1}$, 
$$
\rho^{I}(\overline{s}) := \max\big\{ |s_i-s_j| \, : \, i,j \in I \big\}\quad\hbox{and}
\quad
\rho_{I,J}(\overline{u}):= \min\big\{ |s_i-s_j| \, : \, i \in I, \enskip j \in J \big\},
$$
and if $\cQ$ is a partition of $\{0,\ldots,r\}$, we set
$$
\rho^{\cQ}(\overline{s}) := \max\big\{ \rho^{I}(\overline{s}) \, : \, I \in \cQ \big\}
\quad\hbox{and}\quad
\rho_{\cQ}(\overline{s}) := \min\big\{ \rho_{I,J}(\overline{s}) \, : \, I \neq J, \enskip I, J \in \cQ \big\}.
$$
For $0 \leq \alpha < \beta$, we define
$$
\Delta_{\cQ}(\alpha,\beta) 
:= 
\big\{ 
\overline{s} \in \bR_+^{r+1} \, : \, 
\rho^{\cQ}(\overline{s}) \leq \alpha, 
\qen
\rho_{\cQ}(\overline{s}) > \beta \big\}
$$
and
$$
\Delta(\alpha):= 
\big\{ 
\overline{s} \in \bR_+^{r+1} \, : \, 
\rho(s_i,s_j) \leq \alpha \hbox{ for all $i,j$}\big\}.
$$
The following decomposition of $\bR_+^{r+1}$ was established in our paper \cite[Prop. 6.2]{BG}: given 
\begin{equation}\label{eq:inequality}
0=\alpha_0<\beta_1<\alpha_1=(3+r)\beta_1<\beta_2<\cdots <\beta_r<\alpha_r=(3+r)\beta_r<\beta_{r+1},
\end{equation}
we have
\begin{equation}
\label{eq:decomp}
\bR_+^{r+1} = \Delta(\beta_{r+1}) \cup \Big( \bigcup_{j=0}^{r} \bigcup_{|\cQ| \geq 2} \Delta_{\cQ}(\alpha_j,\beta_{j+1}) \Big),
\end{equation}
where the union is taken over the partitions $\cQ$ of $\{0,\ldots,r\}$ with $|\cQ|\ge 2$. Upon taking restrictions, we also have 
\begin{equation}\label{eq:decomp_000}
\{0,\ldots,N-1\}^r =\Omega(\beta_{r+1};N) \cup \Big( \bigcup_{j=0}^{r} \bigcup_{|\cQ| \geq 2} \Omega_{\cQ}(\alpha_j,\beta_{j+1};N) \Big),
\end{equation}
for all $N \geq 2$, where 
\begin{align*}
\Omega(\beta_{r+1};N) &:=\{0,\ldots,N-1\}^r \cap \Delta(\beta_{r+1}),\\
\Omega_{\cQ}(\alpha_j,\beta_{j+1};N) &:=\{0,\ldots,N-1\}^r \cap \Delta_{\cQ}(\alpha_j,\beta_{j+1}).
\end{align*}
In order to estimate the cumulant \eqref{eq:cum_sum}, we shall separately estimate the sums over $\Omega(\beta_{r+1};N)$ and $\Omega_{\cQ}(\alpha_j,\beta_{j+1};N)$, the exact choices of the sequences $(\alpha_j)$ and $(\beta_j)$ will be fixed at the very
end of our argument. 

\subsubsection{\textbf{\underline{Case 0:} Summing over $\Omega(\beta_{r+1};N)$}.}

In this case, $s_i\le \beta_{r+1}$ for all $i$, and thus 
$$
|\Omega(\beta_{r+1};N)|\le (\beta_{r+1}+1)^r.
$$
Hence, 
\begin{equation}
\label{eq:cum1}
\sum_{(s_1,\ldots,s_r)\in \Omega(\beta_{r+1};N)} 
\hbox{Cum}_{\mu_\cY}^{(r)}(\psi_{s_1},\ldots,\psi_{s_r})\ll (\beta_{r+1}+1)^r
\|\phi\|_{C^0}^r,
\end{equation}
where the implied constants may depend on $r$, but we shall henceforth omit this subscript to simplify notations. 

\subsubsection{\textbf{\underline{Case 1:} Summing over $\Omega_{\cQ}(\alpha_j,\beta_{j+1};N)$ with
$\cQ=\{\{0\},\{1,\ldots,r\} \}$}.}

In this case, we have
$$
|s_{i_1}-s_{i_2}|\le \alpha_j\quad\hbox{ for all $i_1,i_2$, }
$$
so that
it follows that 
$$
|\Omega_{\cQ}(\alpha_j,\beta_{j+1};N)|\ll N\alpha_{j}^{r-1}.
$$
Hence, 
\begin{equation}
\label{eq:cum2}
\frac{1}{N^{r/2}} \sum_{(s_1,\ldots,s_r)\in \Omega_Q(\alpha_j,\beta_{j+1};N)}
|\hbox{Cum}_{\mu_\cY}^{(r)}(\psi_{s_1},\ldots,\psi_{s_r})|
\ll N^{1-r/2}\alpha_j^{r-1}\|\phi\|_{C^0}^r.
\end{equation}

\subsubsection{\textbf{\underline{Case 2:} Summing over $\Omega_{\cQ}(\alpha_j,\beta_{j+1};N)$ with
	$|\cQ|\ge  2$ and $\cQ\ne \{\{0\},\{1,\ldots,r\} \}$}.}

In this case, the partition $\cQ$ defines a non-trivial partition $\cQ'=\{I_0,\ldots,I_\ell\}$ of $\{1,\ldots,r\}$ such that 
for all $(s_1,\ldots, s_r)\in \Omega_Q(\alpha_j,\beta_{j+1};N)$, we have 
\begin{equation}
\label{eq:ineq0}
|s_{i_1}-s_{i_2}|\le \alpha_j \;\; \hbox{if $i_1\sim_{\cQ'} i_2$} \qand
|s_{i_1}-s_{i_2}|> \beta_{j+1} \;\; \hbox{if $i_1\not\sim_{\cQ'} i_2$},
\end{equation}
and
\[
s_i\le \alpha_j\;\; \hbox{for all $i\in I_0$},\nonumber \qand s_i>\beta_{j+1} \;\; \hbox{for all $i\notin I_0$}. 
\]
In particular,
\begin{equation}
\label{eq:Dcase2}
D(s_{i_1},\ldots, s_{i_\ell})\ge \beta_{j+1},
\end{equation}
Let $I$ be an arbitrary subset of $\{1,\ldots,r\}$; we shall show that 
\begin{equation}
\label{eq:approx}
\int_{\cY} \left( \prod_{i\in I} \psi_{s_i}\right)\, d\mu_\cY
\approx \prod_{h=0}^\ell \left(\int_{\cY} \left(\prod_{i\in I\cap I_h} \psi_{s_i} \right)\,d\mu_\cY\right),
\end{equation}
where we henceforth shall use the convention that the product is equal to one when $I\cap I_h=\emptyset$. \\

Let us estimate  the right hand side of \eqref{eq:approx}. We begin by setting
$$
\Phi_0:=\prod_{i\in I\cap I_0} \psi_{s_i}.
$$
It is easy to see that there exists $\xi=\xi(m,n,k)>0$ such that
\begin{equation}
\label{eq:Phi_0}
\|\Phi_0\|_{C^k}\ll \prod_{i\in I\cap I_0} \|\phi\circ a^{s_i}-\mu_\cY(\phi\circ a^{s_i})\|_{C^k} \ll e^{|I\cap I_0|\xi\, \alpha_j}\,\|\phi\|_{C^k}^{|I\cap I_0|}.
\end{equation}
To prove \eqref{eq:approx}, 
we expand $\psi_{s_i}=\phi\circ a^{s_i}-\mu_\cY(\phi\circ a^{s_i})$
for $i\in I\backslash I_0$ and get
\begin{eqnarray*}
\label{eq:psi00}
\int_{\cY} \left(\prod_{i\in I} \psi_{s_i}\right)\,d\mu_\cY 
&=& 
\sum_{J\subset I\backslash I_0} (-1)^{|I\backslash (J\cup I_0)|} \cdot \\
&\cdot&
\left(\int_{\cY} \Phi_0 \left(\prod_{i\in J} \phi\circ a^{s_i}\right)\,d\mu_\cY\right) 
\prod_{i\in I\backslash (J\cup I_0)} \left(\int_{\cY} (\phi\circ a^{s_i})\,d\mu_\cY\right).
\end{eqnarray*}
We recall that when $i\notin I_0$, we have $s_i\ge \beta_{j+1}$, and thus it follows from Corollary \ref{cor:corr} with $r=1$ that
\begin{equation}\label{eq:psi_0}
\int_{\cY} (\phi\circ a^{s_i})\,d\mu_\cY
=\mu_{\cX}(\phi)+O\left(e^{-\delta \beta_{j+1} }\,\|\phi\|_{C^k}\right), \quad \hbox{ with $i\notin I_0$}.
\end{equation}
To estimate the other integrals in \eqref{eq:psi00}, we also apply Corollary \ref{cor:corr}.
Let us  first fix a subset $J \subset I \setminus I_0$ and for each $1 \leq h \leq l$, we pick $i_h\in I_h$, 
and set
$$
\Phi_h:=\prod_{i\in J\cap I_h} \phi\circ a^{s_i-s_{i_h}}.
$$
Then 
$$
\int_{\cY} \Phi_0 \left(\prod_{i\in J} \phi\circ a^{s_i}\right)\,d\mu_\cY=
\int_{\cY} \Phi_0 \left( \prod_{h=1}^\ell \Phi_h\circ a^{s_{i_h}}
\right)\,d\mu_\cY.
$$
We note that for $i\in I_h$, we have $|s_i-s_{i_h}|\le \alpha_j$, and 
thus there exists $\xi=\xi(m,n,k)>0$ such that
\begin{equation}
\label{eq:sss}
\|\Phi_h\|_{C^k}\ll \prod_{i\in J\cap I_h} \|\phi\circ a^{s_i-s_{i_h}}\|_{C^k} \ll e^{|J\cap I_h|\xi\, \alpha_j}\,\|\phi\|_{C^k}^{|J\cap I_h|}.
\end{equation}
Using \eqref{eq:Dcase2}, Corollary \ref{cor:corr} implies that
\begin{align*}
\int_{\cY} \Phi_0 \left( \prod_{h=1}^\ell \Phi_h\circ a^{s_{i_h}}\right)\,d\mu_\cY
= &\left(\int_{\cY}\Phi_0 \, d\mu_\cY \right)\prod_{h=1}^\ell \left(\int_{\cX}\Phi_h \, d\mu_\cX\right)\\
 &+O\left( e^{-\delta \beta_{j+1}}\, \prod_{h=0}^\ell \|\Phi_h\|_{C^k}\right).
\end{align*}
Using \eqref{eq:Phi_0} and \eqref{eq:sss} and invariance of  the measure $\mu_\cX$, we deduce that
\begin{align*}
\int_{\cY} \Phi_0 \left( \prod_{h=1}^\ell \Phi_h\circ a^{s_{i_h}}\right)\,d\mu_\cY
= &\left(\int_{\cY}\Phi_0 \, d\mu_\cY \right)\prod_{h=1}^\ell \left(\int_{\mathcal{X}}\left(\prod_{i\in J\cap I_h} \phi\circ a^{s_i} \right)\, d\mu_\cX\right)\\
&+O\left( e^{-(\delta \beta_{j+1}-r\xi \alpha_j)}\, \|\phi\|_{C^k}^{|(I\cap I_0)\cup J|}\right).
\end{align*}
Hence, we conclude that
\begin{align}
\int_{\cY} \Phi_0 \left(\prod_{i\in J} \phi\circ a^{s_i}\right)\,d\mu_\cY
= &\left(\int_{\mathcal{Y}}\Phi_0 \, d\mu_\cY \right)\prod_{h=1}^\ell \left(\int_{\mathcal{X}}\left(\prod_{i\in J\cap I_h} \phi\circ a^{s_i} \right)\, d\mu_\cX\right) \label{eq:p000}\\
&+O\left( e^{-(\delta \beta_{j+1}-r\xi \alpha_j)}\,  \|\phi\|_{C^k}^{|(I\cap I_0)\cup J|}\right).\nonumber
\end{align}
We shall choose the parameters $\alpha_j$ and $\beta_{j+1}$ so that 
\begin{equation}\label{eq:i1}
\delta\beta_{j+1}-r\xi\alpha_j>0.
\end{equation}
Substituting \eqref{eq:psi_0} and \eqref{eq:p000} in \eqref{eq:psi00}, we deduce that 
\begin{align} \label{eq:ppsi}
&\int_{\cY} \left(\prod_{i\in I} \psi_{s_i}\right)\,d\mu_\cY  \\ 
=&
\sum_{J\subset I\backslash I_0} (-1)^{|I\backslash (J\cup I_0)|} 
\left(\int_{\mathcal{Y}}\Phi_0 \, d\mu_\cY \right)\prod_{h=1}^\ell \left(\int_{\mathcal{X}}\left(\prod_{i\in J\cap I_h} \phi\circ a^{s_i} \right)\, d\mu_\cX\right) \mu_\cX(\phi)^{|I\backslash (J\cup I_0)|} \nonumber \\
& +O \left( e^{-(\delta \beta_{j+1}-r\xi \alpha_j)}\, \|\phi\|_{C^k}^{|I|}\right).\nonumber
\end{align}

Next, we estimate the right hand side of \eqref{eq:approx}.
Let us fix $1 \leq h \leq l$ and for a subset $J \subset I \cap I_h$, we define 
$$
\Phi_J:=\prod_{i\in J} \phi\circ a^{s_i-s_{i_h}}.
$$
As in \eqref{eq:sss}, for some $\xi>0$,
$$
\|\Phi_J\|_{C^k}\ll \prod_{i\in J} \|\phi\circ a^{s_i-s_{i_h}}\|_{C^k} \ll e^{|J|\xi\, \alpha_j}\,\|\phi\|_{C^k}^{|J|}.
$$
Applying Corollary \ref{cor:corr} to the function $\Phi_J$ and using that $s_{i_h}>\beta_{j+1}$, we get 
\begin{align}
\int_{\cY} \left(\prod_{i\in J} \phi\circ a^{s_i}\right)\,d\mu_\cY
&=\int_{\cY} (\Phi_J\circ a^{s_{i_h}})\,d\mu_\cY \label{eq:PPsi_0}\\
&=\int_{\cX}\Phi_J\, d\mu_\cX+O\left(e^{-\delta \beta_{j+1} }\,\|\Phi_J\|_{C^k}\right) \nonumber\\
&=\int_{\cX}\left(\prod_{i\in J} \phi\circ a^{s_i}\right)\, d\mu_\cX+O\left(e^{-\delta \beta_{j+1} } e^{r\xi\, \alpha_j}\,\|\phi\|_{C^k}^{|J|}\right), \nonumber
\end{align}
where we have used $a$-invariance of $\mu_\cX$.
Combining \eqref{eq:psi_0} and \eqref{eq:PPsi_0}, we deduce that 
\begin{align}
\int_{\cY} \left(\prod_{i\in I\cap I_h} \psi_{s_i}\right)\,d\mu_\cY \label{eq:prod}
=&
\sum_{J\subset I\cap I_h} (-1)^{|(I\cap I_h)\backslash J|} 
\left(\int_{\cX} \left(\prod_{i\in J} \phi\circ a^{s_i}\right)\,d\mu_\cX\right)
\mu_{\cX}(\phi)^{|(I\cap I_h)\backslash J|}\\
&+O\left(e^{-\delta \beta_{j+1} } e^{r\xi\, \alpha_j}\,\|\phi\|_{C^k}^{|I\cap I_h|}\right) \nonumber\\
=& \int_{\cX} \prod_{i\in I\cap I_h} \left(\phi\circ a^{s_i}-\mu_\cX(\phi)\right)\,d\mu_\cX
+O\left(e^{-(\delta \beta_{j+1}-r\xi\alpha_j )}\,\|\phi\|_{C^k}^{|I\cap I_h|}\right),\nonumber
\end{align}
which implies
\begin{align*}
&\prod_{h=0}^\ell \left(\int_{\cY} \left(\prod_{i\in I\cap I_h} \psi_{s_i} \right)\,d\mu_\cY\right)\\
=& \left(\int_{\mathcal{Y}}\Phi_0 \, d\mu_\cY \right) \prod_{h=1}^\ell \left(\int_{\cX} \prod_{i\in I\cap I_h} \left(\phi\circ a^{s_i}-\mu_\cX(\phi)\right)\,d\mu_\cX\right)\\
&+O\left(e^{-(\delta \beta_{j+1} - r\xi\alpha_j)} \,\|\phi\|_{C^k}^{r}\right).
\end{align*}
Furthermore, multiplying out the products over $I\cap I_h$, we get
\begin{align}\label{eq:ppp}
&\prod_{h=0}^\ell \left(\int_{\cY} \left(\prod_{i\in I\cap I_h} \psi_{s_i} \right)\,d\mu_\cY\right)\\
=& \left(\int_{\mathcal{Y}}\Phi_0 \, d\mu_\cY \right) 
\sum_{J\subset I\backslash I_0} (-1)^{|I\backslash (I_0\cup J)|}\prod_{h=1}^\ell \left(\int_{\cX} \prod_{i\in I_h\cap J} \phi\circ a^{s_i}\,d\mu_\cX\right)
\mu_\cX(\phi)^{|I\backslash (I_0\cup J)|} \nonumber \\
&+O\left(e^{-(\delta \beta_{j+1} - r\xi\alpha_j)} \,\|\phi\|_{C^k}^{|I|}\right). \nonumber
\end{align}
Comparing \eqref{eq:ppsi} and \eqref{eq:ppp}, we finally conclude that
\begin{align*}
\int_{\cY} \left(\prod_{i\in I} \psi_{s_i}\right)\,d\mu_\cY =& \prod_{h=0}^\ell \left(\int_{\cY} \left(\prod_{i\in I\cap I_h} \psi_{s_i} \right)\,d\mu_\cY\right)\\
&+O\left(e^{-(\delta \beta_{j+1} - r\xi\alpha_j)} \,\|\phi\|_{C^k}^{|I|}\right)
\end{align*}
when $(s_1,\ldots,s_r)\in\Omega_{\cQ}(\alpha_j,\beta_{j+1};N)$.
This establishes \eqref{eq:approx} with an explicit error term.
This estimate implies that for the partition $\cQ'=\{I_0,\ldots, I_\ell\}$,
$$
\hbox{Cum}_{\mu_\cY}^{(r)}(\psi_{s_1},\ldots,\psi_{s_r})=
\hbox{Cum}_{\mu_\cY}^{(r)}(\psi_{s_1},\ldots,\psi_{s_r}|\cQ')+
O\left(e^{-(\delta \beta_{j+1} - r\xi\alpha_j)} \,\|\phi\|_{C^k}^{r}\right)
$$
By Proposition \ref{p:cum_cond},
$$
\hbox{Cum}_{\mu_\cY}^{(r)}(\psi_{s_1},\ldots,\psi_{s_r}|\cQ')=0,
$$
so it follows that for all $(s_1,\ldots, s_r)\in \Omega_Q(\alpha_j,\beta_{j+1};N)$,
\begin{equation}
\label{eq:cum3}
\left|\hbox{Cum}_{\mu_\cY}^{(r)}(\psi_{s_1},\ldots,\psi_{s_r})\right|
\ll e^{-(\delta \beta_{j+1} - r\xi\alpha_j)} \,\|\phi\|_{C^k}^{r}.
\end{equation}

\subsubsection{\textbf{Final estimates on the cumulants.}}

Let us now return to \eqref{eq:cum_sum}. Upon decomposing this sum into the regions discussed
above, and applying the estimates \eqref{eq:cum1}, \eqref{eq:cum2} and \eqref{eq:cum3} to respective
region, we get the bound
\begin{align}
\left|\hbox{Cum}_{\mu_\cY}^{(r)}(F_N)\right|\ll 
&\left( (\beta_{r+1}+1)^r N^{-r/2}+ \left({\max}_j\,\, \alpha_j^{r-1}\right) N^{1-r/2}  \right) \|\phi\|_{C^0}^r \label{eq:cumm_last} \\
&\;\; +
  N^{r/2} \left({\max}_{j}\,\, e^{-(\delta \beta_{j+1} - r\xi\alpha_j)}\right) \,\|\phi\|_{C^k}^{r}.\nonumber
\end{align}
This estimate holds provided that \eqref{eq:inequality} and \eqref{eq:i1} hold, namely when
\begin{equation}
\label{eq:alpha}
\alpha_j=(3+r)\beta_{j}<\beta_{j+1}\quad\hbox{and}\quad \delta\beta_{j+1}-r\xi\alpha_j>0\quad\quad\hbox{ for $j=1,\ldots,r$.}
\end{equation}
Given any $\gamma>0$, we define the parameters $\beta_j$ inductively by the formula
\begin{equation}
\label{eq:beta}
\beta_1=\gamma\quad\hbox{and}\quad \beta_{j+1}=\max\left(\gamma+(3+r)\beta_{j}, \gamma+\delta^{-1}r(3+r)\xi\beta_{j}\right).
\end{equation}
It easily follows by induction that $\beta_{r+1}\ll_r \gamma$,
and we deduce from \eqref{eq:cumm_last} that
$$
\left|\hbox{Cum}_{\mu_\cY}^{(r)}(F_N)\right|\ll 
\left( (\gamma+1)^r N^{-r/2}+\gamma^{r-1}N^{1-r/2}\right) \,\|\phi\|_{C^0}^{r} +
N^{r/2} e^{-\delta \gamma} \,\|\phi\|_{C^k}^{r}.
$$
Taking $\gamma=(r/\delta)\log N$, we conclude that 
when $r\ge 3$,
\begin{equation}
\label{eq:CumNcptzero}
\hbox{Cum}_{\mu_\cY}^{(r)}(F_N)\to 0\quad\hbox{as $N\to\infty$.}
\end{equation}

\subsection{Estimating the variance}\label{sec:var_compact}

In this section, we wish to compute the limit
\begin{align*}
\|F_N\|_{L^2(\cY)}^2 =&\frac{1}{N}\sum_{s_1=0}^{N-1}\sum_{s_2=0}^{N-1} \int_{\cY}
\psi_{s_1}\psi_{s_2}\, d\mu_\cY.
\end{align*}
Setting $s_1=s+t$ and $s_2=t$, we rewrite the above sums as
\begin{align}\label{eq:FF_N}
\|F_N\|_{L^2(\cY)}^2 
=& \Theta_N(0)
+2\sum_{s=1}^{N-1}\Theta_N(s),
\end{align}
where
$$
\Theta_N(s):=\frac{1}{N}\sum_{t=1}^{N-1-s} \int_{\cY}\psi_{s+t}\psi_t\, d\mu_\cY.
$$
Since $\psi_t=\phi\circ a^t-\mu_\cY(\phi\circ a^t)$, 
\begin{equation}
\label{eq:psist}
\int_{\cY}\psi_{s+t}\psi_t\, d\mu_\cY
=\int_{\cY}(\phi\circ a^{s+t})(\phi\circ a^{t})\, d\mu_\cY
-\mu_{\cY}(\phi\circ a^{s+t})
\mu_{\cY}(\phi\circ a^{t}).
\end{equation}
It follows from Corollary \ref{cor:corr} that 
for \emph{fixed} $s$ and as $t\to\infty$,
$$
\int_{\cY}(\phi\circ a^{s+t})(\phi\circ a^{t})\, d\mu_\cY\to
\int_{\cX} (\phi\circ a^{s})\phi\, d\mu_\cX,
\qand
\mu_{\cY}(\phi\circ a^{t})\to \mu_{\cX}(\phi).
$$
We conclude that 
$$
\int_{\cY}\psi_{s+t}\psi_t\, d\mu_\cY
\to \Theta_\infty(s):=\int_{\cX} (\phi\circ a^{s})\phi\, d\mu_\cX-\mu_\cX(\phi)^2,
$$
as $t \ra \infty$, and for fixed $s$,
$$
\Theta_N(s)\to \Theta_\infty(s)\quad \hbox{ as $N\to\infty$.}
$$
If one carelessly interchange limits above, one \emph{expects} that as $N\to \infty$,
\begin{align}\label{eq:limit}
\|F_N\|_{L^2(\cY)}^2 
\to  \Theta_\infty(0)
+2\sum_{s=1}^{\infty}\Theta_\infty(s)
=\sum_{s=-\infty}^{\infty} \Theta_\infty(s).
\end{align}
To prove this limit rigorously, we need to say a bit more to ensure, say, dominated convergence. \\

It follows from Corollary \ref{cor:corr} that 
\begin{align*}
\int_{\cY}(\phi\circ a^{s+t})(\phi\circ a^{t})\, d\mu_\cY&=
\mu_{\cX}(\phi)^2+
O\left(e^{-\delta \min(s,t)}\, \|\phi\|_{C^k}^2\right),\\
\int_{\cY}(\phi\circ a^{s+t})\, d\mu_\cY
& = \mu_{\cX}(\phi)+
O\left(e^{-\delta(s+t)} \, \|\phi\|_{C^k}\right),\\
\int_{\cY}(\phi\circ a^{t})\, d\mu_\cY
& =  \mu_{\cX}(\phi)+
O\left(e^{-\delta t} \, \|\phi\|_{C^k}\right).
\end{align*}
and thus, in combination with \eqref{eq:psist},
\begin{equation}
\label{eq:p1}
\left|\int_{\cY}\psi_{s+t}\psi_t\, d\mu_\cY\right|\ll e^{-\delta \min(s,t)}\, \|\phi\|_{C^k}^2.
\end{equation}
This integral can also be estimated in a different way. If we set $\phi_t=\phi\circ a^{t}$,
then we deduce from Corollary \ref{cor:corr} that
\begin{align*}
\int_{\cY}(\phi\circ a^{s+t})  (\phi\circ a^{t})\, d\mu_\cY
&=
\int_{\cY}(\phi_{t}\circ a^{s}) \phi_{t}\, d\mu_\cY\\
&=
\mu_{\cY}(\phi_t)\mu_{\cX}(\phi_t)+O\left(e^{-\delta s}\, \|\phi_t\|_{C^k}^2\right),
\end{align*}
and 
\begin{align*}
\mu_{\cY}(\phi\circ a^{s+t})=\mu_{\cY}(\phi_t\circ a^{s})
 = \mu_{\cX} (\phi_t)+
O\left(e^{-\delta s}\, \|\phi_t\|_{C^k}\right).
\end{align*}
It is easy to see that for some $\xi=\xi(m,n,k)>0$,
$$
\|\phi_t\|_{C^k}\ll e^{\xi t}\,\|\phi\|_{C^k}\quad \hbox{and} \quad |\mu_\cY(\phi_t)|\le \|\phi\|_{C^k},
$$
and thus
\begin{equation}
\label{eq:p2}
\left|\int_{\cY}\psi_{s+t}\psi_t\, d\mu_\cY\right|\ll e^{-\delta s} e^{\xi t}\, \|\phi\|_{C^k}^2.
\end{equation}
Let us now combine \eqref{eq:p1} and \eqref{eq:p2}: When $t\le \delta/(2\xi)\, s$, we use  \eqref{eq:p2}, and 
when $t\ge \delta/(2\xi)\, s$, we use \eqref{eq:p1}. If we set $\delta'=\min(\delta/2,\delta^2/(2\xi))>0$, then
$$
\left|\int_{\cY}\psi_{s+t}\psi_t\, d\mu_\cY\right|\ll e^{-\delta' s}\, \|\phi\|_{C^k}^2
$$
for all $s\ge 0$, whence
$$
|\Theta_N(s)|\le \frac{1}{N}\sum_{t=1}^{N-1-s} \left|\int_{\cY}\psi_{s+t}\psi_t\, d\mu_\cY\right|\ll e^{-\delta' s}\, \|\phi\|_{C^k}^2
$$
uniformly in $N$. Hence, the Dominated Convergence Theorem applied to \eqref{eq:FF_N} yields \eqref{eq:limit}.

\subsection{Proof of Theorem \ref{th:CLT_compact}}

In this subsection we shall check that the conditions of Proposition \ref{th:CLT} hold for the sequence $(F_N)$
defined in \eqref{eq:FNcpt}. First, by construction, 
It is easy to check that 
$$
\int_\cY F_T\, d\mu_\cY =0,
$$
and by \eqref{eq:limit},
$$
\int_{\cY}F_N^2\, d\mu_\cY\to \sum_{s=-\infty}^{\infty} \Theta_\infty(s)\quad\hbox{as $N\to\infty$}.
$$
Furthermore, by \eqref{eq:CumNcptzero}, 
$$
\cum_{\mu_\cY}^{(r)}(F_N) \to 0\quad \hbox{as $N\to\infty$},
$$
for every $r \geq 3$, which finishes the proof.

\section{Non-divergence estimates for Siegel transforms}\label{sec:non_div}

\subsection{Siegel transforms}

We recall that the space $\cX$ of unimodular lattices in $\bR^{m+n}$ can be identified with the 
quotient space $G/\Gamma$, where $G = \SL_{m+n}(\bR)$ and $\Gamma = \SL_{m+n}(\bZ)$,
which is endowed with the $G$-invariant probability measures $\mu_{\cX}$. We denote by $m_G$
a bi-$G$-invariant Radon measure on $G$. Given a bounded measurable
function $f:\bR^{m+n}\to\bR$ with compact support, we define its \emph{Siegel transform}
$\hat f:\cX\to \bR$ by
$$
\hat f(\Lambda):=\sum_{z\in \Lambda\backslash \{0\}} f(z)\quad \hbox{ for $\Lambda \in \cX$.}
$$
We stress that $\hat{f}$ is unbounded on $\cX$, its growth is controlled by an explicit function $\alpha$
which we now introduce. Given a lattice $\Lambda$ in $\bR^{m+n}$, we say that a subspace $V$ of $\bR^{m+n}$ is 
\emph{$\Lambda$-rational} if the intersection $V\cap \Lambda$ is a lattice in $V$. If $V$ is $\Lambda$-rational,
we denote by $d_\Lambda(V)$ the volume of $V/(V\cap \Lambda)$, and define
$$
\alpha(\Lambda):=\sup\left\{d_\Lambda(V)^{-1}:\, \hbox{$V$ is a $\Lambda$-rational subspace of $\bR^{m+n}$}\right\}.
$$	
It follows from the Mahler Compactness Criterion that $\alpha$ is a proper function on $\cX$.

\begin{proposition}[\cite{sch0}, Lem.~2]
	\label{l:alpha}
	If $f:\bR^{m+n}\to\bR$ is a bounded function with compact support, then
	$$
	|\hat f(\Lambda)|\ll_{\hbox{\rm\tiny supp}(f)} \|f\|_{C^0}\, \alpha(\Lambda)\quad\hbox{ for all $\Lambda\in \cX$}.
	$$
\end{proposition}

Using Reduction Theory, it is not hard to derive the following integrability of $\alpha$:

\begin{proposition}[\cite{emm}, Lem. 3.10]\label{p:alpha_int}
$\alpha\in L^p(\cX)$ for $1\le p<m+n$. In particular,
$$
\mu_\cX\left(\{\alpha\ge L\}\right)\ll_p L^{-p}\quad\hbox{for all $p<m+n$.}
$$
\end{proposition}

In what follows, $d\overline{z}$ denotes the volume element on $\bR^{m+n}$ which assigns volume one to the unit cube. In
our arguments below, we will make heavy use of the following two integral formulas:

\begin{proposition}[Siegel Mean Value Theorem; \cite{sie}] \label{p:siegel_mean}
If $f:\bR^{m+n}\to\bR$ is a bounded Riemann integrable function with compact support, then
$$
\int_\cX \hat f\, d\mu_\cX= \int_{\bR^{m+n}} f(\overline{z})\, d\overline{z}.
$$

\end{proposition}

\begin{proposition}[Rogers Formula; \cite{rog}, Theorem~5] \label{p:rogers}
If $F : \mathbb{R}^{m+n}\times\mathbb{R}^{m+n} \ra \bR$ is a non-negative measurable function, then
\begin{align*}
\int_{\cX} \left(\sum_{\overline{z}_1,\overline{z}_2\in (\mathbb{Z}^{m+n})^*} F(g\overline{z}_1,g\overline{z}_2)\right) d\mu_\cX(g\Gamma)=&\zeta(m+n)^{-2}\int_{\mathbb{R}^{m+n}\times\mathbb{R}^{m+n}} F(\overline{z}_1,\overline{z}_2)\,d\overline{z}_1d\overline{z}_2\\
&+\zeta(m+n)^{-1}\int_{\mathbb{R}^{m+n}} F(\overline{z},\overline{z})\,d\overline{z}\\
&+\zeta(m+n)^{-1}\int_{\mathbb{R}^{m+n}} F(\overline{z},-\overline{z})\,d\overline{z},
\end{align*}
where the sum is taken over primitive integral vectors, and $\zeta$ denotes Riemann's $\zeta$-function.
\end{proposition}

\subsection{Non-divergence estimates}

We retain the notation from Section \ref{sec:correlations}. Given 
$$
0<w_1,\ldots, w_m<n \qand  w_1+\cdots+w_m=n,
$$
we denote by $a$ the self-map on $\cX$ induced by
\begin{equation}
\label{eq_anondiv}
a=\hbox{diag}(e^{w_1},\ldots,e^{w_m},e^{-1},\ldots,e^{-1}).
\end{equation}
Our goal in this subsection is to analyze the escape of mass for the submanifolds $a^s\cY$ and 
bound the Siegel transforms $\hat f(a^s y)$ for $y\in \cY$. The following proposition will play a very
important role in our arguments.

\begin{proposition}\label{p:div1}
There exists $\kappa>0$ such that for every $L\ge 1$ and $s\ge \kappa\log L$,
$$
\mu_\cY\left(\{y\in \cY:\, \alpha(a^sy)\ge L\}\right)\ll_p L^{-p}\quad\hbox{for all $p<m+n$}.
$$
\end{proposition}

\begin{proof}
Let $\chi_L$ be the characteristic function of the subset $\{\alpha< L\}$
of $\cX$. By Mahler's Compactness Criterion, $\chi_L$ has a compact support.
We further pick a non-negative $\rho\in C_c^\infty(G)$ with $\int_G \rho\, dm_G=1$.
Let 
$$
\eta_L(x):=(\rho*\chi_L)(x)=\int_G \rho(g) \chi_L(g^{-1}x)\, dm_G(g), \quad x\in \cX.
$$
Since $\mu_\cX$ is $G$-invariant, 
$$
\int_{\cX} \eta_{L}\, d\mu_\cX=\int_{\cX} \chi_{L}\, d\mu_\cX=\mu_\cX(\{\alpha<L\}).
$$
It follows from invariance of $m_G$ that if $\cD_Z$ is a differential operator as defined as in \eqref{eq:diff}, then $\cD_Z\eta_L=(\cD_Z\rho)*\chi_L$. Hence,
$\eta_L\in C^\infty_c(\cX)$, and $\|\eta_L\|_{C^k}\ll \|\rho\|_{C^k}$.  \\

Note that there exists $c>1$ such that for every $g\in \hbox{supp}(\rho)$
and all $x\in \cX$, we have
$\alpha(g^{-1}x)\ge c^{-1}\,\alpha(x)$, and thus $\{\alpha\circ g^{-1}<L\}\subset \{\alpha<c L\}$ and 
$\eta_{L}\le \chi_{cL}$. 
This implies the lower bound
$$
\mu_\cY\left(\{y\in \cY:\, \alpha(a^sy)< cL\}\right)
=\int_\cY \chi_{cL}(a^sy)\, d\mu_\cY(y)\ge \int_\cY \eta_{L}(a^sy)\, d\mu_\cY(y).
$$
By Corollary \ref{cor:corr}, there exists $\delta > 0$ and $k \geq 1$ such that
\begin{align*}
\int_\cY \eta_{L}(a^sy)\, d\mu_\cY(y) &=
\int_{\cX} \eta_{L}\, d\mu_\cX+O\left( e^{-\delta s}\, \|\eta_{L}\|_{C^k} \right)\\
&=\mu_\cX\left(\{\alpha<L\}\right)+O\left( e^{-\delta s}\right),
\end{align*}
and by Proposition \ref{p:alpha_int},
$$
\mu_\cX\left(\{\alpha\ge L \}\right)\ll_p L^{-p} \quad\hbox{for all $p<m+n$}.
$$
Combining these bounds, we get
$$
\mu_\cY(\{y\in \cY:\, \alpha(a^sy)< cL\})\ge 
\mu_\cX(\{\alpha< L \})+O\left( e^{-\delta s}\right)
=1+O_{p}\left(L^{-p} +e^{-\delta s}\right),
$$
and thus
$$
\mu_\cY\left(\{y\in \cY:\, \alpha(a^sy)\ge c L\}\right)\ll_p L^{-p} +e^{-\delta s}.
$$
By taking $s \geq \kappa \log L$ where $\kappa = \frac{p}{\delta}$, the proof is finished.
\end{proof}

\begin{proposition}\label{prop:sup}
Let $f$ be a bounded measurable function on $\mathbb{R}^{m+n}$ 
with compact support contained in the open set $\{(x_{m+1},\ldots, x_{m+n})\ne 0\}$.
Then, with $a$ as in \eqref{eq_anondiv},
$$
\sup_{s\ge 0} \int_{\cY} |\hat f\circ a^s|\, d\mu_\cY<\infty.
$$
\end{proposition}

\begin{proof}
We note that there exist $0<\upsilon_1<\upsilon_2$ and $\vartheta>0$ such that the support of $f$ is contained in the set
\begin{equation}
\label{eq:set}
\left\{(\overline{x},\overline{y})\in\bR^{m+n}:\; \upsilon_1\le \|\overline{y}\|\le \upsilon_2,\,\; |x_i|\le \vartheta\, \|\overline{y}\|^{-w_i},\,\; i=1,\ldots,n\right\},
\end{equation}
and without loss of generality we may assume that $f$ is the characteristic function of this set. 
We recall that $\cY$ can be identified with the collection of lattices 
$$
\left\{\Lambda_u:\; u=\left(u_{ij}:i=1,\ldots m, j=1,\ldots,n\right)\in [0,1)^{m\times n}\right\}.
$$
We set $\overline{u}_i=(u_{i1},\ldots,u_{in})$.
Then by the definition of the Siegel transform,
\begin{align*}
\hat f(a^s\Lambda_u)
=\sum_{(\overline{p},\overline{q})\in \mathbb{Z}^{m+n}\backslash \{0\}}
f\left(e^{w_1 s}(p_1+\left<\overline{u}_1,\overline{q}\right>),\ldots, e^{w_m s}(p_m+\left<\overline{u}_m,\overline{q}\right>), e^{-s}\overline{q} \right).
\end{align*}
Denoting by $\chi^{(i)}_{\overline{q}}$ the characteristic function of the interval $\left[-\vartheta\, \|\overline{q}\|^{-w_i},\vartheta\, \|\overline{q}\|^{-w_i}\right]$,
we rewrite this sum as
\begin{align}\label{eq:formula}
\hat f(a^s\Lambda_u)&=\sum_{\upsilon_1 e^{s}\le \|\overline{q}\|\le \upsilon_2 e^s} \sum_{\overline{p}\in \mathbb{Z}^m}  \prod_{i=1}^m  \chi^{(i)}_{\overline{q}}(p_i+\left<\overline{u}_i,\overline{q}\right>) \\
&=\sum_{\upsilon_1 e^{s}\le \|\overline{q}\|\le \upsilon_2 e^s} \prod_{i=1}^m \left( \sum_{p_i\in \mathbb{Z}} \chi^{(i)}_{\overline{q}}(p_i+\left<\overline{u}_i,\overline{q}\right>) \right).\nonumber
\end{align}
Hence, 
\begin{align}\label{eq:integral}
\int_{\cY} (\hat f\circ a^s)\, d\mu_\cY=\sum_{\upsilon_1 e^{s}\le \|\overline{q}\|\le \upsilon_2 e^s} \prod_{i=1}^m \left( \sum_{p_i\in \mathbb{Z}} \int_{[0,1]^n} \chi^{(i)}_{\overline{q}}(p_i+\left<\overline{u}_i,\overline{q}\right>) d\overline{u}_i\right).\nonumber
\end{align}
We observe that for each $i$ and $p_i \in \bZ$, the volume of the set
$$
\left\{\overline{u}\in [0,1]^n:\, |p+\left<\overline{u},\overline{q}\right>|\le \vartheta\, \|\overline{q}\|^{-w_i}\right\}
$$
is estimated from above by
$$
 \frac{2\vartheta\,\|\overline{q}\|^{-w_i}}{\max_j |q_j|}\ll \|\overline{q}\|^{-1-w_i},
$$
and we note that the set is empty whenever $|p|> \sum_j |q_j|+\vartheta\, \|\overline{q}\|^{-w_i }$. In particular, it
is non-empty for at most $O(\|\overline{q}\|)$ choices of $p\in\bZ$. Hence, we deduce that
$$
\int_{\cY} (\hat f\circ a^s)\, d\mu_\cY\ll \sum_{\upsilon_1 e^{s}\le \|\overline{q}\|\le \upsilon_2 e^s} \prod_{i=1}^m \|\overline{q}\|^{-w_i}=\sum_{\upsilon_1 e^{s}\le \|\overline{q}\|\le \upsilon_2 e^s} \|\overline{q}\|^{-n}\ll 1,
$$
uniformly in $s$. This completes the proof.
\end{proof}

\begin{proposition}\label{prop:sup2}
	Let $f$ be a bounded measurable function on $\mathbb{R}^{m+n}$ 
	with compact support contained in the open set $\{(x_{m+1},\ldots,x_{m+n})\ne 0\}$.
	Then 
	$$
	\sup_{s\ge 0} (1+s)^{-\nu_m}\left\|\hat f\circ a^s\right\|_{L^2(\cY)}<\infty,
	$$
	where $\nu_1=1$ and $\nu_m=0$ when $m\ge 2$.
\end{proposition}

\begin{proof}
As in the proof of Proposition \ref{prop:sup}, it is sufficient to consider the case when $f$ is the characteristic function of the set \eqref{eq:set}.
Then $\hat f(a^sy)$ is given by \eqref{eq:formula}, and we get
\begin{align*}
\left\|\hat f\circ a^s\right\|_{L^2(\cY)}^2
&=\int_{\cY} \hat f( a^s y)\hat f( a^s y)\, d\mu_\cY(y)\\
&=\sum_{\upsilon_1\, e^{s}\le \|\overline{q}\|,\|\overline{\ell}\|\le \upsilon_2\, e^s} \prod_{i=1}^m
\left( \sum_{p_i,r_i\in\mathbb{Z}}
\int_{[0,1]^n}\chi^{(i)}_{\overline{q}}\left(p_i+\left<\overline{u}_i,\overline{q}\right>\right)\chi^{(i)}_{\overline{\ell}}\left(r_i+\left<\overline{u}_i,\overline{\ell}\right>\right)\, d\overline{u}_i\right).
\end{align*}
For fixed $\overline{q}, \overline{l} \in \bZ^n$, we wish to estimate
$$
I_i(\overline{q},\overline{\ell}):=
\sum_{p,r\in\mathbb{Z}}
\int_{[0,1]^n}\chi^{(i)}_{\overline{q}}\left(p+\left<\overline{u},\overline{q}\right>\right)\chi^{(i)}_{\overline{\ell}}\left(r+\left<\overline{u},\overline{\ell}\right>\right)\, d\overline{u}.
$$
First, we consider the case when $\overline{q}$ and $\overline{\ell}$
are linearly independent. Then there exist indices $j,k=1,\ldots,n$ such that 
$q_j\ell_k-q_k\ell_j\ne 0.$ Let us consider the function $\psi$ on $\bR^2$ defined by
$\psi(x_1,x_2)=\chi^{(i)}_{\overline{q}}(x_1)\chi^{(i)}_{\overline{\ell}}(x_2)$
as well as the periodized function $\bar\psi$ on $\bR^2/\bZ^2$ defined by 
$\bar\psi(x)=\sum_{z\in \bZ^2} \psi(z+x)$.
If we set
\[
\omega:=\sum_{\zeta\ne j,k} q_\zeta u_\zeta \qand \rho:=\sum_{\zeta\ne j,k} \ell_\zeta u_\zeta,
\] 
then we denote by $S$ the affine map 
$$
S:(x_1,x_2)\mapsto (q_jx_1 +q_kx_2+\omega, \ell_j x_1+\ell_k x_2+\rho),
$$
which induces an affine endomorphism of the torus $\bR^2/\bZ^2$. We note that
$$
\sum_{p,r\in\bZ}\int_{[0,1]^2}\chi^{(i)}_{\overline{q}}\left(p+\left<\overline{u},\overline{q}\right>\right)\chi^{(i)}_{\overline{\ell}}\left(r+\left<\overline{u},\overline{\ell}\right>\right)\, du_jdu_k=\int_{\bR^2/\bZ^2} \bar{\psi}(Sx)\, d\mu(x),
$$
where $\mu$ denotes the Lebesgue probability measure on the torus $\bR^2/\bZ^2$.
Since the endomorphism $S$ preserves $\mu$, we see that
$$
\int_{\bR^2/\bZ^2} \bar{\psi}(Sx)\, d\mu(x)=\int_{\bR^2/\bZ^2} \bar{\psi}(x)\, d\mu(x)
=\int_{\bR^2} \psi(x)\, dx=4\vartheta^2\, \|\overline{q}\|^{-w_i} \|\overline{\ell}\|^{-w_i}.
$$
Therefore, we conclude that in this case,
\begin{equation}
\label{eq:II1}
I_i(\overline{q},\overline{\ell})\ll \|\overline{q}\|^{-w_i} \|\overline{\ell}\|^{-w_i}.
\end{equation}

Let us now we consider the second case when $\overline{q}$ and $\overline{\ell}$ are linearly dependent. Upon re-arranging
indices if needed, we may assume that 
\begin{equation}\label{eq:q}
|q_1|=\max(|q_1|,\ldots,|q_n|,|\ell_1|,\ldots, |\ell_n|).
\end{equation}
In particular, $q_1\ne 0$, and thus $\ell_1 \neq 0$, since $\overline{q}$ and $\overline{\ell}$ are linearly dependent,
so we can define the new variables 
$$
v_1=\sum_{j=1}^n (q_j/q_1)u_j = \sum_{j=1}^n (\ell_j/\ell_1)u_j \qand v_2=u_2,\ldots, v_n=u_n,
$$
and thus
$$
I_i(\overline{q},\overline{\ell})\le J_i(\overline{q},\overline{\ell})
$$
where
$$
J_i(\overline{q},\overline{\ell}):=
\sum_{p,r\in\mathbb{Z}}
\int_{-n}^n\chi^{(i)}_{\overline{q}}\left(p+q_1 v_1\right)\,\chi^{(i)}_{\overline{\ell}}\left(r+\ell_1v_1\right)\, dv_1.
$$
We note that the last integral is non-zero only when $|p|\ll |q_1|$ and $|r|\ll |\ell_1|$.
We set $q_1=q'd$ and $\ell_1=\ell'd$ where $d=\gcd(q_1,\ell_1)$.
Then $q_1r-\ell_1 p=j d$ for some $j\in \mathbb{Z}$.
We observe that when $j$ is fixed, then the integers $p$ and $r$ satisfy the 
equation $q'r-\ell'p=j$. Since $\gcd(q',\ell')=1$,
all the solutions of this equation are of the form $p=p_0+kq'$, $r=r_0+k\ell'$
for $k\in \mathbb{Z}$. In particular, it follows that the number of such solutions
satisfying $|p|\ll |q_1|$ and $|r|\ll |\ell_1|$ is at most $O(d)$.
We write 
$$
J_i(\overline{q},\overline{\ell})=J^{(1)}_i(\overline{q},\overline{\ell})+J^{(2)}_i(\overline{q},\overline{\ell}),
$$
where the first sum is taken over those $p,r$ with $q_1r-\ell_1 p\ne 0$,
and the second sum is taken over those $p,r$ with $q_1r-\ell_1 p=0$.\\

Upon applying a linear change of variables, we obtain 
\begin{align*}
J^{(1)}_i(\overline{q},\overline{\ell})
&=\sum_{p,r:\, q_1r-\ell_1 p\ne 0}
\int_{-n+p/q_1}^{n+p/q_1}\chi^{(i)}_{\overline{q}}(q_1v_1)\,\chi^{(i)}_{\overline{\ell}}\left((q_1r-\ell_1 p)/q_1+\ell_1v_1\right)\, dv_1\\
&\ll d \sum_{j\in \mathbb{Z}\backslash\{0\}} 
\int_{-\infty}^{\infty}\chi^{(i)}_{\overline{q}}(q_1 v_1)\,\chi^{(i)}_{\overline{\ell}}\left(jd/q_1+\ell_1 v_1\right)\, dv_1.
\end{align*}
Let us consider the function
$$
\rho_i(x):=  \int_{-\infty}^{\infty} \chi^{(i)}_{\overline{q}}(q_1 v_1)\,\chi^{(i)}_{\overline{\ell}}\left(xd/q_1+\ell_1 v_1\right)\, dv_1.
$$
We note that the integrand equals the indicator function of the intersection of the intervals
$$
\left[-\vartheta\,|q_1|^{-1}\|\overline{q}\|^{-w_i},\vartheta\, |q_1|^{-1}\|\overline{q}\|^{-w_i}\right]
$$
and 
$$
\left[-xd/(q_1\ell_1)-\vartheta\,|\ell_1|^{-1}\|\overline{\ell}\|^{-w_i},-xd/(q_1\ell_1)+\vartheta\,|\ell_1|^{-1}\|\overline{\ell}\|^{-w_i}\right],
$$
and thus it follows that $\rho_i$ is non-increasing when $x\ge 0$, and non-decreasing when $x\le 0$.
This implies that 
$$
\sum_{j\in \mathbb{Z}\backslash \{0\}} \rho_i(j)
\le \int_{-\infty}^{\infty} \rho_i(x)\, dx.
$$
Since
$$
\int_{-\infty}^{\infty} \rho_i(x)\, dx=
\left(\int_{-\infty}^{\infty} \chi^{(i)}_{\overline{q}}(q_1 v_1)\,dv_1\right) \left(\int_{-\infty}^{\infty}\chi^{(i)}_{\overline{\ell}}\left(xd/q_1\right)\, dx\right)
\ll d^{-1} \|\overline{q}\|^{-w_i} \|\overline{\ell}\|^{-w_i},
$$
we conclude that
\begin{align*}
J_i^{(1)}(\overline{q},\overline{\ell})\ll d \sum_{j\in \mathbb{Z}\backslash \{0\}} \rho_i(j)
\ll \|\overline{q}\|^{-w_i}\|\overline{\ell}\|^{-w_i}.
\end{align*}
Next, we proceed with estimation of $J^{(2)}_i(\overline{q},\overline{\ell})$.
Let $c_0:=\min\{\|\overline{q}\|: \overline{q} \in \bZ^{n}\backslash \{0\}\}$.
Denoting by $N(q_1,\ell_1)$ the number of solutions $(p,r)$ of the equation
$$
q_1r-\ell_1 p= 0\quad\hbox{ with $|p|\le (c_0^{-n}\vartheta+n)|q_1|$,} 
$$
we also obtain 
\begin{align*}
	J^{(2)}_i(\overline{q},\overline{\ell})
	&=\sum_{p,r:\, q_1r-\ell_1 p= 0}
	\int_{-n+p/q_1}^{n+p/q_1}\chi^{(i)}_{\overline{q}}(q_1v_1)\,\chi^{(i)}_{\overline{\ell}}\left(\ell_1v_1\right)\, dv_1\\
	&\le N(q_1,\ell_1) \int_{-\infty}^{\infty}\chi^{(i)}_{\overline{q}}(q_1v_1)\,\chi^{(i)}_{\overline{\ell}}\left(\ell_1v_1\right)\, dv_1\\
	&\ll N(q_1,\ell_1)|q_1|^{-1}\|\overline{q}\|^{-w_i}\ll N(q_1,\ell_1)\max\left(\|\overline{q}\|,\|\overline{\ell}\|\right)^{-(1+w_i)},
\end{align*}
where we used that $q_1$ is chosen according to \eqref{eq:q}.
Combining the obtained estimates for $J_i^{(1)}(\overline{q},\overline{\ell})$
and $J^{(2)}_i(\overline{q},\overline{\ell})$, we conclude that when 
$\overline{q}$ and $\overline{\ell}$ are 
linearly dependent,
\begin{equation}
\label{eq:II2}
J_i(\overline{q},\overline{\ell})\ll \|\overline{q}\|^{-w_i}\|\overline{\ell}\|^{-w_i}+
N(q_1,\ell_1)\max\left(\|\overline{q}\|,\|\overline{\ell}\|\right)^{-(1+w_i)},
\end{equation}
where $q_1$ is chosen according to \eqref{eq:q}.\\

Now we proceed to estimate
\begin{align}\label{eq:ssum}
\left\|\hat f\circ a^s\right\|_{L^2(\cY)}^2&\ll \sum_{\upsilon_1\, e^{s}\le \|\overline{q}\|,\|\overline{\ell}\|\le \upsilon_2\, e^s}
\prod_{i=1}^m I_i(\overline{q},\overline{\ell}).
\end{align}
Using \eqref{eq:II1}, the sum in \eqref{eq:ssum} over linearly independent $\overline{q}$ and $\overline{\ell}$
can be estimated as
\begin{align*}
\ll \sum_{\upsilon_1\, e^{s}\le \|\overline{q}\|,\|\overline{\ell}\|\le \upsilon_2\, e^s}
\prod_{i=1}^m \|\overline{q}\|^{-w_i}\|\overline{\ell}\|^{-w_i}\ll
\sum_{\upsilon_1\, e^{s}\le \|\overline{q}\|,\|\overline{\ell}\|\le \upsilon_2\, e^s}
\|\overline{q}\|^{-n}\|\overline{\ell}\|^{-n}\ll 1.
\end{align*}
For a subset $I$ of $\{1,\ldots,m\}$, we set 
$w(I):=\sum_{i\in I} w_i$. Then using \eqref{eq:II2}, we deduce that
the sum in \eqref{eq:ssum} over linearly dependent $\overline{q}$ and $\overline{\ell}$
is bounded by
\begin{align}
&\ll {\sum}_{\upsilon_1\, e^{s}\le \|\overline{q}\|,\|\overline{\ell}\|\le \upsilon_2\, e^s}^{*} \sum_{I\subset \{1,\ldots,m\}} \|\overline{q}\|^{-w(I)}\|\overline{\ell}\|^{-w(I)}  N(q_1,\ell_1)^{|I^c|}\max\left(\|\overline{q}\|,\|\overline{\ell}\|\right)^{-(|I^c|+w({I^c}))} \nonumber \\
&\ll \sum_{I\subset \{1,\ldots,m\}} (e^s)^{-(n+|I^c|+w(I))} {\sum}_{\upsilon_1\, e^{s}\le \|\overline{q}\|,\|\overline{\ell}\|\le \upsilon_2\, e^s}^* N(q_1,\ell_1)^{|I^c|}.\label{eq:long}
\end{align}
The star indicates that the sum is taken over linearly dependent $\overline{q}$ and $\overline{\ell}$. \\

When $I^c=\emptyset$, then $w(I)=n$.
Since the number of $(\overline{q},\overline{\ell})$
satisfying $\upsilon_1\, e^{s}\le \|\overline{q}\|,\|\overline{\ell}\|\le \upsilon_2\, e^s$
is estimated as $O(e^{2ns})$, it is clear that the corresponding term in
the above sum is uniformly bounded. 

Now we suppose that $I^c\ne \emptyset$.
Since $\overline{q}$ and $\overline{\ell}$ are linearly dependent,
the vector $\overline{\ell}$ is uniquely determined given $\ell_1$ and $\overline{q}$,
and we obtain that for some $\upsilon_1',\upsilon_2'>0$,
\begin{align*}
{\sum}_{\upsilon_1\, e^{s}\le \|\overline{q}\|,\|\overline{\ell}\|\le \upsilon_2\, e^s}^* N(q_1,\ell_1)^{|I^c|}\ll (e^{s})^{n-1}\sum_{\upsilon_1'\, e^{s}\le |q_1|\le \upsilon'_2\, e^s} \sum_{1\le |\ell_1|\le |q_1|} N(q_1,\ell_1)^{|I^c|}.
\end{align*}

We shall use the following lemma:

\begin{lemma}\label{l:solutions}
For every $k\ge 1$,
$$	
\sum_{1\le q\le T} \sum_{1\le \ell\le q} N(q,\ell)^k\ll
T^{k+1}(\log T)^{\nu_k}	
$$
where $\nu_1=1$ and $\nu_k=0$ when $k\ge 2$.
\end{lemma}

\begin{proof}
We observe that the sum of $N(q,\ell)^k$ over $1\le \ell\le q$ is equal 
to the number of solutions $(p_1,\ldots,p_k,r_1,\ldots, r_k,\ell)$
of the system of equations 
\begin{equation}
\label{eq:eq1}
q r_1-\ell p_1= 0,\ldots, q r_k-\ell p_k= 0
\end{equation}
satisfying
$$
|p_1|,\ldots,|p_k|\le (c_0^{-n}\vartheta+n)q\quad\hbox{and}\quad
1\le \ell\le q.
$$
We order these solutions according to $d:=\gcd(q,\ell)$. 
Let $q=q'd$ and $\ell=\ell'd$. Then $q'$ and $\ell'$ are coprime, and 
the system \eqref{eq:eq1} is equivalent to
\begin{equation}
\label{eq:eq2}
q' r_1-\ell' p_1= 0,\ldots, q' r_k-\ell' p_k= 0.
\end{equation}
Because of coprimality, each $p_i$ have to be divisible by $q'$,
so that the number of such $p_i$'s is at most $O(q/q')=O(d)$.
We note that given $d$ the number of possible choices for $\ell$ is at most $q/d$,
and $(p_1,\ldots,p_k,\ell)$ uniquely determine $(r_1,\ldots,r_k)$.
Hence, the number of solutions of \eqref{eq:eq1} is estimated by
$$
\ll \sum_{d|q} (q/d)d^k=q\sigma_{k-1}(q),
$$
where $\sigma_{k-1}(q)=\sum_{d|q} d^{k-1}$.
Writing $q=q'd$, we conclude that
\begin{align*}
\sum_{1\le q\le T} \sum_{1\le \ell\le q} N(q,\ell)^k &\ll
\sum_{1\le q\le T} q\sigma_{k-1}(q)\le T\sum_{1\le q'\le T} \sum_{q=1}^{\lfloor T/q'\rfloor} d^{k-1}\\
&\ll T^{k+1}\sum_{1\le q'\le T} (q')^{-k}\ll
T^{k+1}(\log T)^{\nu_k}.
\end{align*}
This proves the lemma.
\end{proof}

A simple modification of this argument also gives that
$$	
\sum_{1\le |q|\le T} \sum_{1\le |\ell|\le |q|} N(q,\ell)^k\ll
T^{k+1}(\log T)^{\nu_k}.
$$
Hence, it follows that 
$$
{\sum}_{\upsilon_1\, e^{s}\le \|\overline{q}\|,\|\overline{\ell}\|\le \upsilon_2\, e^s}^* N(q_1,\ell_1)^{|I^c|}\ll (e^s)^{n+|I^c|}(1+s)^{\nu(I)},
$$ 
where $\nu(I)=1$ when $|I^c|=1$ and $\nu(I)=0$ otherwise.
Thus, the sum \eqref{eq:long} is estimated as
$$
\ll 1+ \sum_{I\subsetneq \{1,\ldots,m\}} e^{-s w(I)} (1+s)^{\nu(I)}.
$$
The terms in this sum are uniformly bounded unless $I=\emptyset$ and $|I^c|=1$,
namely, when $m>1$. When $m=1$, we obtain the bound $O(1+s)$.
This proves the proposition.
\end{proof}
	
\subsection{Truncated Siegel transform}\label{sec:sieg_tran}

The Siegel transform of a compactly supported function is typically unbounded on $\cX$; to deal with this complication,
it is natural to approximate $\hat f$ by compactly supported functions on $\cX$,
the so called \emph{truncated Siegel transforms}, which we shall denote by $\hat f^{(L)}$.
They will be constructed using a smooth cut-off function $\eta_L$, which will be defined in the
following lemma.
 
\begin{lemma}\label{eq:eta}
For every $c>1$, there exists a family $(\eta_L)$ in $C_c^\infty(\cX)$ 
satisfying:
\begin{align*}
0\le \eta_L \le 1,\quad\eta_L=1 \;\; \hbox{on $\{\alpha\le c^{-1}\, L\}$},\quad
\eta_L=0 \;\; \hbox{on $\{\alpha> c\,L\}$},\quad \|\eta_L\|_{C^k}\ll 1.
\end{align*}
\end{lemma} 

\begin{proof}
Let $\chi_L$ denote the indicator function of the subset $\{\alpha\le L\} \subset \cX$, and pick
a non-negative $\phi\in C_c^\infty(G)$ with $\int_G \phi\, dm_G=1$ and with support in a 
sufficiently small neighbourhood of identity in $G$ to ensure that for all $g\in \hbox{supp}(\phi)$ 
and $x\in \cX$,
$$
c^{-1}\, \alpha(x)\le \alpha(g^{\pm 1}x)\le c\, \alpha(x).
$$
We now define $\eta_L$ as 
$$
\eta_L(x):=(\phi*\chi_L)(x)=\int_G \phi(g)\chi_L(g^{-1}x)\, dm_G(g).
$$
Since $\phi\ge 0$ and $\int_G \phi\, dm_G=1$, it is clear that $0\le \eta_L \le 1$. If $\alpha(x)\le c^{-1}\, L$, then 
for $g \in\hbox{supp}(\phi)$, we have $\alpha(g^{-1}x) \le L$, so that $\eta_L(x)=\int_G \phi\, dm_G=1$. If $\alpha(x)>c\, L$, 
then for $g\in \hbox{supp}(\phi)$, we have $\alpha(g^{-1}x) > L$, so that $\eta_L(x)=0$. \\

To prove the last property, we observe that it follows from invariance of $m_G$
that for a differential operator $\cD_Z$ as in \eqref{eq:diff},  we have $\cD_Z\eta_L=(\cD_Z\phi)*\chi_L$. Therefore, $\hbox{supp}(\cD_Z\eta_L)\subset \{\alpha\le c\,L\}$ and $\|\cD_Z\eta_L\|_{C^0}\le \|\cD_Z\phi\|_{L^1(G)}$, whence
$\|\eta_L\|_{C^k}\ll 1$.
\end{proof}

For a bounded function $f:\bR^{m+n}\to\bR$ with compact support, 
we define the \emph{truncated Siegel transform} of $f$ as
$$
\hat f^{(L)}:=\hat f\cdot \eta_L.
$$
We record some basic properties of this transform that will be used later in the proofs.

\begin{lemma}\label{l:siegel_trans}
For $f\in C_c^\infty(\bR^{m+n})$, the truncated Siegel transform $\hat f^{(L)}$ is
in $C_c^\infty(\cX)$, and it satisfies
\begin{align}
& \left\| \hat f^{(L)}\right\|_{L^p(\cX)}\le \|\hat f\|_{L^p(\cX)}  \ll_{\hbox{\rm\tiny supp}(f),p} \|f\|_{C^0} \quad\hbox{for all $p<m+n$},
\label{eq:ss0}\\
&\left\|\hat f^{(L)}\right\|_{C^0}\ll_{\hbox{\rm\tiny supp}(f)} L\, \|f\|_{C^0}, \label{eq:ss1}\\
&\left\|\hat f^{(L)}\right\|_{C^k}\ll_{\hbox{\rm\tiny supp}(f)} L\,  \|f\|_{C^k}, \label{eq:ss2} \\
& \left\|\hat f-\hat f^{(L)}\right\|_{L^1(\cX)}\ll_{\hbox{\rm\tiny supp}(f),\tau} L^{-\tau}\, \|f\|_{C^0}\quad\hbox{ for all $\tau<m+n-1$,} \label{eq:ss4} \\
& \left\|\hat f-\hat f^{(L)}\right\|_{L^2(\cX)}\ll_{\hbox{\rm\tiny supp}(f),\tau} L^{-(\tau-1)/2}\, \|f\|_{C^0}\quad\hbox{ for all $\tau<m+n-1$.} \label{eq:ss5}
\end{align}
Moreover, the implied constants are uniform when $\hbox{\rm supp}(f)$ is contained in a fixed compact set.
\end{lemma}

\begin{proof}
It follows from Proposition \ref{l:alpha} that
$$
\left|\hat f^{(L)}\right|\ll_{\hbox{\rm\tiny supp}(f)} \|f\|_{C^0} 
\,\alpha\, \eta_L .
$$
Since $0\le \eta_L\le 1$, \eqref{eq:ss0} follows from Proposition \ref{p:alpha_int}, and the upper bound in \eqref{eq:ss1} holds 
since $\hbox{supp}(\eta_L)\subset \{\alpha\le cL\}$. \\

We observe that for a differential operator $\cD_Z$ as in \eqref{eq:diff}, we have $\cD_Z(\hat f)=
\widehat{\cD_Z f}$. Hence, we deduce from Proposition \ref{l:alpha} that 
$$
\left|\cD_Z(\hat f)\right| \ll_{\hbox{\rm\tiny supp}(f)} \|f\|_{C^k}\, \alpha.
$$
Since $\hbox{supp}(\eta_L)\subset \{\alpha\le cL\}$ and $\|\eta_L\|_{C^k}\ll 1$,
we deduce that 
$$
\left\|\hat f^{(L)}\right\|_{C^k} \ll_{\hbox{\rm\tiny supp}(f)} L\,  \|f\|_{C^k}
$$
proving \eqref{eq:ss2}. \\

To prove \eqref{eq:ss4}, we observe that since $0\le \eta_L\le 1$ and 
$\eta_L=1$ on $\{\alpha< c^{-1}L\}$, it follows from Proposition \ref{l:alpha}
that
$$
\left\|\hat f-\hat f^{(L)}\right\|_{L^1(\cX)}=\int_{\cX} |\hat f|\cdot |1-\eta_L|\, d\mu_\cX \ll_{\hbox{\rm\tiny supp}(f)} \|f\|_{C^0}\int_{\{\alpha \ge c^{-1}L\}} \alpha\, d\mu_\cX.
$$
Hence, applying the H\"older inequality with $1\le p<m+n$ and $q=(1-1/p)^{-1}$,
we deduce that
$$
\left\|\hat f-\hat f^{(L)}\right\|_{L^1(\cX)}
\ll_{\hbox{\rm\tiny supp}(f)}
\|f\|_{C^0}\, \|\alpha\|_p\, \mu_\cX\left(\{\alpha \ge c^{-1}L\}\right)^{1/q}.
$$
Now it follows from Proposition \ref{p:alpha_int} that 
$$
\left\|\hat f-\hat f^{(L)}\right\|_{L^1(\cX)}
\ll_{\hbox{\rm\tiny supp}(f),p} \|f\|_{C^0}\,L^{-(p-1)},
$$
which proves \eqref{eq:ss4}. The proof of \eqref{eq:ss5} is similar,
and we omit the details.
\end{proof}

\section{CLT for smooth Siegel transforms}\label{sec:CLT_siegel_smooth}

Assume that $f\in C^\infty_c(\mathbb{R}^{m+n})$ satisfies $f\ge 0$ and $\hbox{supp}(f)\subset \{(x_{m+1},\ldots,x_{m+n})\ne 0\}$.
We shall in this section analyze the asymptotic behavior of the averages
$$
F_N(y):=\frac{1}{\sqrt{N}} \sum_{s=0}^{N-1} \left(\hat f(a^s y)-\mu_\cY(\hat f\circ a^s)\right)\quad \hbox{ with $y\in \cY,$}
$$
and prove the following result:

\begin{theorem}\label{th:CLT_smooth_Siegel}
	If $m\ge 2$ and $f$ is as above,  then the variance
	\begin{align*}
	\sigma_f^2&:=\sum_{s=-\infty}^\infty \left( \int_{\cX} (\hat f\circ a^s)\hat f \, d\mu_\cX
		-\mu_\cX(\hat f)^2\right)\\
		&=\zeta(m+n)^{-1} \sum_{s=-\infty}^\infty\sum_{p,q\ge 1} \left(\int_{\mathbb{R}^{m+n}} f(pa^sz)f(qz)\,dz+\int_{\mathbb{R}^{m+n}} f(pa^sz)f(-qz)\,dz\right).
	\end{align*}
	is finite, and for every $\xi\in \mathbb{R}$, 
	\begin{align}\label{eq:distr}
	\mu_\cY\left(\{y\in \cY:\, F_N(y)<\xi\}\right)\to
	\hbox{\rm Norm}_{\sigma_f}(\xi)
	\end{align}
	as $N\to \infty$. 
\end{theorem}

The proof of Theorem \ref{th:CLT_smooth_Siegel} follows the same plan as the proof of Theorem \ref{th:CLT_compact}, but we need to develop an additional approximation argument which involves truncations of the Siegel transform
$\hat f$. This can potentially change the behavior of the averages $F_N$,
so we will have to take into account possible escape of mass for 
the sequences of submanifolds $a^s\cY$ in $\cX$. 

\medskip

Throughout the proof, 
we shall frequently make use of the basic observation that 
if we approximate $F_N$ by $\tilde F_N$ in such a way that $\|F_N-\tilde F_N\|_{L^1(\cY)}\to 0$, then $F_N$ and $\tilde F_N$ will
have the same convergence in distribution. 

\medskip

Let
\begin{equation}
\label{eq:F_N_t}
\tilde F_N:=\frac{1}{\sqrt{N}} \sum_{s=M}^{N-1} \left(\hat f\circ a^s-\mu_\cY(\hat f\circ a^s)\right)
\end{equation}
for some $M=M(N)\to \infty$ that will be chosen later.
We observe that
\begin{equation}
\label{eq:tilde}
\|F_N-\tilde{F}_N\|_{L^1(\cY)}\le
\frac{1}{\sqrt{N}} \sum_{s=0}^{M-1} \left\|\hat f\circ a^s-\mu_\cY(\hat f\circ a^s)\right\|_{L^1(\cY)}\le \frac{2M}{\sqrt{N}}
\sup_{s\ge 0} \int_\cY |\hat f\circ a^s|\, d\mu_\cY.
\end{equation}
and thus, by Proposition \ref{prop:sup} we see that 
$$
\|F_N-\tilde{F}_N\|_{L^1(\cY)}\to 0\quad\hbox{as $N\to \infty$ }.
$$
provided 
\begin{equation}
\label{eq:cond1}
M=o(N^{1/2}).
\end{equation}
It particular, it follows that if \eqref{eq:distr} holds for $\tilde F_N$, then it also holds for $F_N$,
we shall prove the former. In order to simplify notation, let us drop the tilde, and assume
from now on that $F_N$ is given by \eqref{eq:F_N_t}. \\


Given a sequence $L = L(N)$, which shall be chosen later, we consider the average 
$$
F^{(L)}_N:=\frac{1}{\sqrt{N}} \sum_{s=M}^{N-1} \left(\hat f^{(L)}\circ a^s-\mu_\cY(\hat f^{(L)}\circ a^s)\right)
$$
defined for the truncated Siegel transforms $\hat f^{(L)}$ introduced in Section \ref{sec:sieg_tran}. We have 
\begin{align*}
\left\| F_N-F^{(L)}_N\right\|_{L^1(\cY)}
&\le \frac{1}{\sqrt{N}} \sum_{s=M}^{N-1} \left\|\left(\hat f\circ a^s -\hat f^{(L)}\circ a^s\right)-\mu_\cY\left(\hat f\circ a^s-\hat f^{(L)}\circ a^s\right)\right\|_{L^1(\cY)}\\
&\le \frac{2}{\sqrt{N}} \sum_{s=M}^{N-1} \left\|\hat f\circ a^s -\hat f^{(L)}\circ a^s\right\|_{L^1(\cY)}.
\end{align*}
We recall that $\hat f^{(L)}=\hat f\cdot \eta_L$, $0\le \eta_L\le 1$, and $\eta_L(x)=1$
when $\alpha(x)\le c^{-1}\,L$, so that
\begin{align*}
\left\|\hat f\circ a^s -\hat f^{(L)}\circ a^s\right\|_{L^1(\cY)}=
\left\| (\hat f\circ a^s) (1-\eta_L\circ a^s)\right\|_{L^1(\cY)} 
\le \int_{\{\alpha(a^sy)\ge c^{-1}\, L\}} |\hat f(a^sy)|\, d\mu_\cY(y).
\end{align*}
Hence, by the Cauchy-Schwarz inequality,
$$
\left\|\hat f\circ a^s -\hat f^{(L)}\circ a^s\right\|_{L^1(\cY)}
\le  \left\| \hat f\circ a^s \right\|_{L^2(\cY)}\,
\mu_\cY\left(\left\{y\in \cY:\, \alpha(a^sy)\ge c^{-1}L \right\}\right)^{1/2}.
$$
Let us now additionally assume that 
\begin{equation}
\label{eq:cond2}
M\gg \log L,
\end{equation}
so that the assumption of Proposition \ref{p:div1} is satisfied when $s\ge M$. This implies that
$$
\mu_\cY\left(\left\{y\in \cY:\, \alpha(a^sy)\ge c^{-1}L \right\}\right)\ll_p L^{-p}\quad \hbox{for all $p<m+n$}.
$$
We also recall that by Proposition \ref{prop:sup2} when $m\ge 2$, 
$$
\sup_{s\ge 0} \left\| \hat f\circ a^s \right\|_{L^2(\cY)}<\infty,
$$
whence, for $s\ge M$,
$$
\left\|\hat f\circ a^s -\hat f^{(L)}\circ a^s\right\|_{L^1(\cY)}\ll_p 
L^{-p/2},
$$
and thus
\begin{align}\label{eq:FNNN}
\left\| F_N-F^{(L)}_N\right\|_{L^1(\cY)}\ll_p N^{1/2} L^{-p/2}\quad \hbox{for all $p<m+n$.}
\end{align}
If we now choose $L=L(N)\to \infty$ so that 
\begin{equation}
\label{eq:cond3}
N=o\left(L^p\right)\quad \hbox{for some $p<m+n$,}
\end{equation}
then it follows that
$$
\left\| F_N-F^{(L)}_N\right\|_{L^1(\cY)}\to 0\quad \hbox{as $N\to\infty$}.
$$
Hence, if we can show Theorem \ref{th:CLT_smooth_Siegel} for the averages $F^{(L)}_N$ with the parameter constraints
above, then it would also hold for $F_N$. In order to prove CLT for $(F^{(L)}_N)$, we follow the route of Proposition 
\ref{th:CLT} and estimate cumulants and $L^2$-norms of the sequence.

\subsection{Estimating cumulants}\label{sec:cum_smooth_siegel}

We set 
$$
\psi^{(L)}_s(y):=\hat f^{(L)}(a^sy)-\mu_\cY(\hat f^{(L)}\circ a^s).
$$
Our aim is to estimate 
\begin{equation}
\label{eq:ccc}
\hbox{Cum}_{\mu_\cY}^{(r)}\left(F_N^{(L)}\right)=\frac{1}{N^{r/2}} \sum_{s_1,\ldots,s_r=M}^{N-1}
\hbox{Cum}_{\mu_\cY}^{(r)}\left(\psi^{(L)}_{s_1},\ldots,\psi^{(L)}_{s_r}\right)
\end{equation}
when $r\ge 3$.
The argument proceeds as in Section \ref{sec:cumm_comp}, but 
we have to refine the previous estimates to take into account 
the dependence on the parameters $L$ and $M$.
Using the notation from Section \ref{sec:cumm_comp}, we have the decomposition
\begin{equation}\label{eq:decomp0}
\{M,\ldots,N-1\}^r =\Omega(\beta_{r+1};M,N) \cup \Big( \bigcup_{j=0}^{r} \bigcup_{|\cQ| \geq 2} \Omega_{\cQ}(\alpha_j,\beta_{j+1};M,N) \Big),
\end{equation}
where 
\begin{align*}
\Omega(\beta_{r+1};M,N) &:=\{M,\ldots,N-1\}^r \cap \Delta(\beta_{r+1}),\\
\Omega_{\cQ}(\alpha_j,\beta_{j+1};M,N) &:=\{M,\ldots,N-1\}^r \cap \Delta_{\cQ}(\alpha_j,\beta_{j+1}).
\end{align*}
We decompose the sum into the sums over $\Omega(\beta_{r+1};M,N)$ and
$\Omega_{\cQ}(\alpha_j,\beta_{j+1};M,N)$. Let us choose $M$ so that 
\begin{equation}
\label{eq:M0}
M> \beta_{r+1}.
\end{equation}
Then $\Omega(\beta_{r+1};M,N)=\emptyset$, and does not contribute to our estimates.

\subsubsection{\textbf{\underline{Case 1:} Summing over $(s_1,\ldots,s_r)\in \Omega_{\cQ}(\alpha_j,\beta_{j+1};M,N)$ with $\cQ=\{\{0\},\{1,\ldots,r\}\}$}.}
In this case, we shall show that
\begin{equation}
\label{eq:cc2_0}
\hbox{Cum}_{\mu_\cY}^{(r)}\left(\psi^{(L)}_{s_1},\ldots,\psi^{(L)}_{s_r}\right)\approx
\hbox{Cum}_{\mu_\cX}^{(r)}\left(\phi^{(L)}\circ a^{s_1},\ldots,\phi^{(L)}\circ a^{s_r}\right)
\end{equation}
where $\phi^{(L)}:= \hat f^{(L)}-\mu_\cX(\hat f^{(L)})$.
This reduces to estimating the integrals
\begin{align}
&\int_{\cY} \left(\prod_{i\in I} \psi^{(L)}_{s_i}\right)\,d\mu_\cY \label{eq:psi}\\
=&
\sum_{J\subset I} (-1)^{|I\backslash J|} 
\left(\int_{\cY} \left(\prod_{i\in J} \hat f^{(L)}\circ a^{s_i}\right)\,d\mu_\cY\right)
\prod_{i\in I\backslash J} \left(\int_{\cY} (\hat f^{(L)}\circ a^{s_i})\,d\mu_\cY\right).\nonumber
\end{align}
If $(s_1,\ldots,s_r)\in \Omega_\cQ(\alpha_j,\beta_{j+1};N)$, and thus
$$
|s_{i_1}-s_{i_2}|\le \alpha_j \quad\hbox{and}\quad s_{i_1}\ge \beta_{j+1}\quad\quad\hbox{for all $1\le i_1,i_2\le r,$}
$$
it follows from Corollary \ref{cor:corr} with $r=1$ that there exists $\delta>0$ such that
\begin{equation}\label{eq:psi1}
\int_{\cY} (\hat f^{(L)}\circ a^{s_i})\,d\mu_\cY
=\mu_{\cX}\left(\hat f^{(L)}\right)+O\left(e^{-\delta \beta_{j+1} }\,\left\|\hat f^{(L)}\right\|_{C^k}\right).
\end{equation}
For a fixed $J \subset I$, we define
$$
\Phi^{(L)}:=\prod_{i\in J} \hat f^{(L)}\circ a^{s_i-s_1},
$$
and note that for some $\xi=\xi(m,n,k)>0$, we have
$$
\left\|\Phi^{(L)}\right\|_{C^k}\ll \prod_{i\in J} \left\|\hat f^{(L)}\circ a^{s_i-s_1}\right\|_{C^k} \ll e^{|J|\xi\, \alpha_j}\,\left\|\hat f^{(L)}\right\|_{C^k}^{|J|}.
$$
If we again apply Corollary \ref{cor:corr} to the function $\Phi^{(L)}$, we obtain 
\begin{align}
\int_{\cY} \left(\prod_{i\in J} \hat f^{(L)}\circ a^{s_i}\right)\,d\mu_\cY
&=\int_{\cY} (\Phi^{(L)}\circ a^{s_1})\,d\mu_\cY \label{eq:PPsi}\\
&=\int_{\cX}\Phi^{(L)}\, d\mu_\cX+O\left(e^{-\delta \beta_{j+1} }\,\left\|\Phi^{(L)}\right\|_{C^k}\right) \nonumber\\
&=\int_{\cX}\left(\prod_{i\in J} \hat f^{(L)}\circ a^{s_i}\right)\, d\mu_\cX+O\left(e^{-\delta \beta_{j+1} } e^{r\xi\, \alpha_j}\,\left\|\hat f^{(L)}\right\|_{C^k}^{|J|}\right), \nonumber
\end{align}
where we used that $\mu_\cX$ is invariant under the transformation $a$.
Let us now choose the exponents $\alpha_j$ and $\beta_{j+1}$ so that
$\delta \beta_{j+1}-r\xi\alpha_j>0$. 
Combining \eqref{eq:psi}, \eqref{eq:psi1} and \eqref{eq:PPsi},
we deduce that 
\begin{align}
\int_{\cY} \left(\prod_{i\in I} \psi^{(L)}_{s_i}\right)\,d\mu_\cY \label{eq:prod}
=&
\sum_{J\subset I} (-1)^{|I\backslash J|} 
\left(\int_{\cX} \left(\prod_{i\in J} \hat f^{(L)}\circ a^{s_i}\right)\,d\mu_\cX\right)
\mu_{\cX}\left(\hat f^{(L)}\right)^{|I\backslash J|}\\
&+O\left(e^{-\delta \beta_{j+1} } e^{r\xi\, \alpha_j}\,\left\|\hat f^{(L)}\right\|_{C^k}^{|I|}\right) \nonumber\\
=& \int_{\cX} \prod_{i\in I} \left(\hat f^{(L)}\circ a^{s_i}-\mu_\cX(\hat f^{(L)})\right)\,d\mu_\cX
+O\left(e^{-(\delta \beta_{j+1}-r\xi\alpha_j )}\,\left\|\hat f^{(L)}\right\|_{C^k}^{|I|}\right), \nonumber
\end{align}
and thus, for any partition $\cP$,
$$
\prod_{I\in \cP}\int_{\cY} \left(\prod_{i\in I} \psi^{(L)}_{s_i}\right)\,d\mu_\cY
=
\prod_{I\in\cP}\int_{\cX} \left(\prod_{i\in I} \phi^{(L)}\circ a^{s_i}\right)\,d\mu_\cX
+O\left(e^{-(\delta \beta_{j+1}-r\xi\alpha_j )}\,\left\|\hat f^{(L)}\right\|_{C^k}^{r}\right),
$$
and consequently,
\begin{align}\label{eq:cumm0}
\hbox{Cum}_{\mu_\cY}^{(r)}\left(\psi^{(L)}_{s_1},\ldots,\psi^{(L)}_{s_r}\right)=&
\;\hbox{Cum}_{\mu_\cX}^{(r)}\left(\phi^{(L)}\circ a^{s_1},\ldots,\phi^{(L)}\circ a^{s_r}\right)\\
&+O\left(e^{-(\delta \beta_{j+1}-r\xi\, \alpha_j)}\,\left\|\hat f^{(L)}\right\|_{C^k}^{r}\right)\nonumber
\end{align}
whenever $(s_1,\ldots,s_r)\in \Omega_\cQ(\alpha_j,\beta_{j+1};M,N)$ with $\cQ=\{\{0\},\{1,\ldots,r\}\}$,
from which \eqref{eq:cc2_0} follows. \\

We now claim that 
\begin{equation}
\label{eq:cum0_bound}
\left|\hbox{Cum}_{\mu_\cX}^{(r)}\left(\phi^{(L)}\circ a^{s_1},\ldots,\phi^{(L)}\circ a^{s_r}\right)\right| \ll_{f}  \left\|\hat f^{(L)}\right\|_{C^0}^{(r-(m+n-1))^+} \left\|\hat f^{(L)} \right\|_{L^{m+n-1}(\cX)}^{\min(r,m+n-1)},
\end{equation}
where we use the notation $x^+=\max(x,0)$. The implied constant in \eqref{eq:cum0_bound} and below in the proof
depend only on $\hbox{supp}(f)$. By the definition of the cumulant, to prove  \eqref{eq:cum0_bound}, it suffices to show 
that for every $z\ge 1$ and indices $i_1,\ldots,i_z$,
\begin{equation}
\label{eq:cum0_bound_0}
\int_{\cX} \left|\left(\phi^{(L)}\circ a^{s_{i_1}}\right)\cdots\left(\phi^{(L)}\circ a^{s_{i_z}}\right)\right|\, d\mu_\cX
\ll_f \left\|\hat f^{(L)}\right\|_{C^0}^{(z-(m+n-1))^+}
\left\|\hat f^{(L)} \right\|_{L^{m+n-1}(\cX)}^{\min(z,m+n-1)}.
\end{equation}
Using 
the generalized H\"older inequality, we deduce that when $z\le m+n-1$,
\begin{align*}
\int_{\cX} \left|\left(\phi^{(L)}\circ a^{s_{i_1}}\right)\cdots\left(\phi^{(L)}\circ a^{s_{i_z}}\right)\right|\, d\mu_\cX&\le 
\left\|\phi^{(L)}\circ a^{s_{i_1}}\right\|_{L^{m+n-1}(\cX)}\cdots\left\|\phi^{(L)}\circ a^{s_{i_z}}\right\|_{L^{m+n-1}(\cX)}\\
&\ll \left\|\hat f^{(L)}\right\|_{L^{m+n-1}(\cX)}^z.
\end{align*}
Also when $z>m+n-1$, 
\begin{align*}
&\int_{\cX} \left|\left(\phi^{(L)}\circ a^{s_{i_1}}\right)\cdots\left(\phi^{(L)}\circ a^{s_{i_z}}\right)\right|\, d\mu_\cX\\
\le& \left\|\phi^{(L)}\right\|_{C^0}^{z-(m+n-1)}
\int_{\cX} \left|\left(\phi^{(L)}\circ a^{s_{i_1}}\right)\cdots\left(\phi^{(L)}\circ a^{s_{i_{m+n-1}}}\right)\right|\, d\mu_\cX\\
\ll_f & \left\|\hat f^{(L)}\right\|_{C^0}^{z-(m+n-1)} \left\|\hat f^{(L)}\right\|_{L^{m+n-1}(\cX)}^{m+n-1}.
\end{align*}
This implies \eqref{eq:cum0_bound_0} and \eqref{eq:cum0_bound}. \\

Finally, we recall that if $(s_1,\ldots,s_r)\in \Omega_{\cQ}(\alpha_j,\beta_{j+1};M,N)$
with $\cQ=\{\{0\},\{1,\ldots,r\}\}$, then we have $|s_{i_1}-s_{i_2}|\le \alpha_j$
for all $i_1\ne i_2$, and thus
\begin{equation}
\label{eq:card}
|\Omega_{\cQ}(\alpha_j,\beta_{j+1};M,N)|\ll N\alpha_j^{r-1}.
\end{equation}
Combining \eqref{eq:cum0_bound} and \eqref{eq:card}, we conclude that
\begin{align*}
&\frac{1}{N^{r/2}} \sum_{(s_1,\ldots,s_r)\in \Omega_{\cQ}(\alpha_j,\beta_{j+1};M,N)}\left|\hbox{Cum}_{\mu_\cX}^{(r)}\left( \phi^{(L)}\circ a^{s_1},\ldots,\phi^{(L)}\circ a^{s_r}\right)\right|\\
\ll_f&\; N^{1-r/2} \alpha_j^{r-1} 
\left\|\hat f^{(L)}\right\|_{C^0}^{(r-(m+n-1))^+} \left\|\hat f^{(L)} \right\|_{L^{m+n-1}(\cX)}^{\min(r,m+n-1)}.
\end{align*}
Hence, it follows from \eqref{eq:cumm0} that
\begin{align*}
&\frac{1}{N^{r/2}} \sum_{(s_1,\ldots,s_r)\in \Omega_{\cQ}(\alpha_j,\beta_{j+1};M,N)}\hbox{Cum}_{\mu_\cY}^{(r)}\left( \psi_{s_1}^{(L)},\ldots,\psi_{s_r}^{(L)}\right)\\
\ll_f &\; N^{r/2}\, e^{-(\delta\beta_{j+1}- r\alpha_j\xi)}\, \left\|\hat f^{(L)}\right\|_{C^k}^r\\
&+N^{1-r/2} \alpha_j^{r-1} 
\left\|\hat f^{(L)}\right\|_{C^0}^{(r-(m+n-1))^+} \left\|\hat f^{(L)} \right\|_{L^{m+n-1}(\cX)}^{\min(r,m+n-1)}
\end{align*}
and using Lemma \ref{l:siegel_trans}, we deduce that 
\begin{align*}
&\frac{1}{N^{r/2}} \sum_{(s_1,\ldots,s_r)\in \Omega_{\cQ}(\alpha_j,\beta_{j+1};M,N)}\hbox{Cum}_{\mu_\cY}^{(r)}\left( \psi_{s_1}^{(L)},\ldots,\psi_{s_r}^{(L)}\right)\\
\ll_f &\; N^{r/2}\, e^{-(\delta\beta_{j+1}- r\alpha_j\xi)}\, L^r \|f\|_{C^k} 
+N^{1-r/2} \alpha_j^{r-1} 
L^{(r-(m+n-1))^+}\|f\|_{C^0}^r.
\end{align*}

\subsubsection{\textbf{\underline{Case 2:} Summing over $(s_1,\ldots,s_r)\in \Omega_{\cQ}(\alpha_j,\beta_{j+1};M,N)$ with $|\cQ|\ge 2$ and 
	$\cQ\ne \{\{0\},\{1,\ldots,r\}\}$}.}
	
In this case, the estimate \eqref{eq:cum3} gives
for $(s_1,\ldots, s_r)\in \Omega_{\cQ}(\alpha_j,\beta_{j+1};M,N)$,
$$
\left|\hbox{Cum}_{\mu_\cY}^{(r)}(\psi^{(L)}_{s_1},\ldots,\psi^{(L)}_{s_r})\right|
\ll e^{-(\delta \beta_{j+1} - r\xi\alpha_j)} \,\left\|\hat f^{(L)}\right\|_{C^k}^{r},
$$
and
\begin{align*}
\sum_{(s_1,\ldots, s_r)\in \Omega_{\cQ}(\alpha_j,\beta_{j+1};M,N)} \left|\hbox{Cum}_{\mu_\cY}^{(r)}(\psi^{(L)}_{s_1},\ldots,\psi^{(L)}_{s_r})\right|
&\ll N^{r/2}e^{-(\delta \beta_{j+1} - r\xi\alpha_j)} \,\left\|\hat f^{(L)}\right\|_{C^k}^{r}\\
&\ll N^{r/2}e^{-(\delta \beta_{j+1} - r\xi\alpha_j)} L^r\,\|f\|_{C^k}^r,
\end{align*}
where we used Lemma \ref{l:siegel_trans}. \\

Finally, we combine the established bounds to estimate $\hbox{Cum}_{\mu_\cY}^{(r)}(F_N^{(L)})$.
We choose the parameters $\alpha_j$  and $\beta_j$ as in \eqref{eq:beta}.
Then $\beta_{r+1}\ll_r \gamma$.
In particular, we may choose 
\begin{equation}
\label{eq:cond4_0}
M\gg_r \gamma
\end{equation}
to guarantee that \eqref{eq:M0} is satisfied.
With these choices of $\alpha_j$ and $\beta_j$,
we obtain the estimate
\begin{align}
\label{eq:cumm_smooth}
\left|\hbox{Cum}_{\mu_\cY}^{(r)}(F^{(L)}_N)\right|
\ll_f N^{r/2}  e^{-\delta \gamma} L^{r} \|f\|_{C^{k}}^r
+N^{1-r/2}   \gamma^{r-1} L^{(r-(m+n-1))^+} \|f\|_{C^0}^r.
\end{align}
We observe that since $m\ge 2$, 
$$
\frac{(r-(m+n-1))^+}{m+n}<r/2-1
\quad\hbox{for all $r\ge 3$,}
$$
Hence, we can choose $q>1/(m+n)$ such that
$$
q(r-(m+n-1))^+<r/2-1
\quad\hbox{for all $r\ge 3$.}
$$
Then we select
$$
L=N^q,
$$
so that, in particular, the condition \eqref{eq:cond3} is satisfied.
Now \eqref{eq:cumm_smooth} can be rewritten as 
\begin{equation}
\label{eq:cumm_smooth_0}
\left|\hbox{Cum}_{\mu_\cY}^{(r)}(F^{(L)}_N)\right|
\ll_f N^{r/2+rq} e^{-\delta\gamma}\, \|f\|_{C^{k}}^r
+N^{q(r-(m+n-1))^+-(r/2-1)} \gamma^{r-1}\, \|f\|_{C^0}^r.
\end{equation}
Choosing $\gamma$ of the form
$$
\gamma=c_r(\log N)
$$
with sufficiently large $c_r>0$, 
we conclude that
$$
\hbox{Cum}_{\mu_\cY}^{(r)}(F^{(L)}_N)\to 0\quad\hbox{ as $N\to\infty$}
$$
for all $r\ge 3$.

\subsection{Estimating variances} 
\label{sec:var_smooth_siegel}

We now turn to the analysis of  the variances of the average $F_N^{(L)}$ which are given by
\begin{align*}
	\left\|F^{(L)}_N\right\|_{L^2(\cY)}^2 =&\frac{1}{N}\sum_{s_1=M}^{N-1}\sum_{s_2=M}^{N-1} \int_{\cY}
	\psi^{(L)}_{s_1}\psi^{(L)}_{s_2}\, d\mu_\cY.
\end{align*}
We proceed as in Section \ref{sec:var_compact} taking into account 
dependence on parameters $M$ and $L$.
We observe that this expression is symmetric with respect to $s_1$ and $s_2$,
writing $s_1=s+t$ and $s_2=t$ with $0\le s\le N-M-1$ and $M\le t \le N-s-1$,
we obtain that
\begin{align}\label{eq:FF_N_L}
\left\|F^{(L)}_N\right\|_{L^2(\cY)}^2 
=& \Theta^{(L)}_N(0)
+2\sum_{s=1}^{N-M-1}\Theta^{(L)}_N(s),
\end{align}
where
$$
\Theta^{(L)}_N(s):=\frac{1}{N}\sum_{t=M}^{N-1-s} \int_{\cY}\psi^{(L)}_{s+t}\psi^{(L)}_t\, d\mu_\cY.
$$
We have 
\begin{align*}
\int_{\cY}\psi^{(L)}_{s+t}\psi^{(L)}_t\, d\mu_\cY
=
\int_{\cY} (\hat f^{(L)}\circ a^{s+t})(\hat f^{(L)}\circ a^t)\, d\mu_\cY -
\mu_\cY(\hat f^{(L)}\circ a^{s+t})\mu_\cY(\hat f^{(L)}\circ a^t).
\end{align*}
To estimate $\Theta^{(L)}_N(s)$, we introduce an additional parameter $K=K(N)\to \infty$ such that $K\le M$ (to be specified later) and consider separately the cases when $s< K$ and when $s\ge  K$. \\

First, we consider the case when $s\ge K$. 
By Corollary \ref{cor:corr}, we have 
\begin{align}\label{eq:second0}
\int_{\cY} (\hat f^{(L)}\circ a^{s+t})(\hat f^{(L)}\circ a^t)\, d\mu_\cY
&=\mu_{\cX}(\hat f^{(L)})^2
 +O\left(e^{-\delta \min(s,t)} \left\|\hat f^{(L)}\right\|_{C^k}^2 \right).
\end{align}
Also, By Corollary \ref{cor:corr},
\begin{align}\label{eq:second}
	\int_{\cY}  (\hat f^{(L)}\circ a^t)\, d\mu_\cY
	&=\mu_{\cX} (\hat f^{(L)}) +O\left( e^{-\delta t} \left\|\hat f^{(L)}\right\|_{C^k} \right).
\end{align}
Hence, combining \eqref{eq:second0} and \eqref{eq:second}, we deduce that
\begin{align*}
	\int_{\cY}\psi^{(L)}_{s+t}\psi^{(L)}_t\, d\mu_\cY
	=O \left(e^{-\delta \min(s,t)} \left\|\hat f^{(L)}\right\|_{C^k}^2  \right).
\end{align*}
Since
\begin{align*}
\sum_{s=K}^{N-M-1}\left(\sum_{t=M}^{N-1-s}e^{-\delta \min(s,t)}\right)
&\le \sum_{s=K}^{N-1}\sum_{t=M}^{N-1} (e^{-\delta s}+e^{-\delta t})\ll N e^{-\delta K},
\end{align*}
we conclude that
\begin{equation}
\label{eq:Theta1}
\sum_{s=K}^{N-M-1}\Theta^{(L)}_N(s) \ll \, e^{-\delta K}\, \left\|\hat f^{(L)}\right\|_{C^k}^2
\ll_f \, e^{-\delta K} L^{2}\, \|f\|_{C^k}^2,
\end{equation}
where we used Lemma \ref{l:siegel_trans}.
The implied constants here and below in the proof depend only on $\supp(f)$. \\

Let us now consider the case $s< K$. We observe that Corollary \ref{cor:corr} (for $r = 1$)
applied to the function $\phi_s^{(L)}:= (\hat f^{(L)}\circ a^s)\hat f^{(L)}$ yields, 
\begin{align*}
\int_{\cY} (\hat f^{(L)}\circ a^{s+t})(\hat f^{(L)}\circ a^t)\, d\mu_\cY
&=\int_{\cY} (\phi_s^{(L)}\circ a^t) \, d\mu_\cY\\
&=\int_{\cX} \phi_s^{(L)}\, d\mu_\cX+O\left(e^{-\delta t}\, \left\|\phi_s^{(L)}\right\|_{C^k} \right).
\end{align*}
Furthermore, for some $\xi=\xi(m,n,k)>0$, we have
$$
\left\|\phi_s^{(L)}\right\|_{C^k}\ll \left\|\hat f^{(L)}\circ a^s\right\|_{C^k}\, \left\|\hat f^{(L)}\right\|_{C^k}\ll e^{\xi s}\,
\left\|\hat f^{(L)}\right\|_{C^k}^2. 
$$
Therefore, we deduce that
\begin{align*}
\int_{\cY} (\hat f^{(L)}\circ a^{s+t})(\hat f^{(L)}\circ a^t)\, d\mu_\cY
=\int_{\cX} (\hat f^{(L)}\circ a^s)\hat f^{(L)} \,d\mu_\cX +O \left(e^{-\delta t} e^{\xi s}\, \left\|\hat f^{(L)}\right\|_{C^k}^2  \right).
\end{align*}
Combining this estimate with \eqref{eq:second}, we conclude that
\begin{align*}
\int_{\cY}\psi^{(L)}_{s+t}\psi^{(L)}_t\, d\mu_\cY
=\int_{\cX} (\hat f^{(L)}\circ a^s)\hat f^{(L)}\, d\mu_\cX -\mu_\cX(\hat f^{(L)})^2
+ O \left(e^{-\delta t} e^{\xi s}\, \left\|\hat f^{(L)}\right\|_{C^k}^2  \right).
\end{align*}
Hence, setting 
$$
\Theta^{(L)}_\infty(s):=\int_{\cX} (\hat f^{(L)}\circ a^s)\hat f^{(L)}\, d\mu_\cX -\mu_\cX(\hat f^{(L)})^2,
$$
we obtain that
\begin{align*}
\Theta^{(L)}_N(s)=&\, \frac{N-M-s}{N} 
\Theta^{(L)}_\infty(s)+O \left(N^{-1} e^{-\delta M} e^{\xi s}\, \left\|\hat f^{(L)}\right\|_{C^k}^2  \right)\\
=& \, \Theta^{(L)}_\infty(s)
+O \left(N^{-1}(M+s)\left\|\hat f^{(L)}\right\|^2_{L^2(\cX)}+ N^{-1} e^{-\delta M} e^{\xi s}\, \left\|\hat f^{(L)}\right\|_{C^k}^2  \right).
\end{align*}
Therefore, using Lemma \ref{l:siegel_trans}, we deduce that
\begin{align}
\label{eq:Theta2}
\Theta^{(L)}_N(0)+2\sum_{s=1}^{K-1}\Theta^{(L)}_N(s) =& \,
\Theta^{(L)}_\infty(0)+2\sum_{s=1}^{K-1}\Theta^{(L)}_\infty (s) \\
&+O_f \left (N^{-1}(M+K)K\,\|f\|^2_{C^0}+ N^{-1} e^{-\delta M}  e^{\xi K} L^{2}\,\|f\|_{C^k}^2\right).\nonumber
\end{align}
Combining \eqref{eq:Theta1} and \eqref{eq:Theta2}, we conclude that 
\begin{align}\label{eq:starr}
&\Theta^{(L)}_N(0)+2\sum_{s=1}^{M-N-1}\Theta^{(L)}_N(s)\\
=& \,
\Theta^{(L)}_\infty(0)+2\sum_{s=1}^{K-1}\Theta^{(L)}_\infty (s) \nonumber\\
& + O_f \left (N^{-1}(M+K)K \, \|f\|^2_{C^0}+ (N^{-1} e^{-\delta M}  e^{\xi K}+e^{-\delta K}) L^{2}\,\|f\|_{C^k}^2 \right).\nonumber
\end{align}
We choose the parameters $K=K(N)$, $M=M(N)$, and $L=L(N)$ so that
\begin{align}
e^{-\delta K} L^{2} \to 0, \label{eq:cond5}\\
N^{-1} e^{-\delta M}  e^{\xi K} L^{2} \to 0, \label{eq:cond6}\\
N^{-1}(M+K)K \to 0 \label{eq:cond7}
\end{align}
as $N\to \infty$. Then
$$
\left\|F^{(L)}_N\right\|_{L^2(\cY)}^2 =
\Theta^{(L)}_\infty(0)+2\sum_{s=1}^{K-1}\Theta^{(L)}_\infty (s)+o(1).
$$

Next, we shall show that with a suitable choice of parameters,
\begin{equation}
\label{eq:fff}
\left\|F^{(L)}_N\right\|_{L^2(\cY)}^2 =
\Theta_\infty(0)+2\sum_{s=1}^{K-1}\Theta_\infty (s) +o(1),
\end{equation}
where
$$
\Theta_\infty(s):=\int_{\cX} (\hat f\circ a^s)\hat f\, d\mu_\cX -\mu_\cX(\hat f)^2.
$$
We recall that by Lemma \ref{l:siegel_trans}, for all $\tau<m+n-1$,
\begin{align}\label{eq:l2}
\left\|\hat f- \hat f^{(L)}\right\|_{L^1(\cX)}\ll_{f,\tau} L^{-\tau}\,  \|f\|_{C^0}\quad\hbox{and}\quad
\left\|\hat f- \hat f^{(L)}\right\|_{L^2(\cX)}\ll_{f,\tau} L^{-(\tau-1)/2}\, \|f\|_{C^0} ,
\end{align}
where the implied constant depends only on $\supp(f)$.
It follows from these estimates that
\begin{align*}
\mu_\cX(\hat f^{(L)})&=\mu_\cX(\hat f)+O_{f,\tau }(L^{-\tau}\,  \|f\|_{C^0}),\\
\int_{\cX} (\hat f^{(L)}\circ a^s)\hat f^{(L)}\, d\mu_\cX &=
\int_{\cX} (\hat f\circ a^s)\hat f\, d\mu_\cX+O_{f,\tau}(L^{-(\tau-1)/2}\,  \|f\|^2_{C^0}),
\end{align*}
so that
\begin{equation}
\label{eq:Theta3}
\Theta^{(L)}_\infty (s)=\Theta_\infty (s)+O_{f,\tau}\left(L^{-(\tau-1)/2} \,  \|f\|^2_{C^0}\right).
\end{equation}
We choose the parameters $K=K(N)\to\infty$ and $L=L(N)\to\infty$ so that 
\begin{equation}
\label{eq:cond8}
K L^{-(\tau-1)/2} \to 0 \quad\hbox{for some $\tau<m+n-1$}.
\end{equation}
Then \eqref{eq:fff} follows.
We conclude that 
\begin{equation}
\label{eq:ffff}
\left\|F^{(L)}_N\right\|_{L^2(\cY)}^2 \to
\Theta_\infty(0)+2\sum_{s=1}^{\infty}\Theta_\infty (s)
\end{equation}
as $N\to\infty$.

\vspace{0.2cm}

Finally, we compute $\Theta_\infty (s)$ using by the Rogers formula (Proposition \ref{p:rogers})
applied to the function 
$$
F_s(\overline{z}_1,\overline{z}_2):=\sum_{p,q\ge 1} f(p a^s\overline{z}_1)f(q \overline{z}_2),\quad (\overline{z}_1,\overline{z}_2)\in\bR^{m+n}\times \bR^{m+n}.
$$
Since 
$$
\int_{\cX} (\hat f\circ a^s)\hat f\, d\mu_\cX=
\int_{\cX} \left(\sum_{\overline{z}_1,\overline{z}_2\in (\mathbb{Z}^{m+n})^*} F_s(g\overline{z}_1,g\overline{z}_2)\right) d\mu_\cX(g\Gamma),
$$
we deduce that
\begin{align*}
\int_{\cX} (\hat f\circ a^s)\hat f\, d\mu_\cX
=&\left(\int_{\mathbb{R}^{m+n}} f(\overline{z})\, d\overline{z} \right)^2\\
&+
\zeta(m+n)^{-1}\sum_{p,q\ge 1} \left(\int_{\mathbb{R}^{m+n}} f(pa^s\overline{z})f(q\overline{z})\,d\overline{z}
+\int_{\mathbb{R}^{m+n}} f(pa^s\overline{z})f(-q\overline{z})\,d\overline{z}\right).
\end{align*}
Since by the Siegel Mean Value Theorem (Proposition \ref{p:siegel_mean}),
$$
\int_{\cX}\hat f \, d\mu_\cX= \int_{\mathbb{R}^{m+n}} f(\overline{z})\, d\overline{z},
$$
we conclude that
$$
\Theta_\infty(s)=\zeta(m+n)^{-1}\sum_{p,q\ge 1} \left(\int_{\mathbb{R}^{m+n}} f(pa^s\overline{z})f(q\overline{z})\,dz+\int_{\mathbb{R}^{m+n}} f(pa^s\overline{z})f(-q\overline{z})\,d\overline{z}\right).
$$

Finally, we show that the sum in \eqref{eq:ffff} is finite.
We represent points $\overline{z}\in \bR^{m+n}$ as $\overline{z}=(\overline{x},\overline{y})$ with $\overline{x}\in \bR^{m}$ and $\overline{y}\in \bR^{n}$.
Since $f$ is bounded, and its compact support is contained in $\{\overline{y}\ne 0\}$, we may assume without loss of generality that $f$
is the characteristic function of the set
$$
\{(\overline{x},\overline{y})\in\bR^{m+n}:\,\upsilon_1\le \|\overline{y}\|\le \upsilon_2,\quad |x_i|\le \vartheta\, \|\overline{y}\|^{-w_i},\; i=1,\ldots,m\}
$$
with $0<\upsilon_1<\upsilon_2$ and $\vartheta>0$.
Let 
$$
\Omega_s(p):=\left\{(\overline{x},\overline{y})\in\bR^{m+n}:\,\frac{\upsilon_1\, e^s}{p}\le \|\overline{y}\|\le \frac{\upsilon_2\, e^s}{p},\quad p^{1+w_i}|x_i|\|\overline{y}\|^{w_i}\le \vartheta,\; i=1,\ldots,m \right\}.
$$
Then 
$$
\int_{\mathbb{R}^{m+n}} f(pa^s\overline{z})f(\pm q\overline{z})\,d\overline{z}
=\hbox{vol}\left(\Omega_s(p)\cap \Omega_0(q)\right).
$$
Setting 
$$
I(u):=\left\{\overline{y}\in\bR^n:\, \upsilon_1\, u\le \|\overline{y}\|\le \upsilon_2 \,u\right\},
$$
we obtain that
\begin{align*}
\hbox{vol}\left(\Omega_s(p)\cap \Omega_0(q)\right)
&=\int_{I(e^sp^{-1})\cap I(q^{-1})} \left(\prod_{i=1}^m 2\vartheta \max(p,q)^{-1-w_i}\|\overline{y}\|^{-w_i}\right)\, d\overline{y}\\
&\ll \max(p,q)^{-m-n}\int_{I(e^sp^{-1})\cap I(q^{-1})} \|\overline{y}\|^{-n}\, d\overline{y}.
\end{align*}  
We note that $I(e^sp^{-1})\cap I(q^{-1})=\emptyset$ unless
$(\upsilon_1\upsilon_2^{-1})e^sq\le p\le (\upsilon_2\upsilon_1^{-1})e^sq$, and also that
$$
\int_{I(q^{-1})} \|\overline{y}\|^{-n}\, d\overline{y}\ll \int_{\upsilon_1/q}^{\upsilon_2/q} r^{-1}\, dr\ll 1.
$$
Hence, it follows
that 
\begin{align*}
\sum_{s=1}^\infty \Theta_\infty(s)\ll
\sum_{s,q=1}^\infty \sum_{(\upsilon_1\upsilon_2^{-1})e^sq\le p\le (\upsilon_2\upsilon_1^{-1})e^sq}
\max(p,q)^{-(m+n)}
\ll  \sum_{s,q=1}^\infty (e^sq)^{-(m+n-1)}<\infty,
\end{align*}
because $m+n\ge 3$.

\subsection{Proof of Theorem \ref{th:CLT_smooth_Siegel}}
As we already remarked above, it is sufficient to show that the sequence of averages
$F_N^{(L)}$ converges in distribution to the normal law.
To verify this, we use the Method of Cumulants (Proposition \ref{th:CLT}).
It is easy to see that 
$$
\int_{\cY} F_N^{(L)}\, d\mu_\cY=0.
$$
Moreover, with a suitable choice of parameters, we have shown in Section \ref{sec:cum_smooth_siegel} that for $r\ge 3$,
$$
\cum_{\mu_\cY}^{(r)}\left(F^{(L)}_N\right)\to 0\quad\hbox{as $N\to\infty$},
$$
and in Section \ref{sec:var_smooth_siegel} that 
$$
\left\|F^{(L)}_N\right\|_{L^2(\cY)}^2\to \sigma_f^2<\infty \quad\hbox{as $N\to\infty$}.
$$
Hence, Proposition \ref{th:CLT} applies, and it remains to verify
that we can choose our parameters that satisfy the stated assumptions.
We recall that 
$$
L=N^q\quad\hbox{ and }\quad \gamma=c_r(\log N).
$$
The parameters $M=M(N)\ge K=K(N)$ need to satisfy the seven conditions
\eqref{eq:cond1},
\eqref{eq:cond2},
\eqref{eq:cond4_0},
\eqref{eq:cond5},
\eqref{eq:cond6},
\eqref{eq:cond7},
\eqref{eq:cond8}.
We take
$$
K(N)=c_1(\log N)
$$
with sufficiently large $c_1>0$
so that \eqref{eq:cond5} is satisfied.
Then taking 
$$
M(N)=(\log N)(\log\log N),
$$ 
we arrange that \eqref{eq:cond2}, \eqref{eq:cond4_0}, and \eqref{eq:cond6} hold
for all $N\ge N_0(r)$. 
We note that the constant $c_r$ and the implicit constant in 
\eqref{eq:cond4_0} depends on $r$, and the $(\log\log N)$-factor here is added to guarantee that the parameter $M$ is independent of $r$.
Finally, the conditions \eqref{eq:cond1}, \eqref{eq:cond7}, \eqref{eq:cond8}
are immediate from our choices.

\section{CLT for counting functions and the proof of Theorem \ref{th:CLT_forms}}
\label{sec:CLT_counting}

We recall from Section \ref{sec:outline} that 
$$
\Delta_T(u)=|\Lambda_u\cap \Omega_T|+O(1),
$$
where $\Lambda_u$ is defined in \eqref{eq:lll} and the domains $\Omega_T$ are defined in \eqref{eq:omega_t}.
We shall decompose this domain into smaller pieces using the linear map $a=\hbox{diag}(e^{w_1},\ldots,e^{w_m},e^{-1},\ldots,e^{-1})$,
we note that for any integer $N \geq 1$, 
$$
\Omega_{e^N}=\bigsqcup_{s=0}^{N-1} a^s\Omega_e,
$$
and thus
$$
|\Lambda_u\cap \Omega_{e^N}|=\sum_{s=0}^{N-1} \hat\chi(a^s\Lambda_u),
$$
where $\chi$ denotes the characteristic function of the set $\Omega_e$. Hence the proof of Theorem \ref{th:CLT_forms} reduces 
to analyzing sums of the form $\sum_{s=0}^{N-1} \hat\chi(a^sy)$ with $y\in \cY$. For this purpose, we define 
\begin{equation}
\label{eq:F_N_0}
F_N:=\frac{1}{\sqrt{N}} \sum_{s=0}^{N-1} \left(\hat \chi\circ a^s-\mu_\cY(\hat \chi\circ a^s)\right).
\end{equation}
Our main result in this section now reads as follows.

\begin{theorem}\label{th:CLT_char_Siegel}
If $m\ge 2$,  then for every $\xi\in \mathbb{R}$, 
	$$
	\mu_\cY\left(\{y\in \cY:\, F_N(y)<\xi\}\right)\to
	\hbox{\rm Norm}_{\sigma}(\xi)
	$$	
	as $N\to \infty$, where
$$
\sigma^2:=
2^{m+1}\left(\prod_{i=1}^m \vartheta_i\right)\left(\int_{S^{n-1}} \|\overline{z}\|^{-n}\, d\overline{z}\right)\left(\frac{2\zeta(m+n-1)}{\zeta(m+n)}-1\right).
$$
\end{theorem}

We approximate $\chi$ by a family of non-negative functions $f_\eps\in C^\infty_c(\bR^{m+n})$ whose supports are contained
in an $\eps$-neighbourhood of the set $\Omega_e$, and
$$
\chi \le f_\eps\le 1,\quad\|f_\eps-\chi\|_{L^1(\bR^{m+n})}\ll \eps,\quad \|f_\eps-\chi\|_{L^2(\bR^{m+n})}\ll \eps^{1/2},\quad
\|f_\eps\|_{C^k}\ll \eps^{-k}.
$$
This approximation allows us to  construct smooth approximations of the Siegel transform $\hat \chi$ in 
the following sense.

\begin{proposition}\label{p:esp_approx}
For every $s\ge 0$,
$$
\int_{\cY} \left|\hat f_\eps\circ a^s-\hat\chi\circ a^s\right|\, d\mu_\cY \ll \eps+e^{-s}.
$$
\end{proposition}

\begin{proof}
We observe that there exists $\vartheta_i(\eps)>\vartheta_i$ such that $\vartheta_i(\eps)=\vartheta_i+O(\eps)$ and 
$f_\eps\le \chi_\eps$, where $\chi_\eps$ denotes 
the characteristic function of the set 
$$
\left\{(\overline{x},\overline{y})\in \bR^{m+n}:\, 1-\eps\le \|\overline{y}\|\le e+\eps,\; |x_i|<\vartheta_i(\eps)\,\|\overline{y}\|^{-w_i}\;\hbox{ for $i=1,\ldots,m$}\right\}.
$$
Then it follows that
$$
|\hat f_\eps(a^s\Lambda)-\hat\chi(a^s\Lambda)|=\sum_{v\in \Lambda\backslash \{0\}} \left(f_\eps(a^s v)-\chi(a^sv)\right)\le 
\sum_{v\in \Lambda\backslash \{0\}} \left(\chi_\eps(a^s v)-\chi(a^sv)\right).
$$
It is clear that $\chi_\eps -\chi$ is bounded by the sum $\chi_{1,\eps}+\chi_{2,\eps}+\chi_{3,\eps}$ of the characteristic functions of the sets
\begin{align*}
&\left\{(\overline{x},\overline{y})\in \bR^{m+n}:\, 1-\eps\le \|\overline{y}\|\le 1,\; |x_i|<\vartheta_i(\eps)\,\|\overline{y}\|^{-w_i}\;\hbox{ for $i=1,\ldots,m$}\right\},\\
&\left\{(\overline{x},\overline{y})\in \bR^{m+n}:\, e\le \|\overline{y}\|\le e+\eps,\; |x_i|<\vartheta_i(\eps)\,\|\overline{y}\|^{-w_i}\;\hbox{ for $i=1,\ldots,m$}\right\},\\
&\left\{(\overline{x},\overline{y})\in \bR^{m+n}:\, 1\le \|\overline{y}\|\le e,\; |x_i|<\vartheta_i(\eps)\,\|\overline{y}\|^{-w_i}\;\hbox{ for all $i$},\; |x_j|\ge c\, \|\overline{y}\|^{-w_j}\hbox{ for some $j$} \right\}
\end{align*}
respectively. In particular, we obtain that  
$$
\hat f_\eps(a^s\Lambda)-\hat\chi(a^s\Lambda)\le \hat\chi_{1,\eps}(a^s\Lambda)+\hat\chi_{2,\eps}(a^s\Lambda)+\hat\chi_{3,\eps}(a^s\Lambda).
$$
Hence, it remains to show that for $j=1,2,3$,
$$
\int_{\cY} (\hat \chi_{j,\eps}\circ a^s)\, d\mu_\cY\ll \eps+e^{-s}
$$
As in \eqref{eq:formula}, we compute that
\begin{align}\label{eq:integral}
\int_{\cY} (\hat \chi_{1,\eps}\circ a^s)\, d\mu_\cY=\sum_{(1-\eps)e^s\le \|\overline{q}\|\le e^s} \prod_{i=1}^m \left( \sum_{p_i\in \mathbb{Z}} \int_{[0,1]^n} \chi_{\vartheta_i(\eps)\|\overline{q}\|^{-w_i}}\left(p_i+\left<\overline{u}_i,\overline{q}\right>\right) d\overline{u}_i\right),
\end{align}
where $\chi_{\theta}$ denotes the characteristic function of the interval
$\left[-\theta, \theta\right]$.
We observe that
$$
\int_{[0,1]^n} \chi_{\vartheta_i(\eps)\|\overline{q}\|^{-w_i}}(p_i+\left<\overline{u}_i,\overline{q}\right>) d\overline{u}_i \ll ({\max}_k |q_k|)^{-1} \|\overline{q}\|^{-w_i}\ll \|\overline{q}\|^{-1-w_i},
$$
and moreover this integral is non-zero only when $|p_i|=O(\|\overline{q}\|)$.
Hence,
$$
\sum_{p_i\in \mathbb{Z}} \int_{[0,1]^n} \chi_{\vartheta_i(\eps)\|\overline{q}\|^{-w_i} }\left(p_i+\left<\overline{u}_i,\overline{q}\right>\right) d\overline{u}_i
\ll \|\overline{q}\|^{-w_i},
$$
and 
\begin{align*}
\int_{\cY} (\hat \chi_{1,\eps}\circ a^s)\, d\mu_\cY
&\ll\sum_{(1-\eps)e^s\le \|\overline{q}\|\le e^s}\prod_{i=1}^m \|\overline{q}\|^{-w_i}\\
&\ll e^{-ns}\, \left|\left\{\overline{q}\in\bZ^n:\, (1-\eps)e^s\le \|\overline{q}\|\le e^s\right\}\right|.
\end{align*}
The number of integral points in the region $\{(1-\eps)e^s\le \|\overline{y}\|\le e^s\}$
can be estimated in terms of its volume. Namely,
there exist $r>0$ (depending only on the norm) such that 
$$
\left|\left\{\overline{q}\in\bZ^n:\, (1-\eps)e^s\le \|\overline{q}\|\le e^s\right\}\right|
\ll \left|\left\{\overline{y}\in\bR^n:\, (1-\eps)e^s-r\le \|\overline{y}\|\le e^s+r\right\}\right|.
$$
Hence,
\begin{align*}
\int_{\cY} (\hat \chi_{1,\eps}\circ a^s)\, d\mu_\cY
&\ll e^{-ns}\left( (e^s+r)^n-((1-\eps)e^s-r)^n\right)\ll 
(1+re^{-s})^n-(1-\eps-re^{-s})^n\\
&\ll \eps+e^{-s}.
\end{align*}
The integral for $\hat \chi_{2,\eps}\circ a^s$ can be estimated similarly. \\

The integral over $\hat \chi_{3,\eps}\circ a^s$ as in \eqref{eq:integral}
 can be written as a sum of the products of the integral 
\begin{align*}
&\int_{[0,1]^n}\left( \chi_{\vartheta_j(\eps)\|\overline{q}\|^{-w_j}}(p_j+\left<\overline{u}_j,\overline{q}\right>)- \chi_{\vartheta_j\|\overline{q}\|^{-w_j}}(p_j+\left<\overline{u}_j,\overline{q}\right>)\right)\,d\overline{u}_j \\
\le&\, 2({\max}_k |q_k|)^{-1} (\vartheta_j(\eps)-\vartheta_j) \|\overline{q}\|^{-w_j}
\ll \eps \|\overline{q}\|^{-1-w_j},
\end{align*}
and the integrals 
\begin{align*}
\int_{[0,1]^n} \chi_{\vartheta_i(\eps)\|\overline{q}\|^{-w_i}}(p_i+\left<\overline{u}_i,\overline{q}\right>)\,d\overline{u}_i &\le 2\vartheta_i(\eps) ({\max}_k |q_k|)^{-1} \|\overline{q}\|^{-w_i}
\ll \|\overline{q}\|^{-1-w_i}
\end{align*}
with $i\ne j$.
We observe that these integrals are non-zero only when $|p_j|=O(\|\overline{q}\|)$ and $|p_i|=O(\|\overline{q}\|)$.
Hence, we conclude that
\begin{align*}
\int_{\cY} (\hat \chi_{3,\eps}\circ a^s)\, d\mu_\cY
&\ll \sum_{e^s\le \|\overline{q}\|\le e^{s+1}} \eps\prod_{i=1}^m \|\overline{q}\|^{-w_i}
= \eps \sum_{e^s\le \|\overline{q}\|\le e^{s+1}} \|\overline{q}\|^{-n}\ll \eps,
\end{align*}
which completes the proof of the proposition.
\end{proof}

Now we start with the proof of Theorem \ref{th:CLT_char_Siegel}.
As in Section \ref{sec:CLT_siegel_smooth},
we  modify $F_N$ and consider instead 
\begin{equation}
\label{eq:F_N_t_2}
\tilde F_N:=\frac{1}{\sqrt{N}} \sum_{s=M}^{N-1} \left(\hat \chi\circ a^s-\mu_\cY(\hat \chi\circ a^s)\right)
\end{equation}
for a parameter $M=M(N)\to \infty$ that will be chosen later.
As in \eqref{eq:tilde} we obtain that
$$
\|F_N-\tilde{F}_N\|_{L^1(\cY)}\to 0\quad \hbox{as $N\to\infty$}
$$
provided that 
\begin{equation}
\label{eq:cc1}
M=o(N^{1/2}).
\end{equation}
Hence, if we can prove the CLT for $(\tilde{F}_N)$, then the CLT for $(F_N)$ would follow. 
From now on, to simplify notations, we assume that $F_N$ is given by \eqref{eq:F_N_t_2}. \\

Our next step is to exploit the approximation $\chi\approx f_\eps$,
 so we introduce 
$$
F^{(\eps)}_N:=\frac{1}{\sqrt{N}} \sum_{s=M}^{N-1} \left(\hat f_\eps\circ a^s-\mu_\cY(\hat f_\eps\circ a^s)\right),
$$
where the parameter $\eps=\eps(N)\to 0$ will be specified later.
We observe that it follows from  Proposition \ref{p:esp_approx} that
$$
\left\|F^{(\eps)}_N-F_N\right\|_{L^1(\cY)}\le 
\frac{2}{\sqrt{N}} \sum_{s=M}^{N-1} \left\|\hat f_\eps\circ a^s-\hat \chi\circ a^s\right\|_{L^1(\cY)}\ll N^{1/2}(\eps+e^{-M}).
$$
We choose $\eps=\eps(N)$ and $M=M(N)$ so that
\begin{equation}
\label{eq:cc2}
N^{1/2}\eps\to 0\quad\hbox{and}\quad  N^{1/2} e^{-M}\to 0.
\end{equation}
Then
$$
\left\|F^{(\eps)}_N-F_N\right\|_{L^1(\cY)}\to 0\quad \hbox{as $N\to\infty$}.
$$
Hence, it remains to prove convergence in distribution for the sequence $F^{(\eps)}_N$. \\

We observe that the sequence $F_N^{(\eps)}$ fits into the framework of Section  \ref{sec:CLT_siegel_smooth}. However, we need to take into account the dependence on the new parameter $\eps$ and refine the previous estimates. 
It will be important for our argument that the supports of the functions $f_\eps$ are uniformly bounded, $\|f_\eps\|_{C^0}\ll 1$,
and $\|f_\eps\|_{C^k}\ll \eps^{-k}$. \\

As in Section  \ref{sec:CLT_siegel_smooth},
we consider the truncation 
$$
F^{(\eps,L)}_N:=\frac{1}{\sqrt{N}} \sum_{s=M}^{N-1} \left(\hat f^{(L)}_\eps\circ a^s-\mu_\cY(\hat f^{(L)}_\eps\circ a^s)\right).
$$
defined for a parameter $L=L(N)\to\infty$. We assume that
\begin{equation}
\label{eq:cc3}
M\gg \log L,
\end{equation}
so that Proposition \ref{prop:sup2} applies when $s\ge M$.
Since the family of functions $f_\eps$ is majorized by a fixed bounded function
with compact support, Proposition \ref{prop:sup2} implies that when $m\ge 2$ ,
$$
\left\|\hat f_\eps\circ a^s\right\|_{L^2(\cY)}\ll 1\quad \hbox{ for all $s\ge 0$,}
$$
 uniformly on $\eps$. Hence, the bound \eqref{eq:FNNN} can be proved exactly as before,
and we obtain 
$$
\left\| F_N^{(\eps)} - F_N^{(\eps, L)}\right\|_{L^1(\cY)}\ll_p N^{1/2}L^{-p/2}\quad\hbox{for all $p<m+n$}.
$$
We choose the parameter $L$ as before so that
\begin{equation}
\label{eq:cc4}
N=o\left( L^p\right)\quad \hbox{for some $p<m+n$}
\end{equation}
to guarantee that 
$$
\left\| F_N^{(\eps)} - F_N^{(\eps, L)}\right\|_{L^1(\cY)}\to 0\quad\hbox{as $N\to\infty$}.
$$
Now it remains to show that the family $F_N^{(\eps, L)}$ satisfies the CLT
with a suitable choice of parameters $M,L,\eps$.
As in Section \ref{sec:CLT_siegel_smooth} we will show that for $r\ge 3$,
\begin{align}\label{eq:cccc}
\cum_{\mu_\cY}^{(r)}\left(F^{(\eps,L)}_N\right)&\to 0\quad \hbox{as $N\to\infty$},
\end{align}
and 
\begin{align}\label{eq:vvv}
\left\|F^{(\eps,L)}_N\right\|_{L^2(\cY)}^2\to \sigma\quad \hbox{as $N\to\infty$}
\end{align}
with an explicit $\sigma\in (0,\infty)$. \\

Under the condition
\begin{equation}
\label{eq:cc5}
M\gg \gamma,
\end{equation}
the estimate \eqref{eq:cumm_smooth} gives the bound
\begin{align*}
\left|\hbox{Cum}_{\mu_\cY}^{(r)}(F^{(\eps,L)}_N)\right| &\ll 
N^{r/2} e^{-\delta\gamma} L^{r}\|f_\eps\|_{C^k}^r+ N^{1-r/2}\gamma^{r-1} L^{(r-(m+n-1))^+} \|f_\eps\|_{C^0}^r\\
&\ll 
N^{r/2} e^{-\delta\gamma} L^{r}\eps^{-rk}+ N^{1-r/2}\gamma^{r-1} L^{(r-(m+n-1))^+}.
\end{align*}
We note that the implicit constant in \eqref{eq:cumm_smooth} depends only
on  $\supp(f_\eps)$ so that it is uniform on $\eps$.
We choose $L=N^q$ as in Section \ref{sec:CLT_siegel_smooth} and 
$\gamma=c_r(\log N)$, where $c_r>0$ will be specified later.
In particular, then
$N^{1-r/2}\gamma^{r-1} L^{(r-(m+n-1))^+}\to 0$, and 
assuming that 
\begin{equation}
\label{eq:cc6}
N^{r/2} L^{r}\eps^{-rk}=o(e^{\delta\gamma}),
\end{equation}
it follows that \eqref{eq:cccc} holds.

\vspace{0.2cm}

To prove \eqref{eq:vvv}, we have to estimate 
\begin{align*}
\left\|F^{(\eps,L)}_N\right\|_{L^2(\cY)}^2 =&\frac{1}{N}\sum_{s_1=M}^{N-1}\sum_{s_2=M}^{N-1} \int_{\cY}
\psi^{(\eps,L)}_{s_1}\psi^{(\eps,L)}_{s_2}\, d\mu_\cY,
\end{align*}
where
$$
\psi^{(\eps,L)}_s(y):=\hat f_\eps^{(L)}(a^sy)-\mu_\cY(\hat f_\eps^{(L)} \circ a^s).
$$
As in \eqref{eq:FF_N_L}, we obtain that
\begin{align*}
\left\|F^{(\eps,L)}_N\right\|_{L^2(\cY)}^2 
=&\, \Theta^{(\eps,L)}_N(0)
+2\sum_{s=1}^{N-M-1}\Theta^{(\eps,L)}_N(s),
\end{align*}
where
$$
\Theta^{(\eps,L)}_N(s):=\frac{1}{N}\sum_{t=M}^{N-1-s} \int_{\cY}\psi^{(\eps,L)}_{s+t}\psi^{(\eps,L)}_t\, d\mu_\cY.
$$
Our estimate proceeds as in Section \ref{sec:CLT_siegel_smooth},
and we shall show that with a suitable choice of parameters,
\begin{align}\label{eq:small}
\left\|F^{(\eps,L)}_N\right\|_{L^2(\cY)}^2 
=& \Theta^{(\eps,L)}_\infty(0)
+2\sum_{s=1}^{K-1}\Theta^{(\eps,L)}_\infty(s)+o(1),
\end{align}
where 
$$
\Theta^{(\eps,L)}_\infty(s):=\int_{\cX} (\hat f_\eps^{(L)}\circ a^s)\hat f_\eps^{(L)}\, d\mu_\cX -\mu_\cX(\hat f_\eps^{(L)})^2.
$$
Indeed, arguing as in \eqref{eq:starr}, we deduce that
\begin{align*}
&\Theta^{(\eps,L)}_N(0)+2\sum_{s=1}^{M-N-1}\Theta^{(\eps,L)}_N(s)\\
=&\,
\Theta^{(\eps,L)}_\infty(0)+2\sum_{s=1}^{K-1}\Theta^{(\eps,L)}_\infty (s)\\
& + O \left (N^{-1}(M+K)K+ N^{-1} e^{-\delta M}  e^{\xi K} L^{2}\eps^{-2k} + e^{-\delta K} L^{2}\eps^{-2k} \right).
\end{align*}
Hence, \eqref{eq:small} holds provided that
\begin{align}
e^{-\delta K} L^{2}\eps^{-2k} \to 0, \label{eq:cc7}\\
N^{-1} e^{-\delta M}  e^{\xi K} L^{2}\eps^{-2k} \to 0, \label{eq:cc8}\\
N^{-1}(M+K)K \to 0. \label{eq:cc9}
\end{align}

Next, we set
$$
\Theta^{(\eps)}_\infty(s):=\int_{\cX} (\hat f_\eps\circ a^s)\hat f_\eps\, d\mu_\cX -\mu_\cX(\hat f_\eps)^2
$$
and observe that as in \eqref{eq:Theta3},
$$
\Theta^{(\eps,L)}_\infty (s)=\Theta^{(\eps)}_\infty (s)+O_\tau\left(L^{-(\tau-1)/2}\right)\quad\hbox{ for all $\tau<m+n-1$,}
$$
uniformly on $\eps$. Hence,
choosing $K$ so that 
\begin{equation}
\label{eq:cc10}
K L^{-(\tau-1)/2} \to 0 \quad \hbox{for some $\tau<m+n-1$},
\end{equation}
we conclude that
$$
\left\|F^{(\eps,L)}_N\right\|_{L^2(\cY)}^2 =
\Theta^{(\eps)}_\infty(0)+2\sum_{s=1}^{K-1}\Theta^{(\eps)}_\infty (s)+o(1).
$$
The terms $\Theta^{(\eps)}(s)$ can be computed using Propositions \ref{p:siegel_mean} and \ref{p:rogers}, and we obtain that
$$
\Theta^{(\eps)}_\infty(s)=\zeta(m+n)^{-1}\sum_{p,q\ge 1} \left(\int_{\mathbb{R}^{m+n}} f_\eps(pa^s\overline{z})f_\eps(q\overline{z})\,d\overline{z}+
\int_{\mathbb{R}^{m+n}} f_\eps(pa^s\overline{z})f_\eps(-q\overline{z})\,d\overline{z}\right).
$$
We also set
\begin{align*}
\Theta_\infty(s):=&\zeta(m+n)^{-1}\sum_{p,q\ge 1} \left(\int_{\mathbb{R}^{m+n}} \chi(pa^s\overline{z})\chi(q\overline{z})\,d\overline{z}+\int_{\mathbb{R}^{m+n}} \chi(pa^s\overline{z})\chi(-q\overline{z})\,d\overline{z}\right)\\
=& 2\zeta(m+n)^{-1}\sum_{p,q\ge 1} \int_{\mathbb{R}^{m+n}} \chi(pa^s\overline{z})\chi(q\overline{z})\,d\overline{z}.
\end{align*}
We claim that 
\begin{equation}
\label{eq:Theta}
|\Theta^{(\eps)}_\infty(s)-\Theta_\infty(s)|\ll \eps^{1/2}e^{-(m+n-2)s/2}.
\end{equation}
This reduces to the estimation of 
\begin{align}
&\left|\int_{\mathbb{R}^{m+n}} f_\eps(pa^s\overline{z})f_\eps(q\overline{z})\,d\overline{z}-
\int_{\mathbb{R}^{m+n}} \chi(pa^s\overline{z})\chi(q\overline{z})\,d\overline{z}\right| \label{eq:sum_int}\\
\le&
\left|\int_{\mathbb{R}^{m+n}} (f_\eps-\chi)(pa^s\overline{z})f_\eps(q\overline{z})\,d\overline{z}\right|+
\left|\int_{\mathbb{R}^{m+n}} \chi(pa^s\overline{z})(\chi-f_\eps)(q\overline{z})\,d\overline{z} \right|.\nonumber
\end{align}
We observe that there exists $0<\upsilon_1<\upsilon_2$ such that 
$$
\int_{\mathbb{R}^{m+n}} (f_\eps-\chi)(pa^s\overline{z})f_\eps(q\overline{z})\,d\overline{z}=0
$$
unless $\upsilon_1\,e^sq\le p\le \upsilon_2\, e^sq$, and by the Cauchy--Schwarz inequality
under these restrictions,
$$
\left|\int_{\mathbb{R}^{m+n}} (f_\eps-\chi)(pa^s\overline{z})f_\eps(q\overline{z})\,d\overline{z}\right|
\le \frac{\|f_\eps-\chi\|_{L^2}}{p^{(m+n)/2}}\frac{\|f_\eps\|_{L^2}}{q^{(m+n)/2}}
\ll \frac{\eps^{1/2}}{q^{m+n}e^{(m+n)s/2}}. 
$$
Since $m+n\ge 3$,
\begin{align*}
\sum_{p,q\ge 1}\left|\int_{\mathbb{R}^{m+n}} (f_\eps-\chi)(pa^s\overline{z})f_\eps(q\overline{z})\,d\overline{z}\right|
&= \sum_{q\ge 1}\sum_{\upsilon_1\,e^sq\le p\le \upsilon_2\, e^sq}
\left|\int_{\mathbb{R}^{m+n}} (f_\eps-\chi)(pa^s\overline{z})f_\eps(q\overline{z})\,d\overline{z}\right|\\
&\ll \sum_{q\ge 1}\frac{\eps^{1/2}e^sq }{q^{m+n}e^{(m+n)s/2}}\ll \eps^{1/2}e^{-(m+n-2)s/2}.
\end{align*}
The sum of the other integral appearing in \eqref{eq:sum_int} is estimated similarly. This proves \eqref{eq:Theta}.

\vspace{0.2cm}

Provided that 
$\eps=\eps(N)\to 0,$
the estimate \eqref{eq:Theta} implies that
$$
\left\|F^{(\eps,L)}_N\right\|_{L^2(\cY)}^2 =
\Theta_\infty(0)+2\sum_{s=1}^{K-1}\Theta_\infty (s)+o(1).
$$
Hence, 
$$
\left\|F^{(\eps,L)}_N\right\|_{L^2(\cY)}^2 \to \sigma^2:=\Theta_\infty(0)+2\sum_{s=1}^{\infty}\Theta_\infty (s)
$$
as $N\to\infty$.

\vspace{0.2cm}

Finally, we compute the limit 
$$
\sigma^2=\sum_{s=-\infty}^{\infty}\Theta_\infty (s)=
2\zeta(m+n)^{-1}\sum_{s=-\infty}^{\infty} \sum_{p,q\ge 1}\int_{\bR^{m+n}} \chi(pa^s\overline{z})\chi(q\overline{z})\, d\overline{z}.
$$
We note that the sum
$$
\Xi:=\sum_{s=-\infty}^{\infty} \chi\circ a^s
$$
is equal to the characteristic function of the set
$$
\left\{(\overline{x},\overline{y})\in \bR^{m+n}:\, \|\overline{y}\|>0, \;\; |x_i|<\vartheta_i\, \|\overline{y}\|^{-w_i},\; i=1,\ldots,m\right\},
$$
and 
\begin{align*}
\int_{\bR^{m+n}} \Xi(p\overline{z})\chi(q\overline{z})\, d\overline{z}=&\int_{1/q\le \|y\|<e/q} \left(\prod_{i=1}^m 2\vartheta_i\max(p,q)^{-1-w_i}\|\overline{y}\|^{-w_i}\right)\,d\overline{y}\\
=&\,2^m\left(\prod_{i=1}^m \vartheta_i\right)\max(p,q)^{-m-n} \int_{1/q\le \|\overline{y}\|<e/q} \|\overline{y}\|^{-n}\,d\overline{y}\\
=&\,2^m\left(\prod_{i=1}^m \vartheta_i\right)\max(p,q)^{-m-n} \int_{S^{n-1}}\|\overline{z}\|^{-n} \left(\int_{q^{-1}\|\overline{z}\|^{-1}}^{eq^{-1}\|\overline{z}\|^{-1}} r^{-1}\, dr\right) \,d\overline{z}\\
=&\,2^m\left(\prod_{i=1}^m \vartheta_i\right)\omega_n\,\max(p,q)^{-m-n},
\end{align*}
where $\omega_n:=\int_{S^{n-1}}\|\overline{z}\|^{-n}\, d\overline{z}$.
We also see that
\begin{align*}
\sum_{p,q\ge 1} \max(p,q)^{-m-n} &=\sum_{p=1}^\infty p^{-m-n}+2\sum_{1\le p<q} q^{-m-n}\\
&=\zeta(m+n)+2\sum_{q\ge 1} \frac{q-1}{q^{m+n}}
=2\zeta(m+n-1)-\zeta(m+n),
\end{align*}
and thus
$$
\sigma^2=\sum_{s=-\infty}^{\infty}\Theta_\infty (s)=
2^{m+1}\left(\prod_{i=1}^m \vartheta_i\right)\omega_n\left(\frac{2\zeta(m+n-1)}{\zeta(m+n)}-1\right).
$$

\subsection{Proof of Theorem \ref{th:CLT_char_Siegel}}

As we already remarked above, it is sufficient to show that the average $F_N^{(\eps,L)}$
converge in distribution to the normal law. According to Proposition \ref{th:CLT},
it is sufficient to check that
$$
\cum^{(r)}_{\mu_\cY}\left(F_N^{(\eps,L)}\right)\to 0\quad\hbox{as $N\to\infty$}
$$
when $r\ge 3$, and 
$$
\left\|F_N^{(\eps,L)}\right\|_{L^2(\cY)}\to\sigma^2\quad\hbox{as $N\to\infty$}.
$$
These properties have been established above provided that the parameters
$$
M=M(N),\;\; \eps=\eps(N),\;\; L=N^q,\;\; \gamma=c_r(\log N),\;\; K=K(N)\le M(N)
$$
satisfy the ten conditions \eqref{eq:cc1}, \eqref{eq:cc2}, \eqref{eq:cc3}, \eqref{eq:cc4}, \eqref{eq:cc5}, \eqref{eq:cc6}, \eqref{eq:cc7},
\eqref{eq:cc8}, \eqref{eq:cc9}, \eqref{eq:cc10}.
It remains to show that such choice of parameters is possible.
The condition \eqref{eq:cc4} is guaranteed by the choice of $L$.
First, we take 
$$
\eps(N)=1/N.
$$
Then the first part of \eqref{eq:cc2} holds. Then we select 
sufficiently large $c_r$ in $\gamma=c_r(\log N)$
so that \eqref{eq:cc6} holds.
After that we choose 
$$
K(N)=c_1(\log N)
$$
with sufficiently large $c_1>0$
so that \eqref{eq:cc7} holds. Then it is clear that \eqref{eq:cc10} also holds.
Given these $\eps$, $L$, $\gamma$, and $K$, we choose
$$
M(N)=(\log N)(\log\log N)
$$
so that the second part of \eqref{eq:cc2}, \eqref{eq:cc3}, \eqref{eq:cc5},
and \eqref{eq:cc8} hold for all $N\ge N_0(r)$. With these choices, it is clear that \eqref{eq:cc1}
and \eqref{eq:cc9} also hold.
Hence, Theorem \ref{th:CLT_char_Siegel} follows from Proposition \ref{th:CLT}.

\subsection{Proof of Theorem \ref{th:CLT_forms}}

For $u\in \hbox{M}_{m,n}([0,1])$, we set
$$
D_T({u}):=\frac{\Delta_T({u})-C_{m,n}\,\log T}{(\log T)^{1/2}},
$$
where $C_{m,n}=2^m\vartheta_1\cdots \vartheta_m\omega_n$ with $\omega_n:=\int_{S^{n-1}} \|\overline{z}\|^{-n}\,d\overline{z}$.
We shall show that $D_T({u})$ can be approximated by the averages $F_N$ defined in \eqref{eq:F_N_0}. This will allow us deduce convergence in distribution 
for $D_T$. We observe that:

\begin{lemma}\label{l:sum_mu}
$$
\sum_{s=0}^{N-1}\int_{\cY} \hat \chi(a^s y)\,d\mu_\cY(y)
= C_{m,n} N+O(1),
$$
where $C_{m,n}$ is defined above.
\end{lemma}

\begin{proof}
We observe that
$$
\sum_{s=0}^{N-1}\int_{\cY} \hat \chi(a^s y)\,d\mu_\cY(y)
=\int_{\cY} \hat \Xi_N( y)\,d\mu_\cY(y),
$$
where $\Xi_N$ denotes the charateristic function of the set
$$
\left\{(\overline{x},\overline{y})\in \bR^{m+n}:\, 1\le \|\overline{y}\|< e^N, \;\; |x_i|<\vartheta_i\, \|\overline{y}\|^{-w_i},\; i=1,\ldots,m\right\}.
$$
Using notation as in the proof of Proposition \ref{prop:sup}, we obtain 
\begin{align*}
\int_{\cY} \hat \Xi_N( y)\,d\mu_\cY(y)&=\sum_{1\le \|\overline{q}\|< e^N} \sum_{\overline{p}\in \bZ^m} \prod_{i=1}^m \int_{[0,1]^m} \chi^{(i)}_{\overline{q}}\left(p_i+\left<\overline{u}_i,\overline{q}\right>\right) d\overline{u}_i\\
&=\sum_{1\le \|\overline{q}\|< e^N}   \prod_{i=1}^m \left(\sum_{p_i\in \mathbb{Z}}\int_{[0,1]^m} \chi^{(i)}_{\overline{q}}\left(p_i+\left<\overline{u}_i,\overline{q}\right>\right)\, d\overline{u}_i\right).
\end{align*}
We claim that
\begin{align}\label{eq:torus}
\sum_{p\in \mathbb{Z}}\int_{[0,1]^m} \chi^{(i)}_{\overline{q}}\left(p+\left<\overline{u},\overline{q}\right>\right)\, d\overline{u}
=2\vartheta_i \|\overline{q}\|^{-w_i}.
\end{align}
To prove this, let us consider more generally a bounded measurable functions $\chi$ on $\bR$ with compact support, the function $\psi(x)=\chi(x_1)$ on $\bR^m$,
and the function $\tilde\psi(x)=\sum_{p\in\bZ} \chi(p+x_1)$ on the torus $\bR^m/\bZ^m$. We suppose without loss of generality that $q_1\ne 0$ and consider 
a non-degenerate linear map 
$$
S:\bR^m\to \bR^m: \overline{u}\mapsto \left(\left<\overline{u},\overline{q}\right>,u_2,\ldots,u_m\right)
$$
which induced a linear epimorphism of the torus $\bR^m/\bZ^m$.
Using that $S$ preserves the the  Lebesgue probability measure $\mu$ on 
$\bR^m/\bZ^m$, we deduce that
\begin{align*}
\sum_{p\in \mathbb{Z}}\int_{[0,1]^m} \chi\left(p+\left<\overline{u},\overline{q}\right>\right)\, d\overline{u}
=\int_{\bR^m/\bZ^m} \tilde\psi(Sx)\, d\mu(x)
=\int_{\bR^m/\bZ^m} \tilde\psi(x)\, d\mu(x)=\int_{\bR}\chi(x_1)\,dx_1,
\end{align*}
which yields \eqref{eq:torus}. \\

In turn, \eqref{eq:torus} implies that
\begin{align*}
\int_{\cY} \hat \Xi_N\,d\mu_\cY=
2^m\left(\prod_{i=1}^m \vartheta_i\right) \sum_{1\le \|\overline{q}\|< e^N} \|\overline{q}\|^{-n}.
\end{align*}
Using that $\|\overline{y}_1\|^{-n}=\|\overline{y}_1\|^{-n}+O\left(\|\overline{y}_1\|^{-n-1}\right)$
when $\|\overline{y}_1-\overline{y}_2\|\ll 1$, we deduce that
\begin{align*}
\sum_{1\le \|\overline{q}\|< e^N} \|\overline{q}\|^{-n}
&=\int_{1\le \|\overline{y}\|< e^N} \|\overline{y}\|^{-n}\, d\overline{y}+O(1),
\end{align*} 
and expressing the integral in polar coordinates, we obtain
\begin{align*}
\int_{1\le \|\overline{y}\|< e^N} \|\overline{y}\|^{-n}\, d\overline{y}
&=\int_{S^{n-1}}\int_{\|\overline{z}\|^{-1}}^{\|\overline{z}\|^{-1} e^N} \|r\overline{z}\|^{-n}\, r^{n-1}drd\overline{z}
=\omega_n N+O(1).
\end{align*} 
This implies the lemma.
\end{proof}

Now we return to the proof of Theorem \ref{th:CLT_forms}.
Since 
$$
\Delta_{e^N}(u)=\sum_{s=0}^{N-1}\hat \chi(a^s\Lambda_u)+O(1),
$$
Lemma \ref{l:sum_mu} implies that
$$
\|D_{e^N}-F_N\|_{C^0}\to 0\quad \hbox{as $N\to\infty$,}
$$
where $(F_N)$ is defined as in \eqref{eq:F_N_0}. Therefore, it follows from Theorem 
\ref{th:CLT_char_Siegel} that for every $\xi\in \bR$,
$$
|\{\xi\in \hbox{M}_{m,n}([0,1]):\, D_{e^N}({u})<\xi \})|\to \hbox{Norm}_\sigma(\xi)
\quad \hbox{ as $N\to\infty$.}
$$
Let us take $N_T=\lfloor \log T\rfloor$. Then 
$$
e^{N_T}\le T<e^{N_T+1}\quad\hbox{and}\quad N_T\le \log T<N_T+1,
$$
so that
$$
D_T\le \frac{\Delta_{e^{N_T+1}}-C_{m,n}\,N_T}{(\log T)^{1/2}}=a_T\, D_{e^{N_T+1}}+ b_T
$$
with $a_T\to 1$ and $b_T\to 0$ as $T\to\infty$.
Hence, we deduce that
$$
|\{{u}\in \hbox{M}_{m,n}([0,1]):\, D_{T}({u})<\xi \}|\ge 
|\{{u}\in \hbox{M}_{m,n}([0,1]):\, D_{e^{N_T+1}}({u})<(\xi-b_T)/a_T \}|.
$$
It follows that for any $\eps>0$ and sufficiently large $T$,
$$
|\{{u}\in \hbox{M}_{m,n}([0,1]):\, D_{T}({u})<\xi \}|\ge
|\{{u}\in \hbox{M}_{m,n}([0,1]):\, D_{e^{N_T+1}}({u})<\xi-\eps \}|.
$$
Therefore,
$$
\liminf_{T\to\infty}|\{{u}\in \hbox{M}_{m,n}([0,1]):\, D_{T}({u})<\xi \}|\ge \hbox{Norm}_\sigma(\xi-\eps)
$$
for all $\eps>0$. This implies that
$$
\liminf_{T\to\infty}|\{{u}\in \hbox{M}([0,1]):\, D_{T}({u})<\xi \}|\ge \hbox{Norm}_\sigma(\xi).
$$
A similar argument also implies the upper bound
$$
\limsup_{T\to\infty}|\{{u}\in \hbox{M}_{m,n}([0,1]):\, D_{T}({u})<\xi \}|\le \hbox{Norm}_\sigma(\xi).
$$
This completes the proof of Theorem \ref{th:CLT_forms}.

\end{document}